\documentclass{article}

\PassOptionsToPackage{numbers, compress}{natbib}
\usepackage[preprint]{neurips_2022}

\usepackage[utf8]{inputenc} % allow utf-8 input
\usepackage[T1]{fontenc}    % use 8-bit T1 fontsd
\usepackage{hyperref}       % hyperlinks
\usepackage{url}            % simple URL typesetting
\usepackage{booktabs}       % professional-quality tables
\usepackage{amsfonts}       % blackboard math symbols
\usepackage{nicefrac}       % compact symbols for 1/2, etc.
\usepackage{microtype}      % microtypography
\usepackage{xcolor}         % colors
\usepackage{amsthm, amssymb, amsfonts, latexsym, mathtools}
\usepackage[ruled,vlined]{algorithm2e}
\usepackage{algorithmic}
\usepackage{comment}
\usepackage{here}
\usepackage{subfigure}
\usepackage{wrapfig}
\usepackage{amsmath, amssymb}
\usepackage{amsthm}
\usepackage{breqn}
\usepackage[shortlabels]{enumitem}
\usepackage{siunitx}
\usepackage{listings}
\usepackage{subfigure}
\usepackage{comment}
\usepackage[ruled,vlined]{algorithm2e}
\usepackage{algorithmic}
\usepackage{here}
\usepackage{mathtools}

\newtheorem{theorem}{Theorem}
\newtheorem{lemma}{Lemma}

\newtheorem{remark}{Remark}
\newtheorem{definition}{Definition}
\newtheorem{assumption}{Assumption}
\newtheorem{example}{Example}

\title{Theoretical Analysis of Primal-Dual Algorithm for Non-Convex Stochastic Decentralized Optimization}

\author{%
  Yuki Takezawa \\
  Kyoto University and RIKEN AIP\\
  \texttt{yuki-takezawa@ml.ist.i.kyoto-u.ac.jp} \\
  \And
  Kenta Niwa \\
  NTT Communication Science Laboratories \\
  \texttt{kenta.niwa.bk@hco.ntt.co.jp} \\
  \And
  Makoto Yamada \\
  Kyoto University and RIKEN AIP \\
  \texttt{myamada@i.kyoto-u.ac.jp}
}

\begin{document}
\maketitle

\begin{abstract}
In recent years, decentralized learning has emerged as a powerful tool not only for large-scale machine learning,
but also for preserving privacy.
One of the key challenges in decentralized learning is that the data distribution held by each node is statistically heterogeneous.
To address this challenge, the primal-dual algorithm called the Edge-Consensus Learning (ECL) was proposed
and was experimentally shown to be robust to the heterogeneity of data distributions. 
However, the convergence rate of the ECL is provided only when the objective function is convex, 
and has not been shown in a standard machine learning setting where the objective function is non-convex.
Furthermore, the intuitive reason why the ECL is robust to the heterogeneity of data distributions has not been investigated.
In this work, we first investigate the relationship between the ECL and Gossip algorithm
and show that the update formulas of the ECL can be regarded as correcting the local stochastic gradient in the Gossip algorithm.
Then, we propose the Generalized ECL (G-ECL), which contains the ECL as a special case,
and provide the convergence rates of the G-ECL in both (strongly) convex and non-convex settings,
which do not depend on the heterogeneity of data distributions.
Through synthetic experiments, we demonstrate that the numerical results of both the G-ECL and ECL coincide with the convergence rate of the G-ECL.
\end{abstract}

\section{Introduction}
Neural networks have achieved promising results in many tasks such as natural language processing \citep{devlin2019bert,brown2020language} and image processing \citep{sarlin2020superglue,liu2021swin}.
To train a large-scale neural network efficiently, decentralized learning is a powerful tool.
Decentralized learning allocates data into multiple nodes (e.g., servers) and trains a neural network in parallel.
Because decentralized learning allows us to train a neural network without aggregating all the data in one server,
it has also attracted considerable attention from the perspective of privacy preservation.

% Gossipの理論解析が発展している。そしてヘテロに弱い
One of the most widely used algorithms for decentralized learning is the decentralized parallel SGD (D-PSGD) \citep{lian2017can} (a.k.a. the Gossip algorithm).
Recently, the effect of various variables on the convergence rate of the Gossip algorithm has been well investigated. 
(for example, noise of stochastic gradient, 
the structure of the network, and the heterogeneity of data distributions held by each node)
\citep{lian2017can,lian2018asynchronous,koloskova2020unified,kong2021consensus}.
%Recently, it has been well studied how the convergence rate of the Gossip algorithm is affected by the various variables 
%(e.g., the noise of the stochastic gradient, 
%the structure of network, and the heterogeneity of data distributions held by each node) \citep{lian2017can,lian2018asynchronous,koloskova2020unified,kong2021consensus}.
These theoretical results show that the convergence rate of the Gossip algorithm slows down when the data distribution held by each node is statistically heterogeneous.

% ECLは実験的にヘテロに強いが、理論解析がない
To address the heterogeneity of data distributions,
the primal-dual algorithm using Douglas-Rachford splitting \citep{douglas1955numerical}
called the Edge-Consensus Learning (ECL) \citep{niwa2020edge} has been proposed.
In image classification tasks, it was shown that 
the ECL outperforms the Gossip algorithm when the data distributions held by each node are statistically heterogeneous,
and the ECL has been experimentally shown to be robust to the heterogeneity of data distributions.
Recently, \citet{rajawat2020primal} provided the convergence rate of the ECL\footnote{The algorithm called the decentralized primal-dual algorithm in \citep{rajawat2020primal} is equivalent to the ECL when $\theta=1$.} when the objective function is convex
and showed that it does not depend on the heterogeneity of data distributions.
However, in a standard machine learning setting where the objective function is non-convex,
the convergence rate of the ECL has not been shown.\footnote{The previous work \cite{niwa2021asynchronous} attempted to analyze the convergence rates of the ECL on both (strongly) convex and non-convex settings. However, strong approximations were used in the proofs, which do not hold in practice. This is discussed in detail in Sec. \ref{sec:issues_of_existing_proof}.}
Furthermore, the relationship between the ECL and Gossip algorithm has not been investigated,
including the differences between the ECL and Gossip algorithm
and an intuitive reason why the ECL is robust to the heterogeneity of data distributions.

In this work, we propose the \textbf{Generalized ECL (G-ECL)}, which contains the ECL as a special case,
and provide the convergence rates of the G-ECL in both (strongly) convex and non-convex settings.
More specifically, we investigate the relationship between the ECL and Gossip algorithm
and show that the update formulas of the ECL can be regarded as correcting the stochastic gradient of each node in the Gossip algorithm.
Then, to make the convergence analysis tractable,
we increase the degrees of freedom of the hyperparameters of the ECL 
and propose the G-ECL, which contains the ECL as a special case.
By using the proof techniques of the Gossip algorithm \citep{koloskova2020unified},
we provide the convergence rates of the G-ECL in (strongly) convex and non-convex settings
and show that they do not depend on the heterogeneity of data distributions.
Table \ref{table:convergence_rate} summarizes the convergence rates of the Gossip algorithm, ECL and G-ECL.
Through synthetic experiments,
we demonstrate that the numerical results of both the G-ECL and ECL are consistent with the convergence rate of the G-ECL.

Our contributions are summarized as follows:
\begin{itemize}
    \item In this work, we investigate the relationship between the Gossip algorithm and ECL and show that the update formulas of the ECL can be regarded as correcting the stochastic gradient of each node in the Gossip algorithm (Sec. \ref{sec:reformulation}).
    \item We propose the G-ECL, which contains the ECL as a special case, and provide the convergence rates of the G-ECL in (strongly) convex and non-convex settings (Sec. \ref{sec:gecl} and \ref{sec:convergence_results}). Then, we show that the convergence rates of the G-ECL do not depend on the heterogeneity of data distributions.
    \item Through synthetic experiments, we demonstrate that the numerical results of both the ECL and G-ECL coincide with the convergence rates of the G-ECL (Sec. \ref{sec:experimental_results}).
\end{itemize}

\textbf{Notation:}
In this work, we denote $[n] = \{1,2,\cdots,n\}$ for any $n\in \mathbb{N}$.
We write $\mathbf{I}$ for the identity matrix,
$\mathbf{1}$ for the vector with all ones,
$\| \cdot \|$ for L2 norm,
and $\| \cdot \|_F$ for the Frobenius norm.

\begin{table*}[t]
\begin{minipage}{\textwidth}
\caption{Convergence rates of the D-PSGD, ECL, and G-ECL. We define $\tilde{\zeta}^2 \coloneqq \frac{1}{p} \zeta^2$, and the definition of other parameters is shown in Sec. \ref{sec:preliminary}, \ref{sec:setup} and \ref{sec:convergence_results}. Note that the D-PSGD requires the assumption about the heterogeneity of data distributions, while the G-ECL and ECL do not require.}
\label{table:convergence_rate}
\center
\resizebox{\linewidth}{!}{
\begin{tabular}{cc}
\begin{tabular}{lc}
\toprule
                                     & Non-Convex \\
\midrule
D-PSGD \citep{koloskova2020unified} 
        & $\mathcal{O} \left( \!\! \sqrt{\frac{r_0 \sigma^2 L}{n R}} + \left( \frac{ r_0^2 L^2 (\sigma^2 + \tilde{\zeta}^2)  (1 - p)}{p R^2} \right)^{\frac{1}{3}} \!\!\! + \frac{s r_0}{R} \!\! \right)$ \\
ECL \citep{rajawat2020primal} 
        & N/A \\
G-ECL (this work) 
        & $\mathcal{O} \left( \!\! \sqrt{\frac{r_0 \sigma^2 L}{n R}} + \left( \frac{ r_0^2 L^2 \tilde{\sigma}^2}{p R^2} \right)^{\frac{1}{3}} \!\!\! + \frac{s r_0}{R} \!\! \right)$ \\
\bottomrule
\end{tabular}

& \\
& \\

\begin{tabular}{lccc}
\toprule
                                     & Strongly Convex & General Convex\\
\midrule
D-PSGD  \citep{koloskova2020unified} 
        & $\tilde{\mathcal{O}} \left( r_0 s \exp \left[ \frac{-\mu(R+1)}{s} \right] + \frac{\sigma^2}{\mu n R} + \frac{L(\sigma^2 + \tilde{\zeta}^2) (1 - p)}{p \mu^2 R^2}\right)$ 
        & $\mathcal{O} \left( \!\! \sqrt{\frac{r_0 \sigma^2}{n R}} + \left( \frac{r_0^2  L (\sigma^2 + \tilde{\zeta}^2) (1 - p)}{ p R^2 } \right)^{\frac{1}{3}} \!\!\! + \frac{s r_0}{R} \!\! \right)$ \\
ECL \citep{rajawat2020primal}\footnote{\citet{rajawat2020primal} evaluated $\frac{1}{n} \sum_{i=1}^n \mathbb{E} [D_{f_i}(\mathbf{x}_i^{(r)}, \mathbf{x}^\star)]$ where $D_{f_i}$ is Bregman divergence associated with $f_i$. Note that if $f_i$ is strictly convex, $D_{f_i}(\mathbf{x}_i^{(r)}, \mathbf{x}^\star)=0$ is equivalent to $\mathbf{x}_i^{(r)}=\mathbf{x}^\star$.} 
        & N/A 
        & $\mathcal{O}\left( \sqrt{\frac{\sigma^2}{R}} + \frac{|\mathcal{E}|}{n R} \right)$  \\
G-ECL (this work) 
        & $\tilde{\mathcal{O}} \left( r_0 s \exp \left[ \frac{-\mu (R+1)}{s} \right] + \frac{\sigma^2}{\mu n R} + \frac{L \tilde{\sigma}^2}{p \mu^2 R^2} \right)$  
        & $\mathcal{O} \left( \!\! \sqrt{\frac{r_0 \sigma^2}{n R}} + \left( \frac{r_0^2  L \tilde{\sigma}^2 }{ p R^2 } \right)^{\frac{1}{3}} \!\!\! + \frac{s r_0}{R} \!\! \right)$ \\
\bottomrule
\end{tabular}
\end{tabular}}
\end{minipage}
\vskip -0.15 in
\end{table*}

\section{Preliminary}
\label{sec:preliminary}
In this section, we briefly introduce the problem setting of decentralized learning, Gossip algorithm, and ECL.
A more detailed discussion of related works is presented in Sec. \ref{sec:related_work}. 

\subsection{Problem Setting}
\label{sec:problem_setting}
We introduce a problem setting for decentralized learning.
Let $G=(\mathcal{V},\mathcal{E})$ be an undirected graph representing the network topology of nodes
where $\mathcal{V}$ denotes a set of nodes and $\mathcal{E}$ denotes a set of edges.
In the following, we denote $\mathcal{V}$ as a set of integers $\{ 1, 2, \ldots, |\mathcal{V}|\}$ for simplicity.
We denote the set of neighbors of node $i$ as $\mathcal{N}_i \coloneqq \{ j \in \mathcal{V} | (i,j) \in \mathcal{E} \}$
and we denote $\mathcal{N}_i^+ \coloneqq \mathcal{N}_i \cup \{ i \}$.
In decentralized learning, nodes $i$ and $j$ are allowed to communicate the parameters only if $(i,j)\in\mathcal{E}$.
The decentralized learning problem is formulated as follows:
\begin{align}
    \label{eq:decentralized_learning}
    \min_{\mathbf{x} \in \mathbb{R}^d} \left[ f(\mathbf{x}) \coloneqq \frac{1}{n} \sum_{i=1}^{n} f_i (\mathbf{x}) \right],
    \qquad 
    f_i(\mathbf{x}) \coloneqq \mathbb{E}_{\xi_i \sim \mathcal{D}_i} \left[ F_i(\mathbf{x} ; \xi_i)\right],
\end{align}
where $\mathbf{x}\in\mathbb{R}^d$ denotes the model parameter, 
$f_i : \mathbb{R}^d \rightarrow \mathbb{R}$ denotes the objective function of node $i$,
$n \coloneqq |\mathcal{V}|$ is the number of nodes, 
and $\mathcal{D}_i$ is the data distribution held by node $i$.
In this work,
we assume that only the stochastic gradient $\nabla F_i(\mathbf{x} ; \xi_i)$ is accessible and the full gradient $\nabla f_i (\mathbf{x})$ is inaccessible
and analyze the convergence rate in both cases when $f_i$ is (strongly) convex and when $f_i$ is non-convex,

\subsection{Gossip Algorithm}
One of the most popular algorithms for decentralized learning is the D-PSGD \citep{lian2017can} (a.k.a. the Gossip algorithm).
In the Gossip algorithm, each node updates its model parameters as follows:
\begin{align}
    \label{eq:gossip}
    \mathbf{x}_i^{(r+1)} &= \sum_{j \in \mathcal{N}_i^+} W_{ij} \left( \mathbf{x}_j^{(r)} - \eta \nabla F_j(\mathbf{x}_j^{(r)} ; \xi_j^{(r)}) \right),
\end{align}
where $\mathbf{x}_i \in \mathbb{R}^d$ is the model parameter of node $i$,
$\eta>0$ is the step size, and $W_{ij}\in [0,1]$ is the weight of the edge $(i,j)\in\mathcal{E}$.
Let $\mathbf{W}$ be an $n\times n$ matrix whose $(i,j)$-element is $W_{ij}$ if $j \in \mathcal{N}_i^+$ and $0$ otherwise.
In the Gossip algorithm, $\mathbf{W}$ is assumed to be a mixing matrix defined as follows.

\begin{definition}[Mixing Matrix]
If $\mathbf{W}\in [0,1]^{n\times n}$ is symmetric ($\mathbf{W}=\mathbf{W}^\top$) and doubly stochastic ($\mathbf{W}\mathbf{1}=\mathbf{1}, \mathbf{1}^\top \mathbf{W} = \mathbf{1}^\top$),
then $\mathbf{W}$ is called a mixing matrix.
\end{definition}

\subsection{Edge-Consensus Learning}
In this section, we briefly introduce the primal-dual algorithm using Douglas-Rachford splitting
called the Edge-Consensus Learning (ECL) \cite{niwa2020edge}.
Reformulating Eq. \eqref{eq:decentralized_learning},
we can define the primal problem as follows:
\begin{align}
\label{eq:primal}
    \min_{\mathbf{x}_1,\ldots,\mathbf{x}_n \in \mathbb{R}^d} \frac{1}{n} \sum_{i=1}^{n} f_i ( \mathbf{x}_i ) 
    \;\; \text{s.t.} \;\; \mathbf{A}_{i|j} \mathbf{x}_i + \mathbf{A}_{j|i} \mathbf{x}_j =  \mathbf{0}, \; (\forall (i,j)\in\mathcal{E}),
\end{align}
where $\mathbf{A}_{i|j}=\mathbf{I}$ when $i > j$ and $\mathbf{A}_{i|j}=-\mathbf{I}$ when $i<j$ for any $i,j\in[n]$.
%where for any $(i,j) \in \mathcal{E}$, $\mathbf{A}_{i|j}=\mathbf{I}$ when $i>j$, and $\mathbf{A}_{i|j}=-\mathbf{I}$ when $i<j$.
Then, by solving the dual problem of Eq. \eqref{eq:primal} using Douglas-Rachford splitting \citep{douglas1955numerical},
the update formulas can be derived as follows \citep{sherson2019derivation,niwa2020edge}:
\begin{align}
    \label{eq:deterministic_ecl_1}
    \mathbf{x}_i^{(r+1)} &= \text{argmin}_{\mathbf{x}_i} \{ f_i (\mathbf{x}_i)
    + \sum_{j\in\mathcal{N}_i} \frac{\alpha_{i|j}}{2} \| \mathbf{A}_{i|j} \mathbf{x}_i - \mathbf{z}_{i|j}^{(r)} \|^2 \}, \\
    \label{eq:ecl_2}
    \mathbf{y}_{i|j}^{(r+1)} &= \mathbf{z}_{i|j}^{(r)} - 2 \mathbf{A}_{i|j} \mathbf{x}_i^{(r+1)}, \\
    \label{eq:ecl_3}
    \mathbf{z}_{i|j}^{(r+1)} &= (1-\theta) \mathbf{z}_{i|j}^{(r)} + \theta \mathbf{y}_{j|i}^{(r+1)},
\end{align}
where $\mathbf{y}_{i|j}\in\mathbb{R}^d$ and $\mathbf{z}_{i|j}\in\mathbb{R}^d$ are dual variables,
and $\theta \in (0, 1]$ and $\alpha_{i|j} \geq 0$ are hyperparameters.
In previous works \citep{niwa2020edge,niwa2021asynchronous},
hyperparameter $\{ \alpha_{i|j} \}_{ij}$ is set such that $\alpha_{i|j}=\alpha_i$ for all $(i,j) \in \mathcal{E}$.
However, in this work, we increase the degrees of freedom of the hyperparameter $\{ \alpha_{i|j} \}_{ij}$,
which plays an important role in discussing the relationship between the ECL and Gossip algorithm in Theorem \ref{theorem:mixing_matrix}.
Douglas-Rachford splitting has been well studied in the convex optimization literature
and converges linearly to the optimal solution \citep{ryu2015primer,bauschke2011convex,giselsson2017linear}.
Therefore, $\{ \mathbf{x}_i^{(r)} \}_i$ generated by Eqs. (\ref{eq:deterministic_ecl_1}-\ref{eq:ecl_3}) converges linearly to the optimal solution when $f_i$ is convex.

However, when $f_i$ is non-convex (e.g., a loss function of a neural network), Eq. \eqref{eq:deterministic_ecl_1} can not be solved in general.
Subsequently, \citet{niwa2020edge} proposed solving Eq. \eqref{eq:deterministic_ecl_1} approximately as follows:
\begin{align}
    \label{eq:stochastic_ecl}
    \mathbf{x}_i^{(r+1)} &= \text{argmin}_{\mathbf{x}_i} \{ \langle \mathbf{x}_i, \nabla F_i (\mathbf{x}^{(r)}_i ; \xi^{(r)}_i) \rangle
    + \frac{1}{2 \eta} \| \mathbf{x}_i - \mathbf{x}_i^{(r)} \|^2 
    + \sum_{j\in\mathcal{N}_i} \frac{\alpha_{i|j}}{2} \| \mathbf{A}_{i|j} \mathbf{x}_i - \mathbf{z}_{i|j}^{(r)} \|^2 \},
\end{align}
where $\eta>0$ corresponds to the step size.
The update formulas Eqs. (\ref{eq:ecl_2}-\ref{eq:stochastic_ecl}) are called the Edge-Consensus Learning (ECL).
The pseudo-code of the ECL is presented in Alg. \ref{alg:stochastic_ecl}
where $\textbf{Transmit}_{i\rightarrow j}(\cdot)$ denotes the operator that transmits parameters from node $i$ to node $j$
and $\textbf{Receive}_{i\leftarrow j}(\cdot)$ denotes the operator for node $i$ to receive parameters from node $j$.
Then, \citet{niwa2020edge,niwa2021asynchronous} experimentally showed that 
the ECL is robust to the heterogeneity of data distributions.
Recently, \citet{rajawat2020primal} provided the convergence rate of the ECL when $f_i$ is convex
and showed that it does not depend on the heterogeneity of data distributions.
However, when $f_i$ is non-convex, the convergence rate of the ECL has not been provided.
Furthermore, the relationship between the Gossip algorithm and ECL has not yet been investigated,
including the differences between the ECL and Gossip algorithm
and an intuitive reason why the ECL is robust to the heterogeneity of data distributions.

\begin{algorithm}[tb]
   \caption{Update procedure at node $i$ in the ECL.}
   \label{alg:stochastic_ecl}
\begin{algorithmic}[1]
   \STATE {\bfseries Input:}  Set $\{ \alpha_{i|j} \}_{ij}$ such that $\alpha_{i|j}=\alpha_{j|i} \geq 0$ for all $(i,j) \in \mathcal{E}$. Initialize $\mathbf{z}_{i|j}^{(0)}$ to zero and $\mathbf{x}_i^{(0)}$ with the same parameters for all $i \in [n]$.
   \FOR{$r = 0, 1, \ldots, R$}
   \STATE Sample $\xi_i^{(r)}$ and compute $\mathbf{g}_i^{(r)} \coloneqq \nabla F_i(\mathbf{x}_i^{(r)} ; \xi_i^{(r)})$.
   \STATE $\mathbf{x}_i^{(r+1)} \leftarrow (1 + \eta \sum_{j\in\mathcal{N}_i} \alpha_{i|j} )^{-1} \{ \mathbf{x}_i^{(r)} - \eta (\mathbf{g}_i^{(r)} - \sum_{j\in\mathcal{N}_i} \alpha_{i|j} \mathbf{A}_{i|j} \mathbf{z}_{i|j}^{(r)}) \}$.
   \FOR{$j \in \mathcal{N}_i$}
   %\STATE $\textbf{Transmit}_{i\rightarrow j}(\mathbf{y}_{i|j}^{(r+1)})$ and $\textbf{Receive}_{i\leftarrow j}(\mathbf{y}_{j|i}^{(r+1)})$.
   \STATE $\mathbf{y}_{i|j}^{(r+1)} \leftarrow \mathbf{z}_{i|j}^{(r)} - 2  \mathbf{A}_{i|j} \mathbf{x}_i^{(r+1)}$.
   %\STATE $\;\textbf{Transmit}_{i\rightarrow j}(\mathbf{y}_{i|j}^{(r+1)})$ and $\textbf{Receive}_{i\leftarrow j}(\mathbf{y}_{j|i}^{(r+1)})$.
   \STATE $\textbf{Transmit}_{i\rightarrow j}(\mathbf{y}_{i|j}^{(r+1)})$.
   \STATE $\textbf{Receive}_{i\leftarrow j}(\mathbf{y}_{j|i}^{(r+1)})$.
   \STATE $\mathbf{z}_{i|j}^{(r+1)} \leftarrow (1-\theta) \mathbf{z}_{i|j}^{(r)} + \theta \mathbf{y}_{j|i}^{(r+1)}$.
   \ENDFOR
   \ENDFOR
\end{algorithmic}
\end{algorithm}

\section{Relationship between ECL and Gossip Algorithm}
\label{sec:reformulation}
It can be observed that the ECL is different from the Gossip algorithm.
However, it is difficult to discuss what is different between the ECL and Gossip algorithm using Eq. \eqref{eq:gossip} and Eqs. (\ref{eq:ecl_2}-\ref{eq:stochastic_ecl}).
In this section, we discuss the relationship between the ECL and Gossip algorithm.
All proofs are provided in Appendix.

\textbf{Organization:}
The remainder of this section is organized as follows.
In Sec. \ref{sec:ecl_as_gossip:reformulation}, we show that
each node implicitly computes the weighted sum with its neighbors in the ECL as well as in the Gossip algorithm.
In Sec. \ref{sec:ecl_as_gossip:assumption_of_hyperparameter},
we show that these weights in the ECL become a mixing matrix when the hyperparameter is set appropriately
as well as in the Gossip algorithm.
In Sec. \ref{sec:ecl_as_gossip:initial_value}, we discuss the property of the sequence of the average $\frac{1}{n} \sum_{i=1}^n \mathbf{x}_i^{(r)}$ in the ECL.

\subsection{Reformulation}
\label{sec:ecl_as_gossip:reformulation}
To discuss the relationship between the ECL and Gossip algorithm,
we reformulate the update formulas of the ECL as follows.
\begin{theorem}
\label{theorem:reformulation2}
Suppose that the hyperparameter $\theta$ is $\frac{1}{2}$,
the dual variable $\mathbf{z}_{i|j}^{(0)}$ is initialized to $\mathbf{A}_{i|j} \mathbf{x}_j^{(0)}$,
and the hyperparameter $\{ \alpha_{i|j} \}_{ij}$ is set such that $\alpha_{i|j} = \alpha_{j|i} \geq 0$
for all $(i,j) \in \mathcal{E}$.
Then, the update formulas Eq. (\ref{eq:stochastic_ecl}) and Eqs. (\ref{eq:ecl_2}-\ref{eq:ecl_3}) are equivalent to the following:
\begin{align}
    \label{eq:reformulated_2_stochastic_ecl_1}
    \tilde{\mathbf{x}}_i^{(r)} &= \sum_{j\in\mathcal{N}_i^+} W_{ij} \mathbf{x}_j^{(r)}, \\
    \label{eq:reformulated_2_stochastic_ecl_2}
    \mathbf{x}_i^{(r+1)} &=  \tilde{\mathbf{x}}_i^{(r)} - \eta^\prime_i \left( \nabla F_i(\mathbf{x}_i^{(r)} ; \xi_i^{(r)}) - \mathbf{c}_i^{(r)} \right), \\
    \mathbf{c}_i^{(r+1)} 
    \label{eq:reformulated_2_stochastic_ecl_3}
    &= \sum_{j \in \mathcal{N}_i^+} W_{ij} \left( \mathbf{c}_j^{(r)} - \nabla F_j(\mathbf{x}_j^{(r)} ; \xi_j^{(r)}) \right)
    + \nabla F_i(\mathbf{x}_i^{(r)} ; \xi_i^{(r)})
    + \sum_{j \in \mathcal{N}_i} \frac{\alpha_{i|j}}{2} ( \tilde{\mathbf{x}}_j^{(r)} - \tilde{\mathbf{x}}_i^{(r)} ),
\end{align}
where $\mathbf{c}_i^{(0)} \coloneqq \frac{1}{2} \sum_{j\in\mathcal{N}_i} \alpha_{i|j} (\mathbf{x}_j^{(0)} - \mathbf{x}_i^{(0)})$, and $W_{ij}$ and $\eta^\prime_i$ are defined as follows:
\begin{gather}
    \label{eq:definition_of_eta_prime}
    \eta^\prime_i \coloneqq \frac{\eta}{1 + \eta \sum_{j\in\mathcal{N}_i} \alpha_{i|j}}, 
    \quad 
    W_{ij} \coloneqq 
    \begin{dcases}
    \frac{2 + \eta \sum_{k\in\mathcal{N}_i} \alpha_{i|k}}{2 ( 1 + \eta \sum_{k\in\mathcal{N}_i} \alpha_{i|k})} & \text{if} \;\;  i=j \\
    \frac{\eta \alpha_{i|j}}{2 ( 1 + \eta \sum_{k\in\mathcal{N}_i} \alpha_{i|k})} & \text{if} \;\;  (i,j) \in \mathcal{E} \\
    0 & \text{otherwise}
    \end{dcases}.
\end{gather}
\end{theorem}
At first glance, 
the update formulas of the ECL Eqs. (\ref{eq:ecl_2}-\ref{eq:stochastic_ecl}) do not explicitly compute the weighted average, in contrast to that of the Gossip algorithm.
However, Theorem \ref{theorem:reformulation2} shows that as well as the update formulas of the Gossip algorithm Eq. \eqref{eq:gossip}, 
the update formula Eq. \eqref{eq:reformulated_2_stochastic_ecl_1} computes the weighted sum whose weights are determined by $\eta$ and $\{ \alpha_{i|j} \}_{ij}$.
Subsequently, the update formula Eq. \eqref{eq:reformulated_2_stochastic_ecl_2} can be regarded as the one where $\mathbf{c}_i$ modifies the local stochastic gradient $\nabla F_i(\mathbf{x}_i ; \xi_i)$ in the update formulas of the Gossip algorithm Eq. \eqref{eq:gossip}.\footnote{The Gossip algorithm and ECL also differ in the order of calculation of weighted average and the stochastic gradient descent. We discuss the effect of this difference in Sec. \ref{sec:discussion}.}
To investigate how the term $\mathbf{c}_i$ modifies the local stochastic gradient $\nabla F_i (\mathbf{x}_i ; \xi_i)$,
we discuss the relationship between the ECL and gradient tracking methods \cite{lorenzo2016next,nedi2017achieving,koloskova2021improved}
in Sec. \ref{sec:gt_and_ecl}.

\subsection{Assumption of Hyperparameters}
\label{sec:ecl_as_gossip:assumption_of_hyperparameter}
Let $\mathbf{W}$ be an $n\times n$ matrix whose $(i,j)$-element is $W_{ij}$.
In general, $\mathbf{W}$ defined by Eq. \eqref{eq:definition_of_eta_prime} is not a mixing matrix
because $\mathbf{W}$ is not symmetric.
In this section,
to further discuss the relationship between the ECL and Gossip algorithm,
we discuss the conditions of hyperparameters for $\mathbf{W}$ to be a mixing matrix.
The following assumption and theorem show that if we set the hyperparameter $\{ \alpha_{i|j} \}_{ij}$ appropriately,
$\mathbf{W}$ is a mixing matrix.

\begin{assumption}
\label{assumption:definition_of_alpha}
The hyperparameter $\{ \alpha_{i|j} \}_{ij}$ is set such that
$\alpha_{i|j}=\alpha_{j|i} \geq 0$ for all $(i,j) \in \mathcal{E}$, 
and there exists $\alpha>0$ that satisfies $\sum_{k\in\mathcal{N}_i} \alpha_{i|k}=\alpha$ for all $i \in [n]$.
\end{assumption}
\begin{theorem}
\label{theorem:mixing_matrix}
Suppose that Assumption \ref{assumption:definition_of_alpha} holds,
then $\mathbf{W}$ defined by Eq. \eqref{eq:definition_of_eta_prime} is a mixing matrix.
\end{theorem}
Moreover, as a by-product, when Assumption \ref{assumption:definition_of_alpha} holds,
$\eta^\prime_i$ defined by Eq. \eqref{eq:definition_of_eta_prime} are same for all $i \in [n]$.
\begin{remark}
Suppose that the hyperparameter $\{ \alpha_{i|j} \}_{ij}$ is set such that Assumption \ref{assumption:definition_of_alpha} holds,
there exists $\eta^\prime>0$ that satisfies for all $i\in [n]$,
\begin{align}
    \eta^\prime 
    = \eta^\prime_i
    = \frac{\eta}{1 + \eta \alpha}.
\end{align}
\end{remark}
%Then, we show an example of the hyperparameter setting for $\{ \alpha_{i|j} \}_{ij}$ that satisfies Assumption \ref{assumption:definition_of_alpha}.
Then, when $G$ is a regular graph, we can set the hyperparameter $\{ \alpha_{i|j} \}_{ij}$ that satisfies Assumption \ref{assumption:definition_of_alpha} as follows.
%The following example indicates that Assumption \ref{assumption:definition_of_alpha} is not a strong assumption.
\begin{example}
\label{example:alpha}
Suppose that $G$ is a $k$-regular graph with $k>0$.
If we set $\alpha_{i|j} = \frac{\alpha}{k}$ for all $(i, j) \in \mathcal{E}$,
then the hyperparameter $\{ \alpha_{i|j} \}_{ij}$ satisfies Assumption \ref{assumption:definition_of_alpha},
and $\mathbf{W}$ is defined as follows:
\begin{align*}
    W_{ij} \coloneqq 
    \begin{dcases}
    \frac{2 + \eta \alpha}{2 ( 1 + \eta \alpha) } & \text{if} \;\; i=j \\
    \frac{\eta \alpha}{2 k ( 1 + \eta \alpha) } & \text{if} \;\;  (i,j) \in \mathcal{E} \\
    0 & \text{otherwise}
    \end{dcases}.
\end{align*}
\end{example}

\subsection{Property of Average Sequence}
\label{sec:ecl_as_gossip:initial_value}

In Sec. \ref{sec:ecl_as_gossip:assumption_of_hyperparameter}, we show that if the hyperparameter $\{ \alpha_{i|j} \}_{ij}$ is set appropriately,
$\mathbf{W}$ defined by Eq. \eqref{eq:definition_of_eta_prime} is a mixing matrix as well as in the Gossip algorithm.
In this section, 
we discuss the relationship between the Gossip algorithm and ECL from the property of the sequence of the average $ \frac{1}{n} \sum_{i=1}^n \mathbf{x}_i^{(r)}$.

In the Gossip algorithm, when $\mathbf{W}$ is a mixing matrix, 
the average $\bar{\mathbf{x}}^{(r)}\coloneqq \frac{1}{n} \sum_{i=1}^{n} \mathbf{x}_i^{(r)}$ generated by Eq. \eqref{eq:gossip} satisfies the following \cite{koloskova2020unified}:
\begin{align}
\label{eq:average}
    \bar{\mathbf{x}}^{(r+1)} &= \bar{\mathbf{x}}^{(r)} - \frac{\eta}{n} \sum_{i=1}^{n} \nabla F_i(\mathbf{x}_i^{(r)} ; \xi_i^{(r)}).
\end{align}
That is, the update formula of the Gossip algorithm is almost equivalent to that of the SGD,
which plays an important role in the convergence analysis of the Gossip algorithm \citep{koloskova2020unified}.
Similarly, 
the property of Eq. \eqref{eq:average} is satisfied in the ECL,
as the following lemma indicates.

\begin{lemma}[Average Sequence]
\label{lemma:c_is_zero}
Suppose that the hyperparameter $\{ \alpha_{i|j} \}_{ij}$ is set such that Assumption \ref{assumption:definition_of_alpha} holds.
Then, under the same assumptions as those in Theorem \ref{theorem:reformulation2},
it holds that $\sum_{i=1}^n \mathbf{c}_i^{(r)}=\mathbf{0}$ for any round $r$,
and the average $\bar{\mathbf{x}}^{(r)} \coloneqq \frac{1}{n} \sum_{i=1}^{n} \mathbf{x}_i^{(r)}$ generated by Eqs. (\ref{eq:reformulated_2_stochastic_ecl_1}-\ref{eq:reformulated_2_stochastic_ecl_3}) satisfies the following:
\begin{align}
\label{eq:average_in_ecl}
    \bar{\mathbf{x}}^{(r+1)} &= \bar{\mathbf{x}}^{(r)} - \frac{\eta^\prime}{n} \sum_{i=1}^{n} \nabla F_i(\mathbf{x}_i^{(r)} ; \xi_i^{(r)}).
\end{align}
\end{lemma}

%Moreover, when the model parameter $\mathbf{x}_i$ is initialized with the same parameters for all $i\in[n]$,
%$\mathbf{c}_i$ can be considered to be initialized to $\mathbf{0}$ for all $i\in[n]$.
%
%\begin{remark}
%Suppose that $\mathbf{x}_i^{(0)}$ is initialized with the same parameters for all $i\in[n]$.
%Then, under the assumptions of Lemma \ref{lemma:c_is_zero},
%we can consider $\mathbf{c}_i^{(0)}=\mathbf{0}$
%in the update formulas Eqs. (\ref{eq:reformulated_2_stochastic_ecl_1}-\ref{eq:reformulated_2_stochastic_ecl_3}).
%\end{remark}

\section{Generalized Edge-Consensus Learning}
\label{sec:gecl}
In Sec. \ref{sec:reformulation},
we show that
a node computes the average with its neighbors using the mixing matrix $\mathbf{W}$ in the ECL as well as in the Gossip algorithm
and then updates the model parameter by the stochastic gradient descent modified by the term $\mathbf{c}_i$. 
However, in contrast with the Gossip algorithm,
$\mathbf{W}$, $\eta^\prime$ and $\{ \alpha_{i|j} \}_{ij}$ depend on each other in the ECL,
which makes the convergence analysis difficult.
Then, we refer to the update formulas Eqs. (\ref{eq:reformulated_2_stochastic_ecl_1}-\ref{eq:reformulated_2_stochastic_ecl_3})
as the \textbf{Generalized ECL (G-ECL)} when $\mathbf{W}$, $\eta^\prime$, and $\{ \alpha_{i|j} \}$ are set independently as hyperparameters
and provide the convergence rate of the G-ECL in Sec. \ref{sec:convergence_results}.
Then, we experimentally demonstrate that the ECL converges at the same rate as the G-ECL in Sec. \ref{sec:experimental_results}.
The pseudo-code of the G-ECL is illustrated in Sec. \ref{sec:algorithm_gecl}.
Note that because the G-ECL is equivalent to the ECL when $\mathbf{W}$ and $\eta^\prime$ are set as in Eq. \eqref{eq:definition_of_eta_prime},
the ECL is a special case of the G-ECL.

\section{Setup}
\label{sec:setup}
In this section, we introduce the assumptions and notations used in the convergence analysis in the next section.
We define $b^\prime \coloneqq \left\| \mathbf{W} - \mathbf{I} \right\|^2_F$
and $\bar{\mathbf{x}}^{(r)} \coloneqq \frac{1}{n} \sum_{i=1}^n \mathbf{x}_i^{(r)}$
and denote $f^\star$ as the optimal value of Eq. \eqref{eq:decentralized_learning}.
When $f_i$ is convex for all $i \in [n]$, we denote $\mathbf{x}^\star \in \mathbb{R}^d$ as the optimal solution of Eq. \eqref{eq:decentralized_learning}.
Next, we introduce the assumptions used for the convergence analysis of the G-ECL.
\begin{assumption}[Mixing Matrix]
\label{assumption:mixing_matrix}
There exists $p \in (0, 1]$ such that for any $\mathbf{x}_1, \cdots, \mathbf{x}_n \in \mathbb{R}^d$,
\begin{align}
\label{eq:assumption:mixing_matrix}
    \left\| \mathbf{X} \mathbf{W} - \bar{\mathbf{X}} \right\|^2_F \leq (1-p) \left\| \mathbf{X} - \bar{\mathbf{X}} \right\|^2_F,
\end{align}
where $\mathbf{X}=(\mathbf{x}_1, \cdots, \mathbf{x}_n)\in\mathbb{R}^{d\times n}$
and $\bar{\mathbf{X}} \coloneqq \frac{1}{n} \mathbf{X} \mathbf{1}\mathbf{1}^\top$.
\end{assumption}
\begin{assumption}[$L$-smoothness]
\label{assumption:smoothness}
For any $i \in [n]$, there exists $L>0$ such that for all $\mathbf{x}, \mathbf{y} \in \mathbb{R}^d$,
\begin{align}
\label{eq:assumption:smoothness}
    \| \nabla f_i(\mathbf{x}) - \nabla f_i (\mathbf{y}) \| \leq L \| \mathbf{x} - \mathbf{y} \|.
\end{align}
\end{assumption}
\begin{assumption}[Bounded Gradient Noise]
\label{assumption:stochastic_gradient}
For any $i\in[n]$, there exists $\sigma \geq 0$ such that for all $\mathbf{x}_i \in \mathbb{R}^d$, 
\begin{align}
\label{eq:assumption:stochastic_gradient}
    \mathbb{E}_{\xi_i \sim \mathcal{D}_i} \| \nabla F_i(\mathbf{x}_i ; \xi_i) - \nabla f_i(\mathbf{x}_i) \|^2 \leq \sigma^2.
\end{align}
\end{assumption}
\begin{assumption}[$\mu$-convexity]
\label{assumption:mu_convexity}
For any $i \in [n]$, there exists $\mu\geq0$ such that for all $\mathbf{x}, \mathbf{y} \in \mathbb{R}^d$,
\begin{align}
\label{eq:assumption:mu_convexity}
    f_i (\mathbf{x}) \geq f_i(\mathbf{y}) + \langle \nabla f_i(\mathbf{y}), \mathbf{x} - \mathbf{y} \rangle + \frac{\mu}{2} \| \mathbf{x} - \mathbf{y} \|^2.
\end{align}
\end{assumption}
Assumptions \ref{assumption:mixing_matrix}, \ref{assumption:smoothness}, \ref{assumption:stochastic_gradient}, and \ref{assumption:mu_convexity} are commonly used in convergence analyses of decentralized learning algorithms \citep{koloskova2021improved,lian2018asynchronous,vogels2020practical}.
In addition, the following assumption, which represents the heterogeneity of data distributions, is commonly used \citep{lian2018asynchronous,vogels2020practical}.
However, this assumption is not necessary for the convergence analysis of the G-ECL shown in Theorem \ref{theoram:convergence_rate}.
\begin{assumption}[Bounded Heterogeneity]
\label{assumption:hetero}
There exists $\zeta \geq 0$ such that for all $\mathbf{x} \in \mathbb{R}^d$,
\begin{align*}
    \frac{1}{n} \sum_{i=1}^{n} \| \nabla f_i(\mathbf{x}) - \nabla f(\mathbf{x}) \|^2 \leq \zeta^2.
\end{align*}
\end{assumption}

\section{Convergence Results}
\label{sec:convergence_results}
In this section, we present the convergence results of the G-ECL.
Our convergence analysis is based on the analysis of the Gossip algorithm \citep{koloskova2020unified}, 
and all the proofs are presented in Sec. \ref{sec:proof_of_main}.

\subsection{Main Theorem}

\begin{theorem}
\label{theoram:convergence_rate}
Suppose that Assumptions \ref{assumption:mixing_matrix}, \ref{assumption:smoothness}, \ref{assumption:stochastic_gradient} hold,
$\{ \alpha_{i|j} \}_{ij}$ is set such that $\alpha_{i|j} = \alpha_{j|i} \geq 0$ for all $(i,j) \in \mathcal{E}$,
$\mathbf{W}$ is set to be a mixing matrix,
and $\mathbf{x}_i^{(0)}$ is initialized with the same parameters for all $i\in[n]$.
%and $\mathbf{c}_{i}^{(0)}$ is initialized to zero for all $i\in[n]$.

\textbf{Non-convex:}
In addition, suppose that $\mathbf{c}_i$ is initialized to $\nabla f_i (\mathbf{x}_i^{(0)}) - \nabla f (\mathbf{x}_i^{(0)})$.
Then, there exists a step size $\eta^\prime < \frac{1}{s}$ such that the average $\bar{\mathbf{x}}^{(r)}$ generated by the G-ECL satisfies
\begin{align*}
    \frac{1}{R} \sum_{r=0}^{R-1} \mathbb{E} \left\| \nabla f(\bar{\mathbf{x}}^{(r)}) \right\|^2 
    \leq 
    \mathcal{O} \left( \sqrt{\frac{r_0 \sigma^2 L}{n R}} 
    + \left( \frac{ r_0^2 L^2 \tilde{\sigma}^2}{p R^2} \right)^{\frac{1}{3}}
    + \frac{s r_0}{R}
    \right),
\end{align*}
where $r_0 \coloneqq f(\bar{\mathbf{x}}^{(0)}) - f^\star$ and $\tilde{\sigma}^2 \coloneqq (1+\frac{b^\prime}{p^3}) \sigma^2$.

\textbf{General Convex:}
In addition, suppose that $f_i$ is convex for all $i \in [n]$, Assumption \ref{assumption:mu_convexity} holds with $\mu=0$,
and $\mathbf{c}_i$ is initialized to $\nabla f_i (\mathbf{x}^\star)$.
Then, there exists a step size $\eta^\prime < \frac{1}{s}$ such that the average $\bar{\mathbf{x}}^{(r)}$ generated by the G-ECL satisfies
\begin{align*}
    \frac{1}{R} \sum_{r=0}^{R-1} (\mathbb{E}[f(\bar{\mathbf{x}}^{(r)})] - f^\star) 
    \leq \mathcal{O} \left(
    \sqrt{\frac{r_0 \sigma^2}{n R}} 
    + \left( \frac{r_0^2  L \tilde{\sigma}^2 }{ p R^2 } \right)^{\frac{1}{3}}
    + \frac{s r_0}{R}
    \right),
\end{align*}
where $r_0 \coloneqq \| \bar{\mathbf{x}}^{(0)} - \mathbf{x}^\star \|^2$
and $\tilde{\sigma}^2 \coloneqq (1 + \frac{b^\prime}{p^2}) \sigma^2$.

\textbf{Strongly Convex:}
In addition, suppose that $f_i$ is convex for all $i \in [n]$, Assumption \ref{assumption:mu_convexity} holds with $\mu>0$,
and $\mathbf{c}_i$ is initialized to $\nabla f_i (\mathbf{x}^\star)$.
Let $w^{(r)} \coloneqq (1 - \frac{\mu \eta^\prime}{2})^{-(r+1)}$ and $W_R \coloneqq \sum_{r=0}^{R} w^{(r)}$.
Then, there exists a step size $\eta^\prime < \frac{1}{s}$ such that the average $\bar{\mathbf{x}}^{(r)}$ generated by the G-ECL satisfies
\begin{align*}
    &\sum_{r=0}^{R} \frac{w^{(r)}}{W_R} (\mathbb{E}[f(\bar{\mathbf{x}}^{(r)})] - f^\star) 
    + \mu \mathbb{E} \| \bar{\mathbf{x}}^{(R+1)} \!\!\! - \mathbf{x}^\star \|^2 
    \leq
    \tilde{\mathcal{O}} \left( \! r_0 s \exp \left[ \frac{-\mu (R+1)}{s} \right]
    \!+\! \frac{\sigma^2}{\mu n R}
    \!+\! \frac{L \tilde{\sigma}^2}{p \mu^2 R^2}
    \! \right),
\end{align*}
where $r_{0} \coloneqq \| \bar{\mathbf{x}}^{(0)} - \mathbf{x}^\star \|^2$, $\tilde{\sigma}^2 \coloneqq (1 + \frac{b^\prime}{p^2}) \sigma^2$, and $\tilde{\mathcal{O}}(\cdot)$ hides polylogarithmic factors.
\end{theorem}

\textbf{Limitation:}
Theorem \ref{theoram:convergence_rate} shows the convergence rates of the G-ECL,
but does not show that of the ECL
because $\mathbf{W}$, $\eta^\prime$, and $\{ \alpha_{i|j} \}_{ij}$ depend on each other in the ECL.
Specifically, our analysis does not prove that there exists a step size $\eta^\prime$
for the ECL to achieve the convergence rate shown in Theorem \ref{theoram:convergence_rate}.
In this work, we only provide the convergence rate of the G-ECL 
and experimentally demonstrate that the ECL also converges at the same rate as the G-ECL in Sec. \ref{sec:experimental_results}.
A more detailed discussion is provided in Sec. \ref{sec:limitation}.

\subsection{Discussion}
\label{sec:discussion}
In this section, we discuss the convergence rate of the G-ECL compared with that of the D-PSGD \citep{lian2017can}.
Table \ref{table:convergence_rate} lists the convergence rates of the G-ECL and D-PSGD.
Here, we discuss only the strongly convex case,
but this discussion holds for the convex and non-convex cases.

% ヘテロについての議論
First, we discuss the effect of the heterogeneity of data distributions $\zeta$ in Assumption \ref{assumption:hetero} on the convergence rate.
Table \ref{table:convergence_rate} shows that the convergence rate of the D-PSGD depends on the heterogeneity of data distributions $\zeta$,
while the convergence rate of the G-ECL does not depend on $\zeta$.
Therefore, Theorem \ref{theoram:convergence_rate} indicates that the G-ECL is robust to the heterogeneity of data distributions,
which is consistent with previous works \citep{niwa2020edge,niwa2021asynchronous}
that experimentally demonstrated that the ECL is robust to the heterogeneity of data distributions.

% (1-p)について。さらにp=1には定義上できないことを述べる?
Next, we discuss the factor $(1-p)$ contained in the convergence rate of the D-PSGD.
In the D-PSGD, the third term in the convergence rate is multiplied by $(1-p)$,
but in the G-ECL, the third term in the convergence rate is not multiplied by $(1-p)$.
That is, in the D-PSGD, the third term is $0$ when $p=1$ (i.e., $G$ is fully connected graph and $\mathbf{W}=\frac{1}{n}\mathbf{1}\mathbf{1}^\top$),
but in the G-ECL, the third term is not $0$ for any $p$.
This is because the orders of calculation of the weighted average and stochastic gradient descent are different.
The analysis of the D-PSGD evaluates the average $\bar{\mathbf{x}}^{(r)}$ 
after computing the weighted average of Eq. \eqref{eq:gossip},
whereas the analysis of the G-ECL evaluates the average $\bar{\mathbf{x}}^{(r)}$
before computing the weighted average of Eq. \eqref{eq:reformulated_2_stochastic_ecl_1}.
Thus, the third term is not multiplied by $(1-p)$ in the G-ECL.

\section{Experiments}
\label{sec:experimental_results}
In this section, using the synthetic dataset,
we experimentally demonstrate that the numerical results of the G-ECL and ECL are consistent with the convergence rate of the G-ECL shown in Theorem \ref{theoram:convergence_rate}.
Following the previous work \citep{koloskova2020unified}, we focus only on the strongly convex case.

\textbf{Comparison Methods:}
We compare the D-PSGD \citep{lian2017can}, ECL \citep{niwa2020edge}, and G-ECL.
In the D-PSGD, we use Metropolis-Hasting weights (i.e., $W_{ij}=W_{ji}=1/(|\mathcal{N}_i|+1)$)
and set the step size $\eta=10^{-3}$.
In the ECL, we set $\{ \alpha_{i|j} \}_{ij}$ as in Example \ref{example:alpha}.
%because ring, torus, and fully connected graph are regular graph.
Then, we set $\eta=0.5$, $\alpha=10^{3}$ (i.e., $\eta^\prime \simeq 10^{-3}$).
In the G-ECL, we set $W_{ij}=W_{ji}=1/(|\mathcal{N}_i|+1)$,
$\eta^\prime=10^{-3}$, and $\alpha_{i|j}=0$.
Note that the ECL can be regarded as a special case of the G-ECL.

\textbf{Synthetic Dataset and Network Topology:}
We set the dimension of the parameter $d=50$ and the number of nodes $n=25$. 
We set the objective function as $f_i(\mathbf{x}) \coloneqq \frac{1}{2} \| \mathbf{x} - \mathbf{b}_i \|^2$
and $\mathbf{b}_i$ is drawn from $\mathcal{N}(\mathbf{0}, \frac{\zeta^2}{d} \mathbf{I})$ for each $i \in [n]$.
The stochastic gradient is defined as $\nabla F_i(\mathbf{x} ; \xi_i) \coloneqq \nabla f_i(\mathbf{x}) + \epsilon$ where $\epsilon$ is drawn from $\mathcal{N}(\mathbf{0}, \frac{\sigma^2}{d} \mathbf{I})$ at each time.
Note that the parameters $\zeta$ and $\sigma$ correspond to Assumptions \ref{assumption:hetero} and \ref{assumption:stochastic_gradient}.
We evaluate the D-PSGD, ECL, and G-ECL on three network topologies consisting $n$ nodes: ring, torus, and fully connected graph.
We implement all comparison methods with PyTorch \cite{paszke2019pytorch},
and all the experiments are executed on a machine with Intel Xeon CPU E7-8890 v4.

\subsection{Numerical Results}
\begin{figure}[!t]
\subfigure[G-ECL]{
  \begin{minipage}[b]{0.33\columnwidth}
    \centering
    \includegraphics[width=\columnwidth]{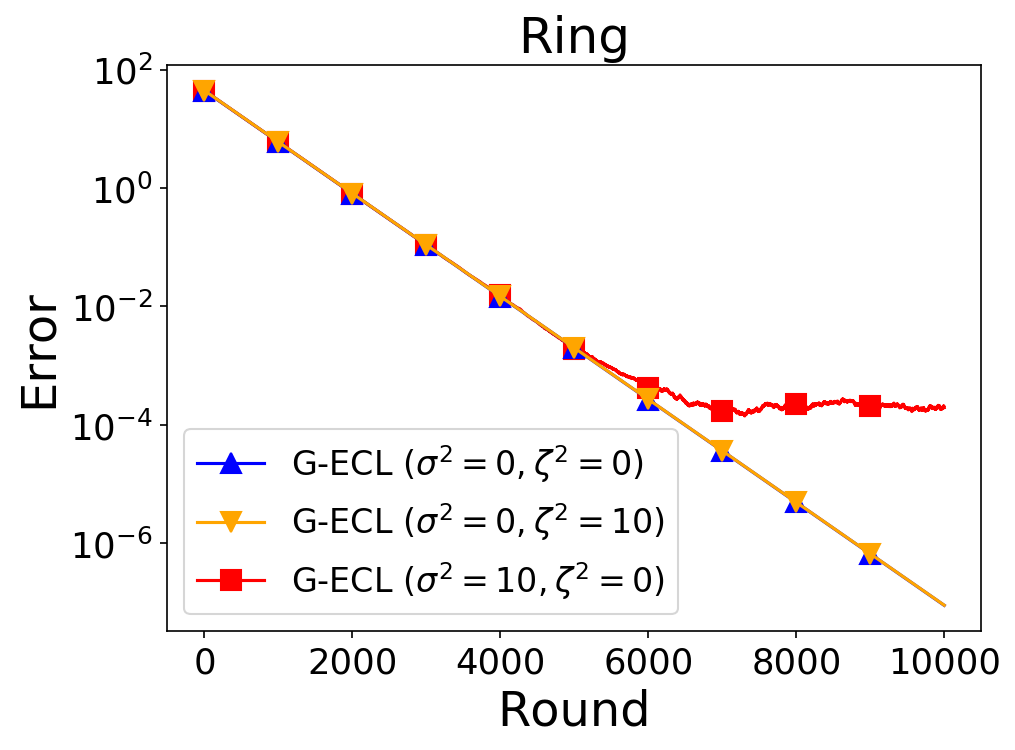}
  \end{minipage}
  \begin{minipage}[b]{0.33\columnwidth}
    \centering
    \includegraphics[width=\columnwidth]{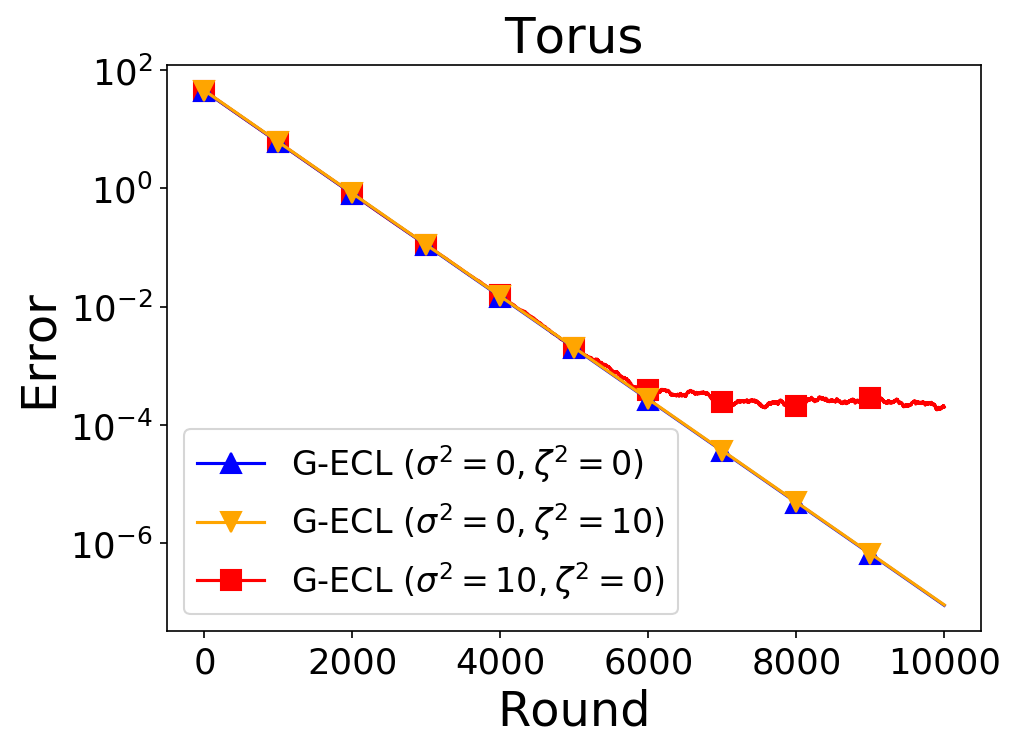}
  \end{minipage}
  \begin{minipage}[b]{0.33\columnwidth}
    \centering
    \includegraphics[width=\columnwidth]{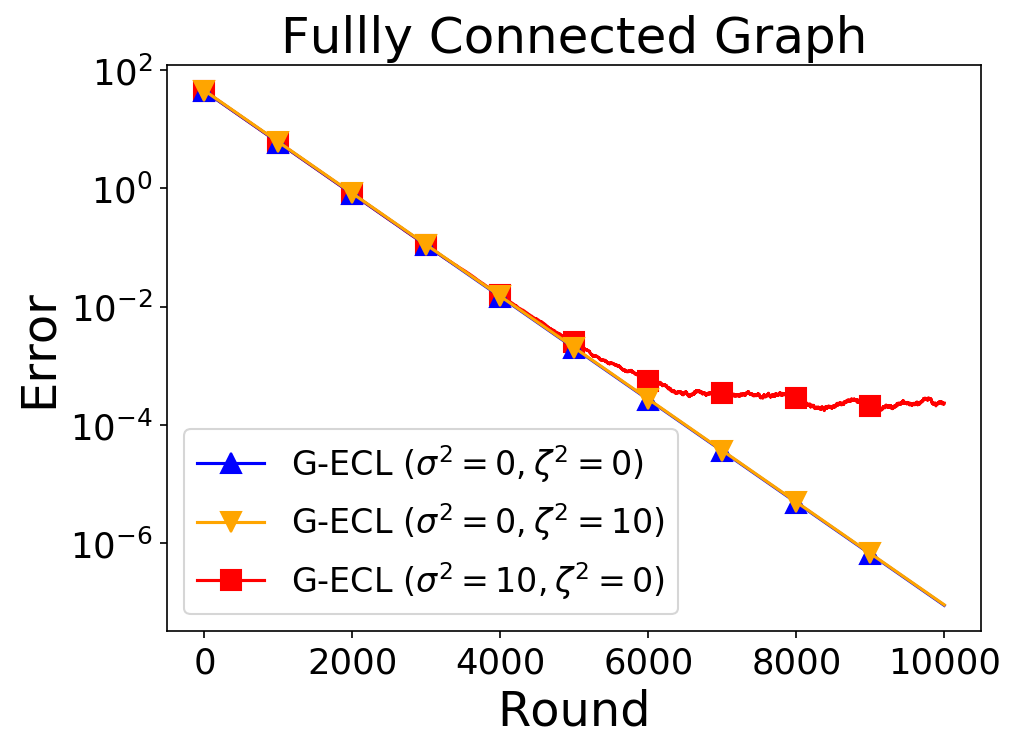}
  \end{minipage}
}
\vskip -0.05 in
\subfigure[ECL]{ 
  \begin{minipage}[b]{0.33\columnwidth}
    \centering
    \includegraphics[width=\columnwidth]{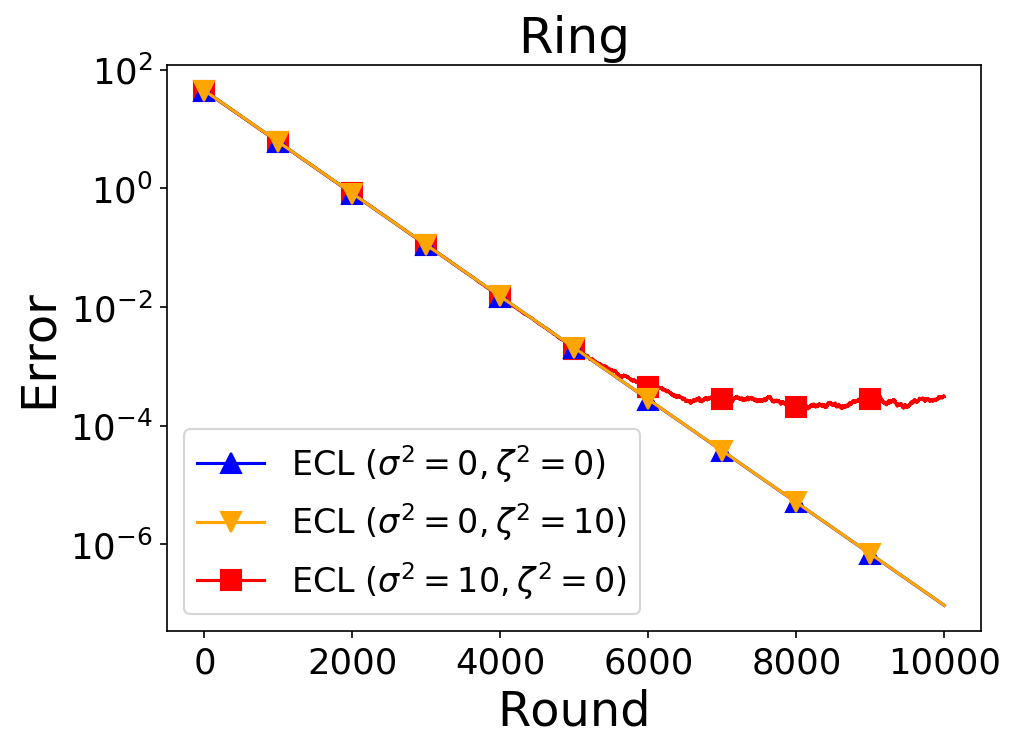}
  \end{minipage}
  \begin{minipage}[b]{0.33\columnwidth}
    \centering
    \includegraphics[width=\columnwidth]{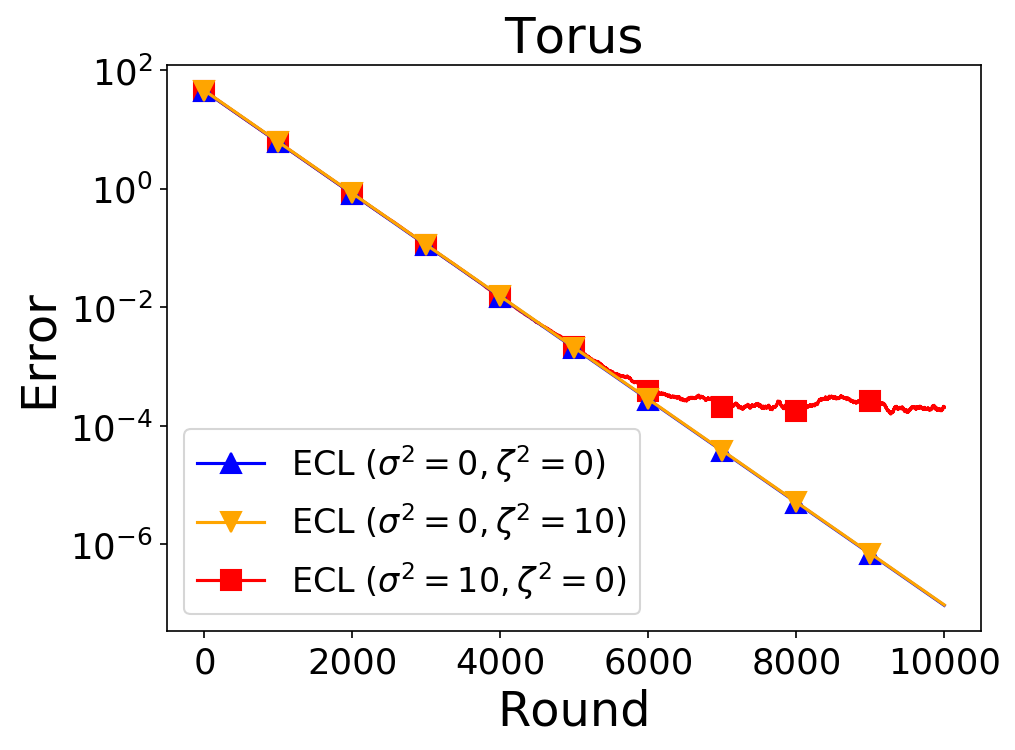}
  \end{minipage}
  \begin{minipage}[b]{0.33\columnwidth}
    \centering
    \includegraphics[width=\columnwidth]{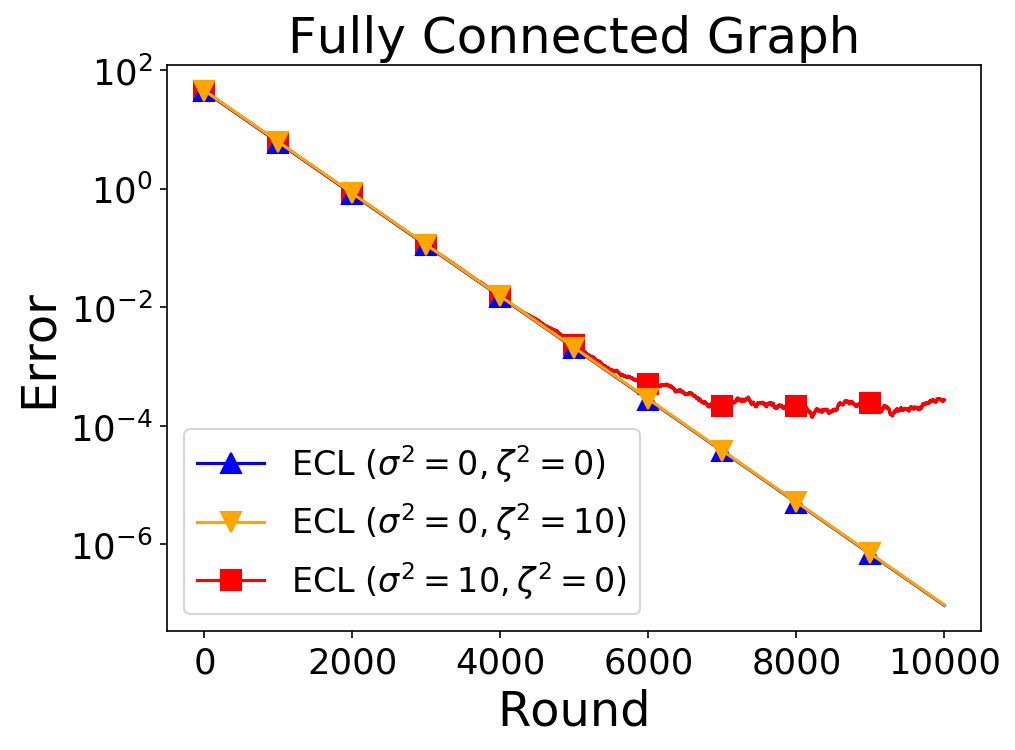}
  \end{minipage}
}
\vskip -0.05 in
\subfigure[D-PSGD]{
  \begin{minipage}[b]{0.33\columnwidth}
    \centering
    \includegraphics[width=\columnwidth]{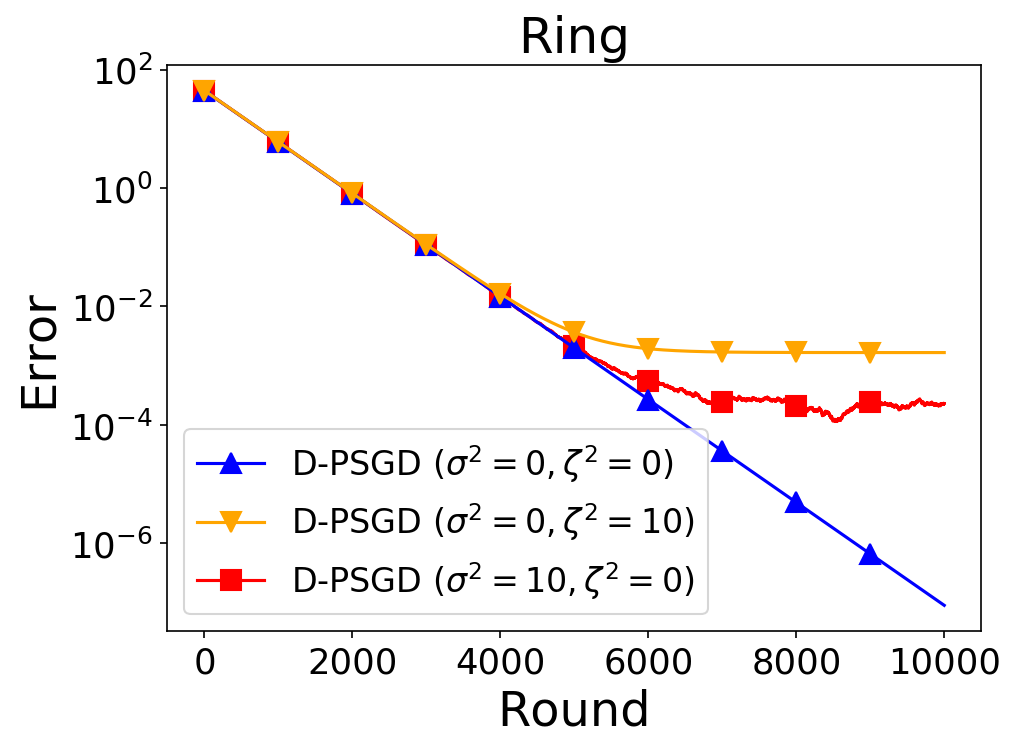}
  \end{minipage}
  \begin{minipage}[b]{0.33\columnwidth}
    \centering
    \includegraphics[width=\columnwidth]{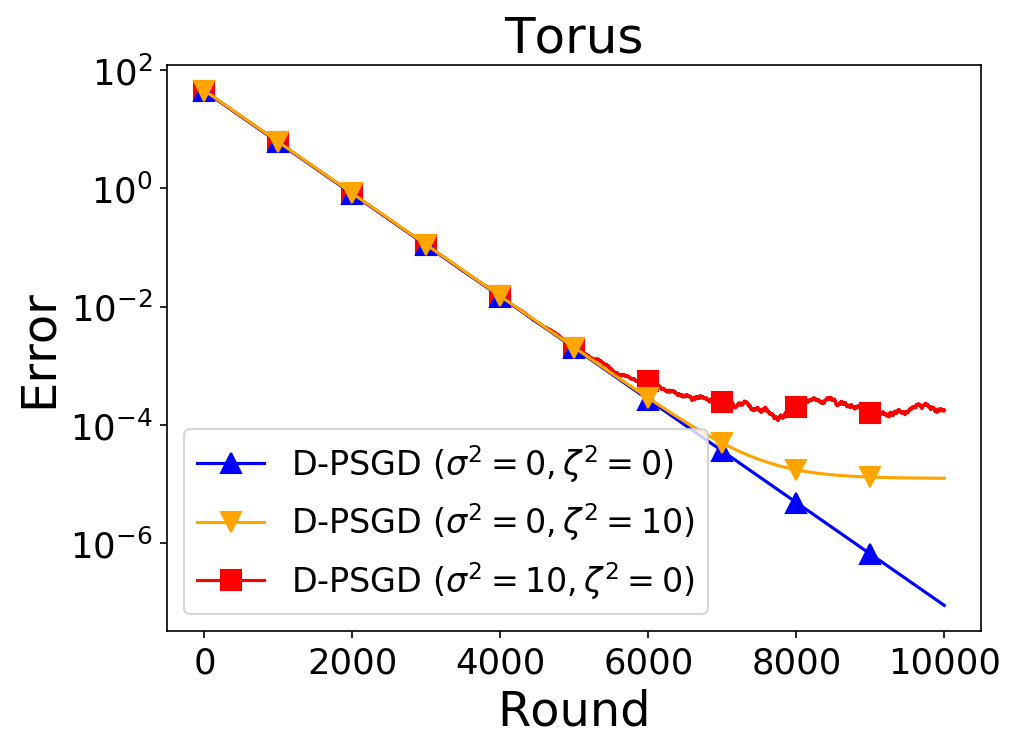}
  \end{minipage}
  \begin{minipage}[b]{0.33\columnwidth}
    \centering
    \includegraphics[width=\columnwidth]{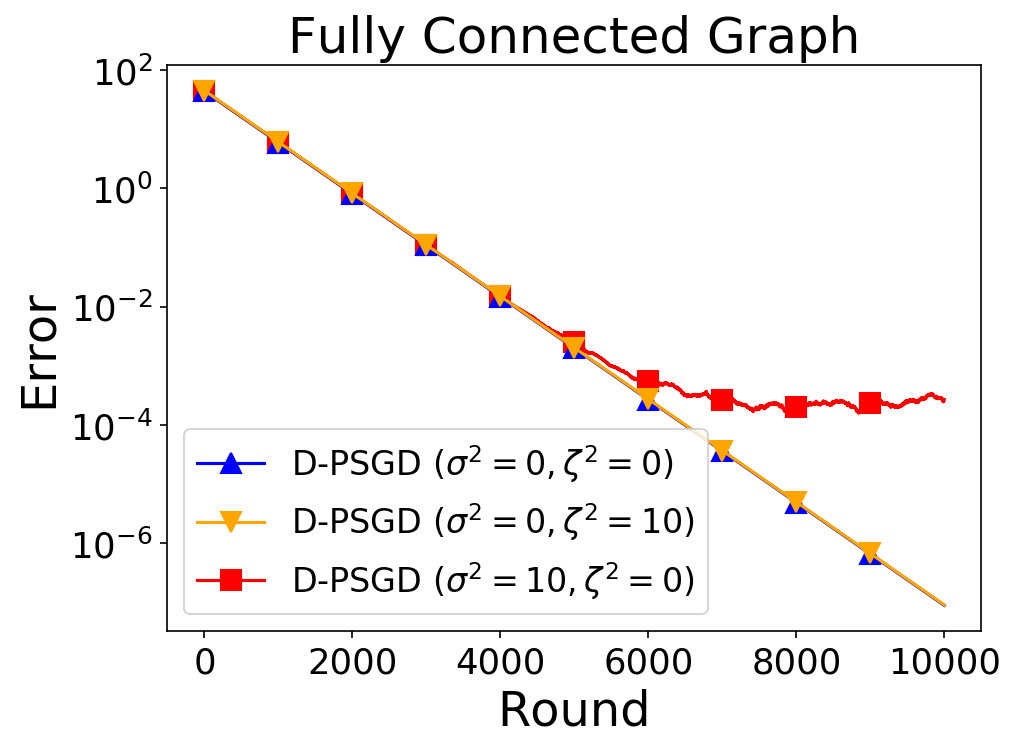}
  \end{minipage}
}
\vskip -0.1 in
\caption{Convergence of the error $\frac{1}{n} \sum_{i=1}^{n} \| \mathbf{x}_i^{(r)} - \mathbf{x}^\star \|^2$ when varying the heterogeneity of data distributions $\zeta$ and the noise of the stochastic gradient $\sigma$.}
\label{fig:convergence_rate}
\vskip -0.2 in
\end{figure}

In this section, we demonstrate that the convergence rate of the G-ECL shown in Theorem \ref{theoram:convergence_rate} coincides with the numerical results of both the G-ECL and ECL.
Fig. \ref{fig:convergence_rate} shows the error $\frac{1}{n} \sum_{i=1}^n \| \mathbf{x}_i^{(r)} - \mathbf{x}^\star \|^2$ at each round $r$ 
when varying the heterogeneity of data distributions $\zeta$ and the noise of stochastic gradient $\sigma$.

\textbf{Effect of Heterogeneity of Data Distributions:}
First, we discuss the effect of the heterogeneity of data distributions $\zeta$ on the convergence rate.
When $\sigma^2=0$, 
the results show that the G-ECL and ECL converge with $\mathcal{O}(\exp{(-R)})$ in both cases when $\zeta^2=0$ and $\zeta^2=10$ for all network topologies.
By contrast, when $\sigma^2=0$ and $G$ is ring or torus (i.e., $p<1$), 
the convergence of the D-PSGD slows down when $\zeta^2=10$ compared to when $\zeta^2=0$.
When $\sigma^2=0$ and $G$ is fully connected graph (i.e., $\mathbf{W}=\frac{1}{n}\mathbf{1}\mathbf{1}^\top$ and $p=1$),
the D-PSGD converges with $\mathcal{O}(\exp{(-R)})$ in both cases when $\zeta^2=0$ and when $\zeta^2=10$,
as the convergence rate of the D-PSGD shown in Table \ref{table:convergence_rate} indicates.
Therefore, these numerical results show that the convergence rates of both the G-ECL and ECL do not depend on the heterogeneity of data distributions $\zeta$.
This is consistent with Theorem \ref{theoram:convergence_rate}.

\textbf{Effect of Noise of Stochastic Gradient:}
Next, we discuss the effect of the noise of the stochastic gradient $\sigma$ on the convergence rate.
When $\zeta^2=0$, the results show that the convergence rates of the D-PSGD, ECL, and G-ECL slow down when $\sigma^2=10$ compared to when $\sigma^2=0$ on all network topologies,
which is consistent with Theorem \ref{theoram:convergence_rate} and the convergence rate of the D-PSGD.

\textbf{Comparison with ECL and G-ECL:}
Next, we compare the results of the ECL and G-ECL.
As we discuss in Sec. \ref{sec:convergence_results} and \ref{sec:limitation},
Theorem \ref{theoram:convergence_rate} provides only the convergence rates of the G-ECL
and does not show that there exists a step size $\eta^\prime$ for the ECL to achieve the convergence rates shown in Theorem \ref{theoram:convergence_rate}.
However, Fig. \ref{fig:convergence_rate} shows that the results of the ECL and G-ECL are almost equivalent for all settings,
and, as discussed above, the numerical results of both the ECL and G-ECL coincide with the convergence rate of the G-ECL.
Thus, experimentally, the ECL also converges with the convergence rates provided in Theorem \ref{theoram:convergence_rate}.

\section{Conclusion}
In this work, we first investigate the relationship between the Gossip algorithm and ECL.
Specifically, we show that if the hyperparameter of the ECL is set such that Assumption \ref{assumption:definition_of_alpha} holds,
a node computes the average with its neighbors in the ECL as well as in the Gossip algorithm,
and the update formulas of the ECL can be regarded as correcting the local stochastic gradient $\nabla f_i(\mathbf{x}_i ; \xi_i)$ in the Gossip algorithm.
Subsequently, to make the convergence analysis tractable, 
we increase the degrees of freedom of hyperparameters of the ECL
and propose the G-ECL, which contains the ECL as a special case.
By using the proof techniques of the Gossip algorithm \citep{koloskova2020unified}, 
we provide the convergence rate of the G-ECL in (strongly) convex and non-convex settings
and show that they do not depend on the heterogeneity of data distributions.
Through the synthetic experiments, 
we demonstrate that the numerical results of both the ECL and G-ECL coincide with the convergence rate of the G-ECL.

\bibliography{example_paper}

%%%%%%%%%%%%%%%%%%%%%%%%%%%%%%%%%%%
%%%%%% SUPPLEMENT (OPTIONAL) %%%%%%
%%%%%%%%%%%%%%%%%%%%%%%%%%%%%%%%%%%
\newpage
\appendix
\section{Related Work}
\label{sec:related_work}

\subsection{Gossip Algorithm}
One of the most widely used algorithms for decentralized learning is the D-PSGD \cite{lian2017can} also known as the Gossip algorithm.
Recently, the convergence rate of the Gossip algorithm has been well investigated.
\citet{lian2018asynchronous} extended the Gossip algorithm to the asynchronous setting and analyzed the convergence rate.
\citet{koloskova2020unified} provided the convergence rate of the Gossip algorithm when the network topology $G$ changes over time or when using the local steps.
\citet{yuan2021decentlam} analyzed the convergence rate of the Gossip algorithm when using the momentum SGD instead of the SGD.
These theoretical analyses indicate that the convergence rate of the Gossip algorithm slows down when the data distribution held by each node is statistically heterogeneous.

\subsection{Primal-Dual Algorithm}
In addition to the Gossip algorithm,
primal-dual algorithms are applicable to decentralized learning \cite{hong2017proximal,liu2021linear,kovalev2021linearly}.
As shown in Eq. \eqref{eq:primal}, the decentralized learning problem can be formulated as a linearly constrained problem.
One of the most famous algorithms for solving a linearly constrained problem is the ADMM, which has been applied to decentralized learning \cite{boyd2011distributed,zhange2014asynchronous}.
\citet{zhang2018distributed} proposed the PDMM and showed the PDMM converges faster than the ADMM.
Recently, \citet{sherson2019derivation} showed that the PDMM can be naturally derived by using Douglas-Rachford splitting \cite{douglas1955numerical},
and \citet{niwa2020edge} applied it to a neural network,
which is called the ECL.
Recently, \citet{rajawat2020primal} provided the convergence rate of the ECL in the convex case
and proposed to apply stochastic variance reduction methods \cite{zhu2016improved,defazip2014saga} to the ECL.

\subsection{Gradient Tracking Method}
One of the most popular algorithms whose convergence rate does not depend on the heterogeneity of the data distributions is the gradient tracking method \cite{lorenzo2016next,nedi2017achieving,koloskova2021improved}.
In addition, in Sec. \ref{sec:gt_and_ecl}, we discuss the relationship between the ECL and gradient tracking method.

\newpage
\section{Relationship between Gradient Tracking Method and ECL}
\label{sec:gt_and_ecl}
In this section, we discuss the relationship between the gradient tracking methods \citep{lorenzo2016next,nedi2017achieving,koloskova2021improved} and ECL.

\subsection{Gradient Tracking Method}
In the gradient tracking method, the model parameter $\mathbf{x}_i$ is updated as follows:
\begin{align}
    \label{eq:gt_1}
    \mathbf{x}_i^{(r+1)} &= \sum_{j \in \mathcal{N}_i^+} W_{ij} \left( \mathbf{x}_j^{(r)} - \eta \mathbf{p}_j^{(r)} \right), \\
    \label{eq:gt_2}
    \mathbf{p}_i^{(r+1)} &= \sum_{j \in \mathcal{N}_i^+} W_{ij} \mathbf{p}_j^{(r)} + \left( \nabla F_i(\mathbf{x}_i^{(r+1)} ; \xi_i^{(r+1)}) - \nabla F_i(\mathbf{x}_i^{(r)} ; \xi_i^{(r)}) \right),
\end{align}
where $\mathbf{W}$ is assumed to be a mixing matrix, as in the Gossip algorithm.

\subsection{Discussion}
To discuss the relationship between the gradient tracking method and ECL, 
we further reformulate the update formulas of the ECL.
\begin{theorem}
Suppose that the hyperparameter $\theta=\frac{1}{2}$,
the dual variable $\mathbf{z}_{i|j}^{(0)}$ is initialized to $\mathbf{A}_{i|j} \mathbf{x}_j^{(0)}$,
and the hyperparameter $\{ \alpha_{i|j} \}_{ij}$ is set such that $\alpha_{i|j} = \alpha_{j|i} \geq 0$
for all $(i,j) \in \mathcal{E}$.
Then, the update formulas Eq. (\ref{eq:stochastic_ecl}) and Eqs. (\ref{eq:ecl_2}-\ref{eq:ecl_3}) are equivalent to the following:
\begin{align}
    \label{eq:reformulated_3_stochastic_ecl_1}
    \tilde{\mathbf{x}}_i^{(r)} &= \sum_{j\in\mathcal{N}_i^+} W_{ij} \mathbf{x}_j^{(r)}, \\
    \label{eq:reformulated_3_stochastic_ecl_2}
    \mathbf{x}_i^{(r+1)} &=  \tilde{\mathbf{x}}_i^{(r)} - \eta^\prime_i \mathbf{p}_i^{(r)}, \\
    \label{eq:reformulated_3_stochastic_ecl_3}
    \mathbf{p}_i^{(r+1)} 
    &= \sum_{j \in \mathcal{N}_i^+} W_{ij} \mathbf{p}_j^{(r)}
    + \left( \nabla F_i(\mathbf{x}_i^{(r+1)} ; \xi_i^{(r+1)}) - \nabla F_i(\mathbf{x}_i^{(r)} ; \xi_i^{(r)}) \right) \\
    &\qquad - \underbrace{\sum_{j \in \mathcal{N}_i} \frac{\alpha_{i|j}}{2} ( \tilde{\mathbf{x}}_j^{(r)} - \tilde{\mathbf{x}}_i^{(r)} )}_{T} \nonumber,
\end{align}
where $W_{ij}$ and $\eta^\prime_i$ are defined by Eq. \eqref{eq:definition_of_eta_prime}.
\end{theorem}
\begin{proof}
Defining $\mathbf{p}_i^{(r)} \coloneqq \nabla F_i(\mathbf{x}_i^{(r)} ; \xi_i^{(r)}) - \mathbf{c}_i^{(r)}$,
the statement follows from Theorem \ref{theorem:reformulation2}.
\end{proof}

If we omit the term $T$ in Eq. \eqref{eq:reformulated_3_stochastic_ecl_3},
the update formulas of the gradient tracking method Eqs. (\ref{eq:gt_1}-\ref{eq:gt_2}) and that of the ECL Eqs. (\ref{eq:reformulated_3_stochastic_ecl_1}-\ref{eq:reformulated_3_stochastic_ecl_3})
are almost equivalent.
The only difference is the order of the calculation of the weighted average and the parameter update.
Moreover, since Theorem \ref{theoram:convergence_rate} indicates that
the G-ECL converges when $\alpha_{i|j}=0$ for all $(i,j)\in\mathcal{E}$,
the term $T$ in Eq. \eqref{eq:reformulated_3_stochastic_ecl_3} does not play an important role in the convergence of the G-ECL.
Therefore, the ECL modifies the local stochastic gradient $\nabla F_i(\mathbf{x}_i ; \xi_i)$ in the update formulas of the Gossip algorithm as well as the gradient tracking methods, which makes the ECL robust to the heterogeneity of data distributions.

\newpage
\section{Issues of Existing Convergence Analysis of Edge-Consensus Learning}
\label{sec:issues_of_existing_proof}
In this section, we point out the issues in the proofs of the previous study \cite{niwa2021asynchronous} that attempted to analyze the convergence rate of the ECL.

The previous work \cite{niwa2021asynchronous} analyzed the ECL and proposed setting $\{ \alpha_{i|j} \}_{ij}$ as follows:
\begin{align}
\label{eq:ecl_isvr}
    \alpha_{i|j} = \frac{1}{\eta |\mathcal{N}_i| (K-1)},
\end{align}
where $K$ denotes the number of local steps.
Note that in this work, we provide the convergence rate of the G-ECL without local steps
(i.e., we provide the convergence rate when each node communicates with its neighbors at each update).
Then, when $\{ \alpha_{i|j} \}_{ij}$ is set as in Eq. \eqref{eq:ecl_isvr}, the ECL is named the ECL-ISVR,
and the previous work \citep{niwa2021asynchronous} attempted to analyze the convergence rate of the ECL-ISVR in both (strongly) convex and non-convex cases.

However, there are some errors in the proofs.
In the strongly convex and convex cases, strong approximations were used in the first and third equations in \cite[Sec. C.1]{niwa2021asynchronous},
and these equations do not hold for either the ECL or ECL-ISVR in practice.
Similarly, in the non-convex case, strong approximations were used in the first and third equations in \cite[Sec. C.2]{niwa2021asynchronous},
and these equations do not hold.
Therefore, the convergence rates shown in this previous work can not be regarded as those of the ECL and ECL-ISVR.

\newpage
\section{Additional Experiments}
In this section, we present a more detailed analysis of the effect of the heterogeneity of data distributions $\zeta$ 
and noise of the stochastic gradient $\sigma$ on the convergence rate.
\begin{figure}[!t]
  \begin{minipage}[b]{0.33\columnwidth}
    \centering
    \includegraphics[width=\columnwidth]{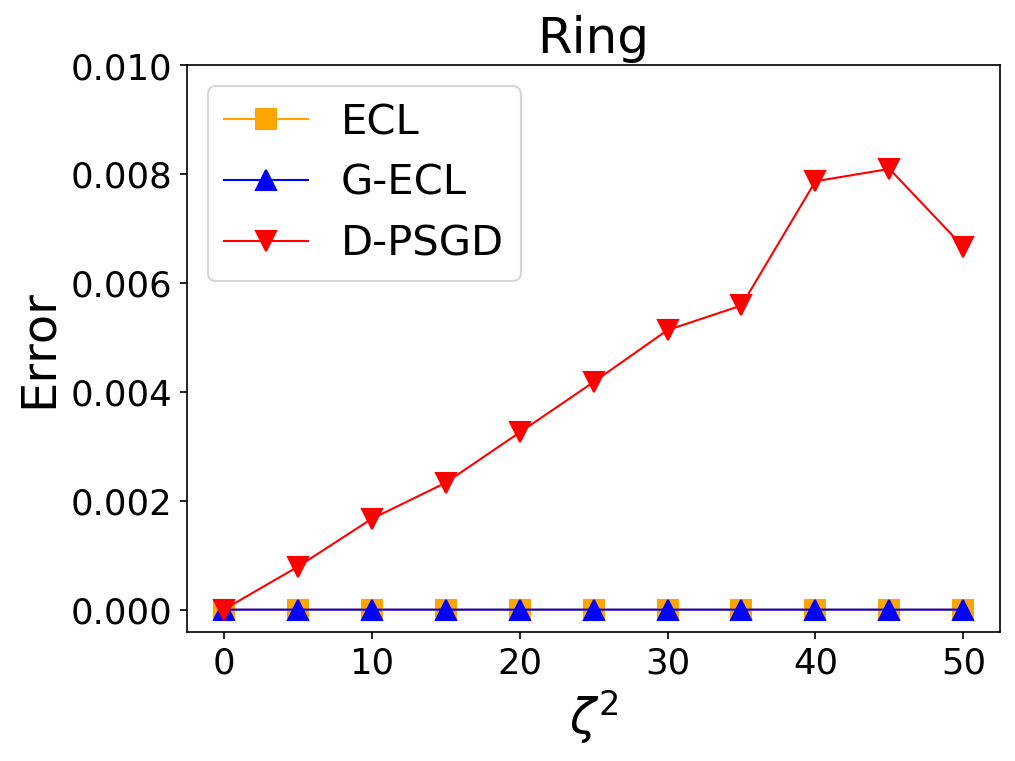}
  \end{minipage}
  \begin{minipage}[b]{0.33\columnwidth}
    \centering
    \includegraphics[width=\columnwidth]{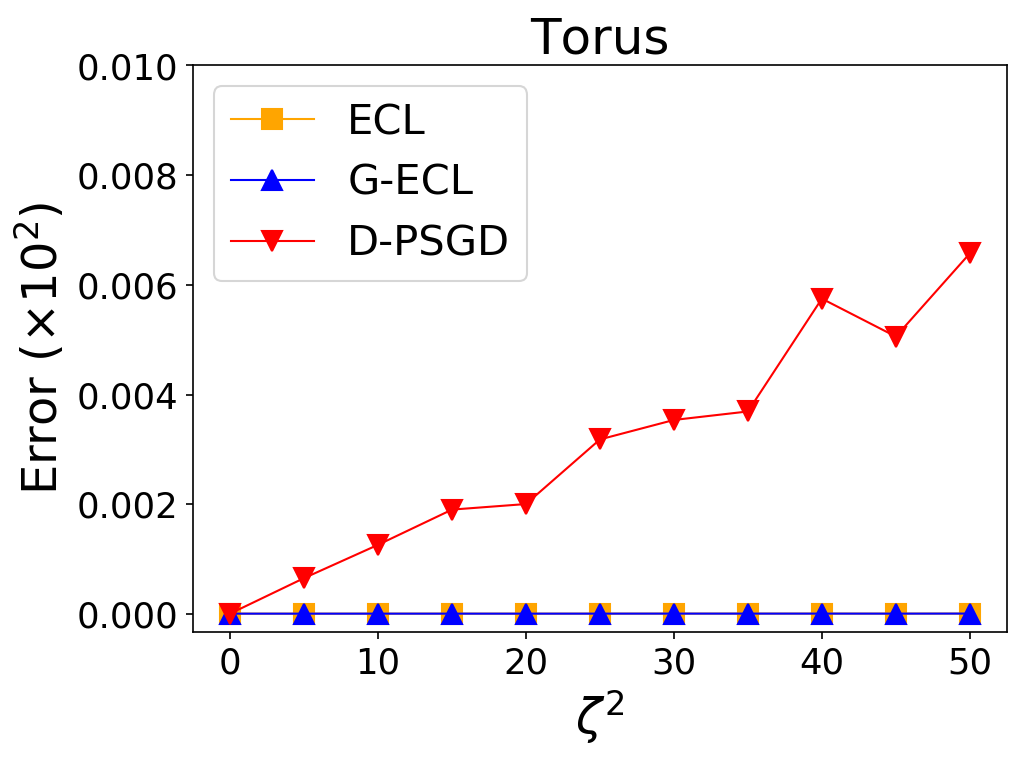}
  \end{minipage}
  \begin{minipage}[b]{0.33\columnwidth}
    \centering
    \includegraphics[width=\columnwidth]{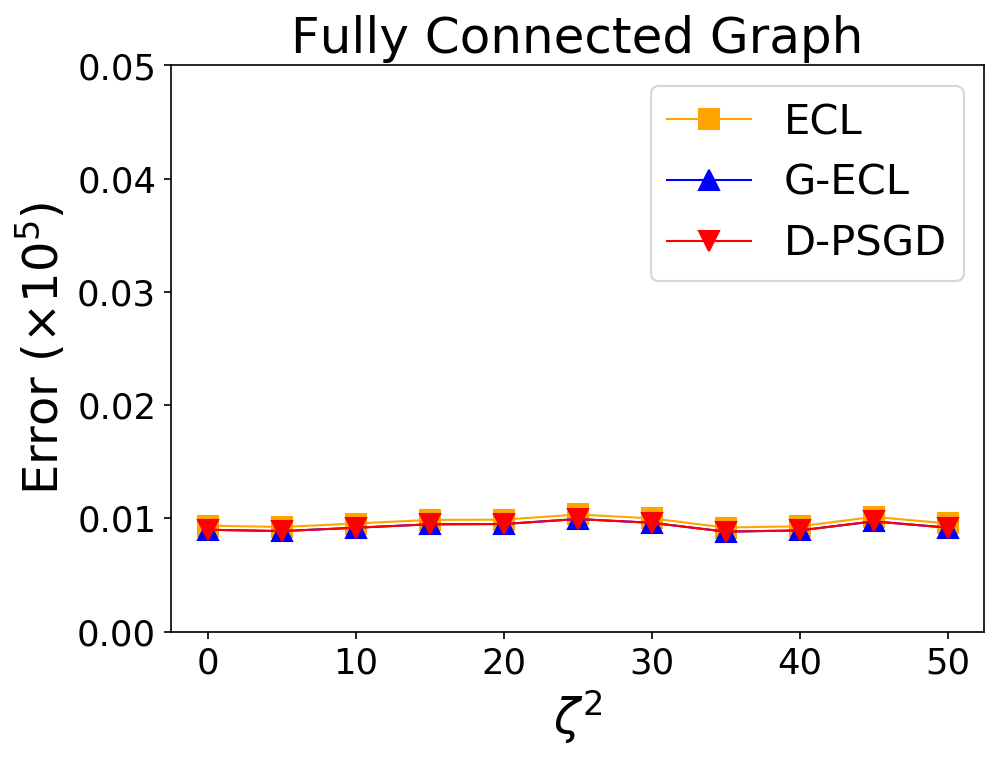}
  \end{minipage}
\vskip -0.1 in
\caption{Error $\frac{1}{n} \sum_{i=1}^{n} \| \mathbf{x}_i^{(r)} - \mathbf{x}^\star \|^2$ after $10^4$ rounds when $\zeta^2$ is varied. We evaluate the D-PSGD, ECL, and G-ECL when $G$ is ring, torus, or fully connected graph with $\sigma^2=0$.}
\vskip -0.1 in
\label{fig:hetero}
\end{figure}

\textbf{Effect of Heterogeneity of Data Distributions ($\sigma^2=0$):}
We first discuss the effect of $\zeta$ on the convergence rate when $\sigma^2=0$.
Fig. \ref{fig:hetero} shows the error $\frac{1}{n} \sum_{i=1}^{n} \| \mathbf{x}_i^{(r)} - \mathbf{x}^\star \|^2$ after $10^4$ rounds when varying $\zeta$ and setting $\sigma^2=0$.
The results show that when $G$ is a ring or torus (i.e., $p<1$), 
the error of the D-PSGD increases linearly with respect to $\zeta^2$,
and when $G$ is a fully connected graph (i.e., $p=1$),
the error of the D-PSGD is almost the same even if $\zeta^2$ is increased.
In contrast, the errors of the ECL and G-ECL are almost the same, even if $\zeta^2$ is increased for all network topologies. Therefore, the numerical results are consistent with the theoretical results in Theorem \ref{theoram:convergence_rate}.

\begin{figure}[!t]
  \begin{minipage}[b]{0.33\columnwidth}
    \centering
    \includegraphics[width=\columnwidth]{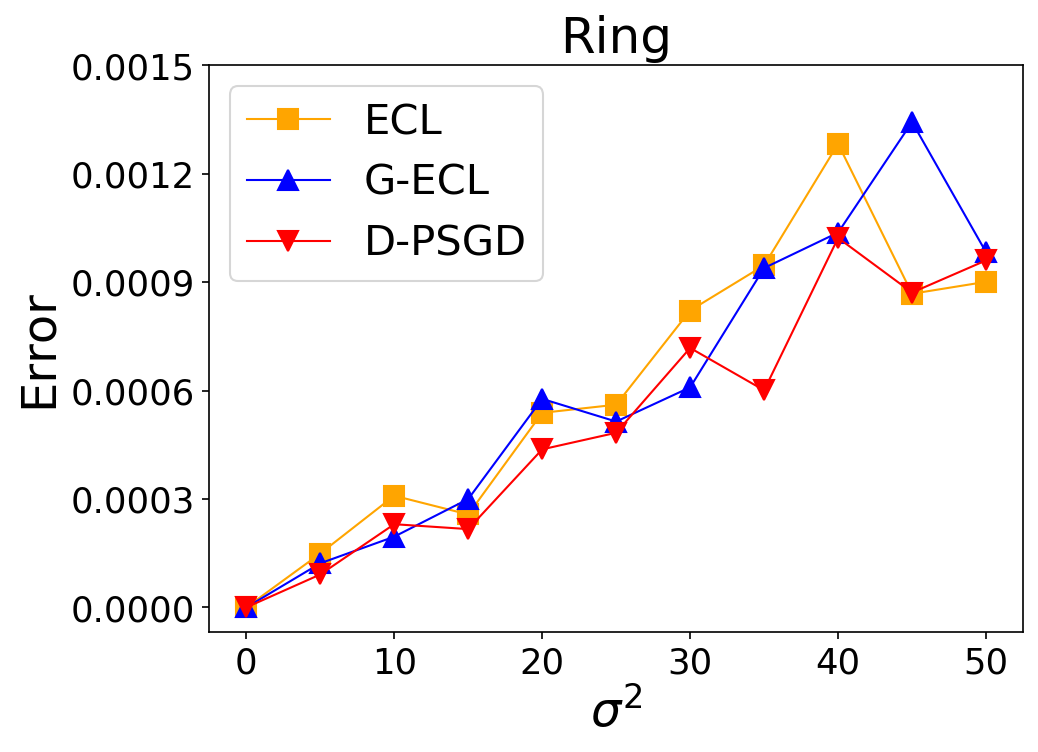}
  \end{minipage}
  \begin{minipage}[b]{0.33\columnwidth}
    \centering
    \includegraphics[width=\columnwidth]{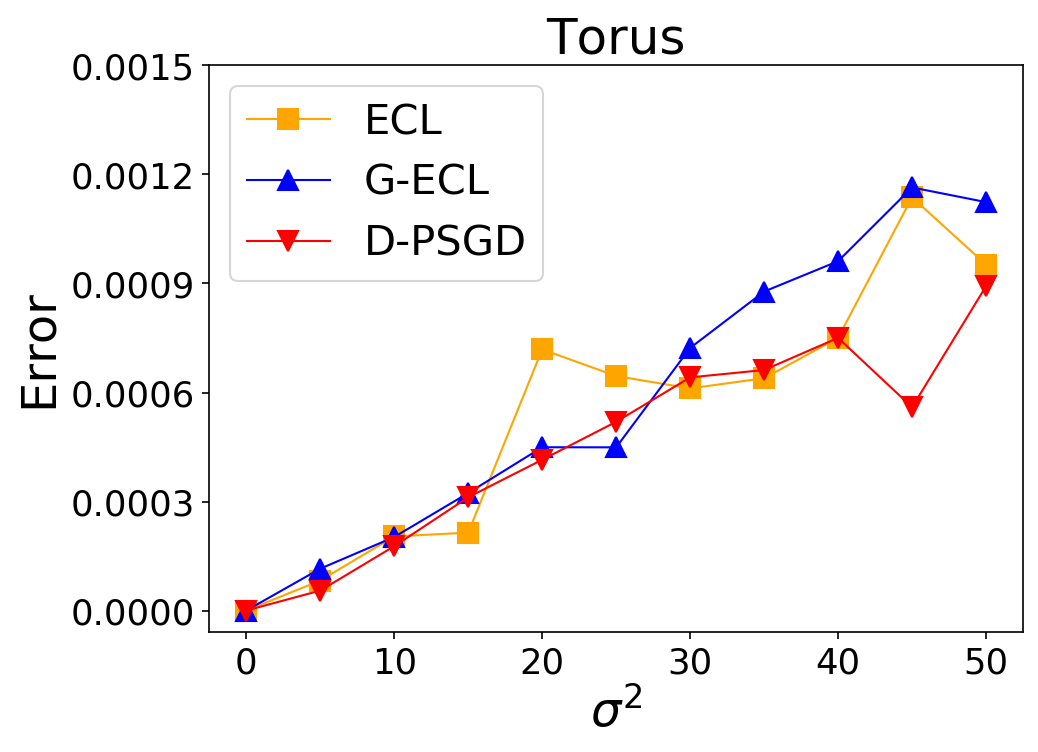}
  \end{minipage}
  \begin{minipage}[b]{0.33\columnwidth}
    \centering
    \includegraphics[width=\columnwidth]{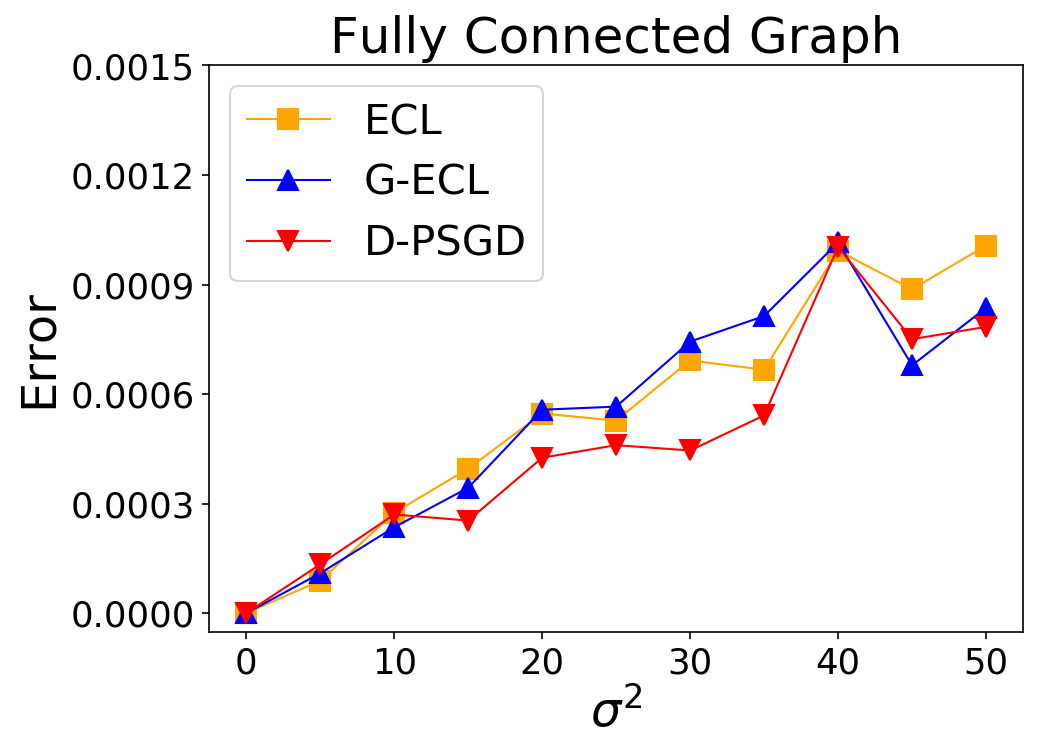}
  \end{minipage}
\vskip -0.1 in
\caption{Error $\frac{1}{n} \sum_{i=1}^{n} \| \mathbf{x}_i^{(r)} - \mathbf{x}^\star \|^2$ after $10^4$ rounds when $\sigma^2$ is varied. We evaluate the D-PSGD, ECL, and G-ECL when $G$ is ring, torus, or fully connected graph with $\zeta^2=0$.}
\vskip -0.1 in
\label{fig:stochastic_gradient}
\end{figure}

\textbf{Effect of Noise of Stochastic Gradient ($\zeta^2=0$):}
Next, we discuss the effect of $\sigma$ on the convergence rate when $\zeta^2=0$.
Fig. \ref{fig:stochastic_gradient} shows the error $\frac{1}{n} \sum_{i=1}^n \| \mathbf{x}_i^{(r)} - \mathbf{x}^\star \|^2$ after $10^4$ rounds when varying $\sigma$ and setting $\zeta^2=0$.
The results show that the errors of all comparison methods increase linearly with respect to $\sigma^2$ for all network topologies.
The theoretical results shown in Table \ref{table:convergence_rate} indicate that the convergence rates of the D-PSGD and G-ECL are $\mathcal{O}(\sigma^2)$.
Thus, the theoretical results are consistent with the numerical results.
Moreover, Fig. \ref{fig:stochastic_gradient} shows that in all comparison methods, 
the effect of $\sigma$ on the convergence is almost the same for all network topologies.
In the convergence rate of both the D-PSGD and G-ECL,
the second term $\mathcal{O}(\frac{\sigma^2}{\mu n R})$, which does not depend on the network topology,
is more dominant than the third term when the number of round $R$ is sufficiently large.
Therefore, the numerical results are consistent with our theoretical results.

\begin{figure}[!t]
  \begin{minipage}[b]{0.33\columnwidth}
    \centering
    \includegraphics[width=\columnwidth]{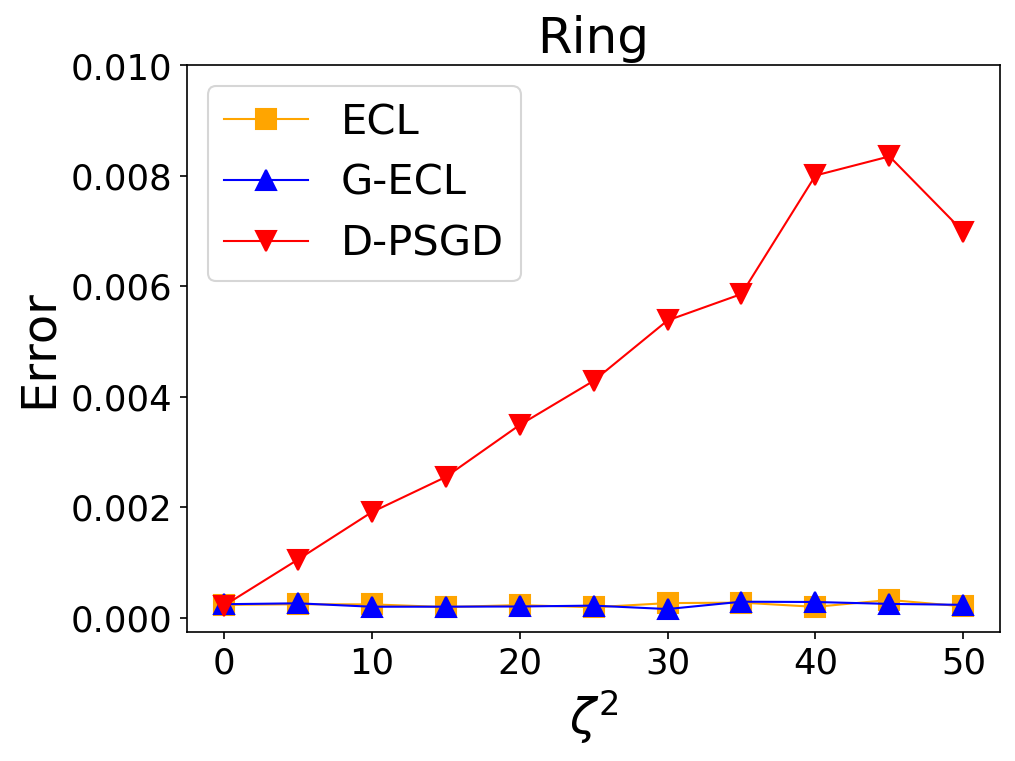}
  \end{minipage}
  \begin{minipage}[b]{0.33\columnwidth}
    \centering
    \includegraphics[width=\columnwidth]{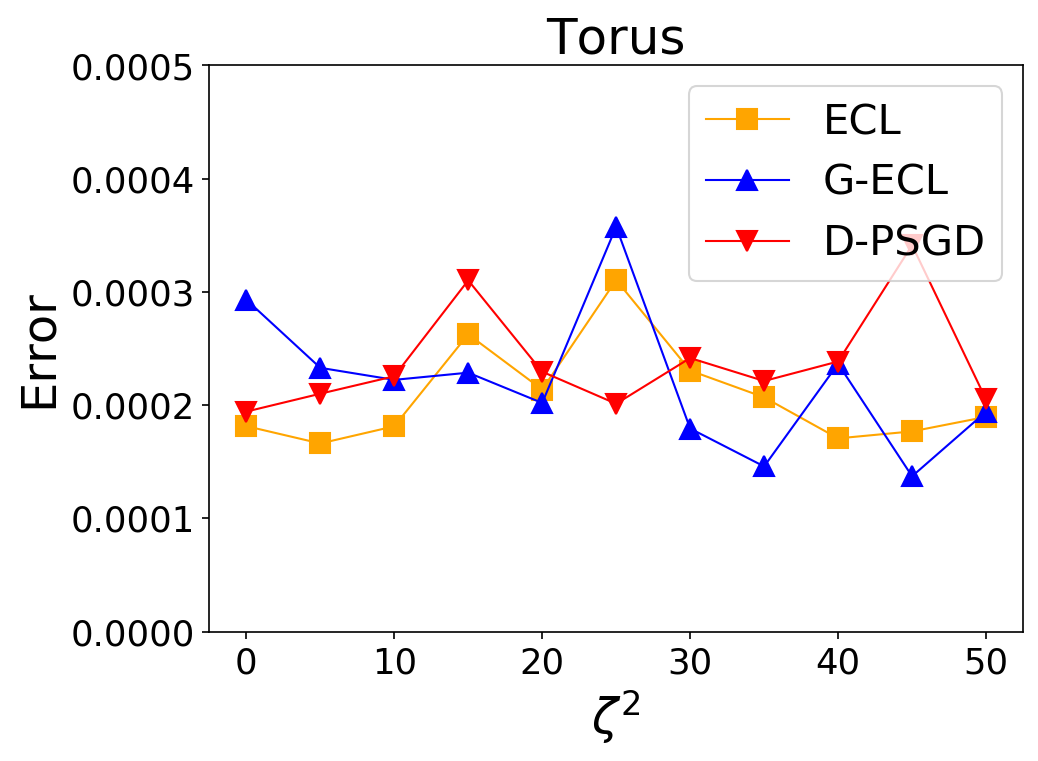}
  \end{minipage}
  \begin{minipage}[b]{0.33\columnwidth}
    \centering
    \includegraphics[width=\columnwidth]{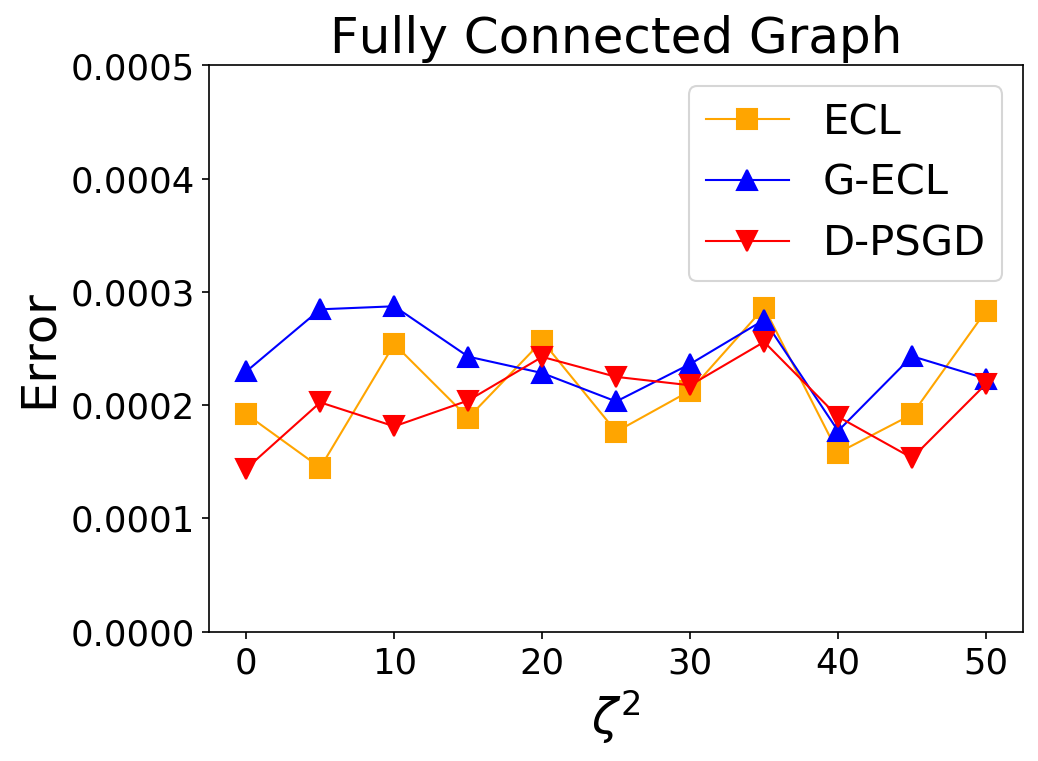}
  \end{minipage}
\vskip -0.1 in
\caption{Error $\frac{1}{n} \sum_{i=1}^{n} \| \mathbf{x}_i^{(r)} - \mathbf{x}^\star \|^2$ after $10^4$ rounds when $\zeta^2$ is varied. We evaluate the D-PSGD, ECL, and G-ECL when $G$ is ring, torus, or fully connected graph with $\sigma^2=10$.}
\vskip -0.1 in
\label{fig:hetero_sigma_10}
\end{figure}

\textbf{Effect of Heterogeneity of Data Distributions ($\sigma^2=10$):}
Next, we discuss the effect of $\zeta$ on the convergence rate when $\sigma^2=10$.
Fig. \ref{fig:hetero_sigma_10} shows the error $\frac{1}{n} \sum_{i=1}^{n} \| \mathbf{x}_i^{(r)} - \mathbf{x}^\star \|^2$ after $10^4$ rounds when varying $\zeta$ and setting $\sigma^2=10$.
The results show that when $G$ is a ring, 
the error of the D-PSGD increases linearly with respect to $\zeta^2$.
When $G$ is a torus or fully connected graph,
the error of the D-PSGD is almost the same even if $\zeta^2$ is increased.
This is because the effect of $\sigma^2$ is more dominant than the one of $\zeta^2$.
Figs. \ref{fig:hetero} and \ref{fig:stochastic_gradient} show that when $G$ is a torus, the error of the D-PSGD is approximately $1.0 \times 10^{-5}$ when $\zeta^2>0$ and $\sigma^2=0$
and is approximately $1.0 \times 10^{-4}$ when $\zeta^2=0$ and $\sigma^2>0$.
Therefore, Fig. \ref{fig:hetero_sigma_10} indicates that when $G$ is a torus or fully connected graph,
the error of the D-PSGD is almost the same even if $\zeta^2$ is increased.
In contrast, the errors of the ECL and G-ECL are almost the same, even if $\zeta^2$ is increased for all network topologies. Therefore, the numerical results are consistent with the theoretical results in Theorem \ref{theoram:convergence_rate}.

\begin{figure}[!t]
  \begin{minipage}[b]{0.33\columnwidth}
    \centering
    \includegraphics[width=\columnwidth]{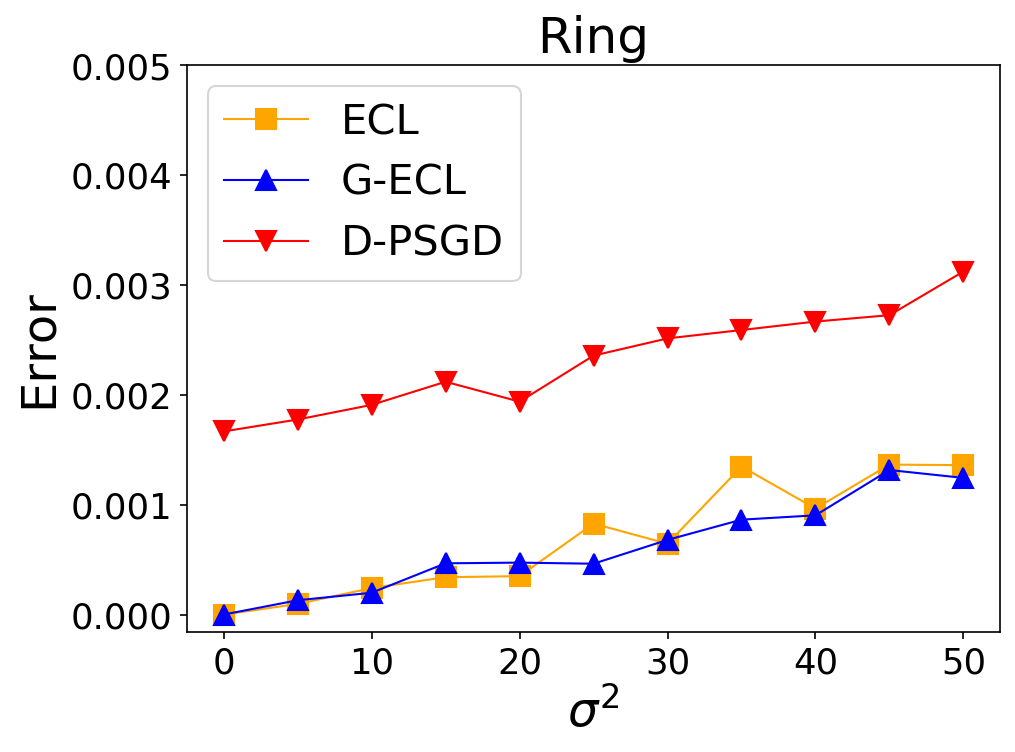}
  \end{minipage}
  \begin{minipage}[b]{0.33\columnwidth}
    \centering
    \includegraphics[width=\columnwidth]{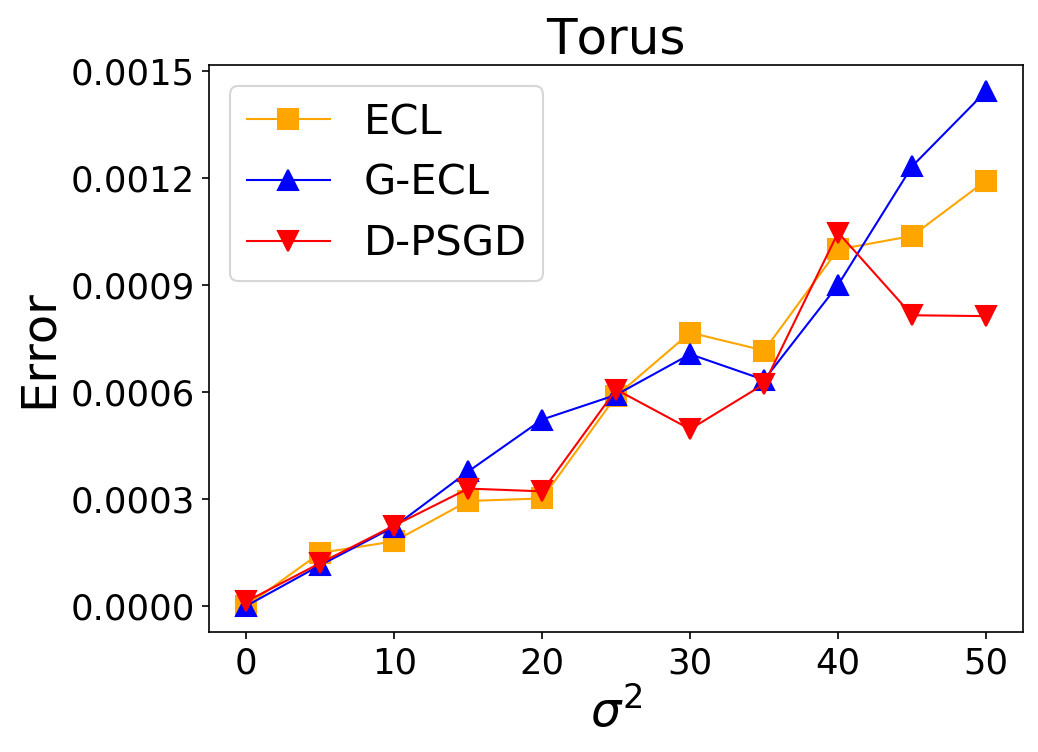}
  \end{minipage}
  \begin{minipage}[b]{0.33\columnwidth}
    \centering
    \includegraphics[width=\columnwidth]{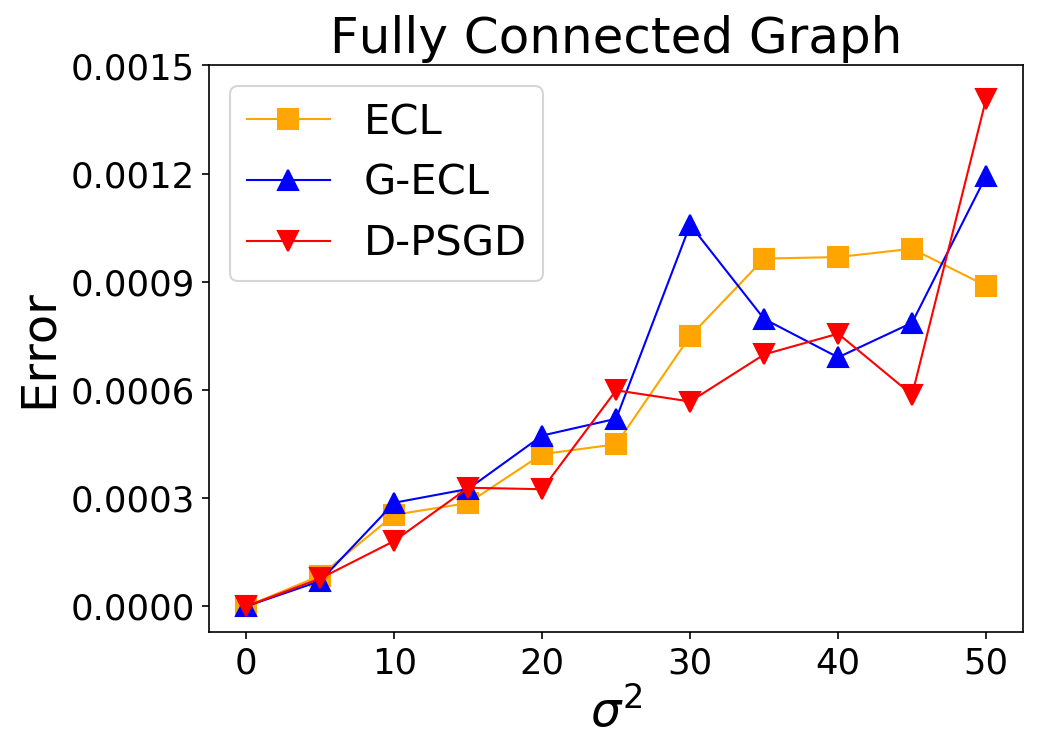}
  \end{minipage}
\vskip -0.1 in
\caption{Error $\frac{1}{n} \sum_{i=1}^{n} \| \mathbf{x}_i^{(r)} - \mathbf{x}^\star \|^2$ after $10^4$ rounds when $\sigma^2$ is varied. We evaluate the D-PSGD, ECL, and G-ECL when $G$ is ring, torus, or fully connected graph with $\zeta^2=10$.}
\vskip -0.1 in
\label{fig:stochastic_gradient_zeta_10}
\end{figure}

\textbf{Effect of Noise of Stochastic Gradient ($\zeta^2=10$):}
Next, we discuss the effect of $\sigma$ on the convergence rate when $\zeta^2=10$.
Fig. \ref{fig:stochastic_gradient_zeta_10} shows the error $\frac{1}{n} \sum_{i=1}^n \| \mathbf{x}_i^{(r)} - \mathbf{x}^\star \|^2$ after $10^4$ rounds when varying $\sigma$ and setting $\zeta^2=10$.
The results show that the errors of all comparison methods increase linearly with respect to $\sigma^2$ for all network topologies.
When $G$ is a ring, the error of the D-PSGD is consistently larger than those of the G-ECL and ECL.
This is because the error of the D-PSGD is larger than those of the G-ECL and ECL when $\zeta^2>0$, as Figs. \ref{fig:hetero} and \ref{fig:hetero_sigma_10} indicate.
When $G$ is a torus or fully connected graph, the errors of all comparison methods are almost the same.
This is because the effect of $\sigma^2$ is more dominant than that of $\zeta^2$, as Fig. \ref{fig:hetero_sigma_10} indicates.
Therefore, the numerical results are consistent with convergence rates of the G-ECL.

\newpage 
\section{Proof of Theorem \ref{theorem:reformulation2}}
\begin{lemma}
\label{lemma:reformulation}
Suppose that the hyperparameter $\theta=\frac{1}{2}$,
the dual variable $\mathbf{z}_{i|j}^{(0)}$ is initialized to $\mathbf{A}_{i|j} \mathbf{x}_j^{(0)}$,
and the hyperparameter $\{ \alpha_{i|j} \}_{ij}$ is set such that $\alpha_{i|j} = \alpha_{j|i} \geq 0$
for all $(i,j) \in \mathcal{E}$.
Then, the update formulas Eq. (\ref{eq:stochastic_ecl}) and Eqs. (\ref{eq:ecl_2}-\ref{eq:ecl_3}) are equivalent to the following:
\begin{align}
    \mathbf{x}_i^{(r+1)} &= \!\!\! \sum_{j\in\mathcal{N}_i^{+}} \!\!\! W_{ij} \mathbf{x}_j^{(r)} 
    - \frac{\eta}{1 + \eta \sum_{j\in\mathcal{N}_i} \alpha_{i|j}} \left( \nabla F_i(\mathbf{x}_i^{(r)} ; \xi_i^{(r)}) - \mathbf{c}_i^{(r)} \right), \\
    \mathbf{c}_{i}^{(r+1)} &= \mathbf{c}_i^{(r)} + \frac{1}{2} \sum_{j\in\mathcal{N}_i} \alpha_{i|j} (\mathbf{x}_j^{(r+1)} - \mathbf{x}_i^{(r+1)} ),
\end{align}
where $\mathbf{c}_i^{(0)} \coloneqq \frac{1}{2} \sum_{j\in\mathcal{N}_i} \alpha_{i|j} (\mathbf{x}_j^{(0)} - \mathbf{x}_i^{(0)})$, and $W_{ij}$ is defined as follows:
\begin{align}
    W_{ij} \coloneqq 
    \begin{dcases}
    \frac{2 + \eta \sum_{k\in\mathcal{N}_i} \alpha_{i|k}}{2 ( 1 + \eta \sum_{k\in\mathcal{N}_i} \alpha_{i|k})} & \text{if} \;\;  i=j \\
    \frac{\eta \alpha_{i|j}}{2 ( 1 + \eta \sum_{k\in\mathcal{N}_i} \alpha_{i|k})} & \text{if} \;\;  (i,j) \in \mathcal{E} \\
    0 & \text{otherwise}
    \end{dcases}.
\end{align}
\end{lemma}
\begin{proof}
The update formulas of the ECL can be written as follows:
\begin{align*}
    \mathbf{x}_i^{(r+1)} &= \frac{1}{1 + \eta \sum_{j\in\mathcal{N}_i} \alpha_{i|j}} \left( \mathbf{x}_i^{(r)} - \eta \nabla F_i(\mathbf{x}_i^{(r)} ; \xi_i^{(r)}) + \eta \sum_{j\in\mathcal{N}_i} \alpha_{i|j} \mathbf{A}_{i|j} \mathbf{z}_{i|j}^{(r)} \right), \\
    \mathbf{z}_{i|j}^{(r+1)} &= \frac{1}{2} (\mathbf{z}_{i|j}^{(r)} + \mathbf{z}_{j|i}^{(r)} ) - \mathbf{A}_{j|i} \mathbf{x}_j^{(r+1)}.
\end{align*}
Defining $\mathbf{u}_{i|j}^{(r)} \coloneqq \mathbf{z}_{i|j}^{(r)} + \mathbf{A}_{j|i} \mathbf{x}_{j}^{(r)}$,
the update formulas of the ECL can be rewritten as follows:
\begin{align*}
    \mathbf{x}_i^{(r+1)} &= \frac{1}{1 + \eta \sum_{j\in\mathcal{N}_i} \alpha_{i|j}} \left( \mathbf{x}_i^{(r)} + \eta \left( \sum_{j\in\mathcal{N}_i} \alpha_{i|j} \mathbf{x}_j^{(r)} \right) - \eta \nabla F_i(\mathbf{x}_i^{(r)} ; \xi_i^{(r)}) + \eta \sum_{j\in\mathcal{N}_i} \alpha_{i|j} \mathbf{A}_{i|j} \mathbf{u}_{i|j}^{(r)} \right), \\
    \mathbf{u}_{i|j}^{(r+1)} &= \frac{1}{2} (\mathbf{u}_{i|j}^{(r)} - \mathbf{A}_{j|i} \mathbf{x}_j^{(r)} ) + \frac{1}{2} (\mathbf{u}_{j|i}^{(r)} - \mathbf{A}_{i|j} \mathbf{x}_i^{(r)} ).
\end{align*}

From the above update formula for $\mathbf{u}_{i|j}$,
the update formulas for $\mathbf{u}_{i|j}$ and $\mathbf{u}_{j|i}$ are equivalent.
Then, it holds that for any round $r>0$,
\begin{align}
\label{eq:reformulation:u}
    \mathbf{u}_{i|j}^{(r)}=\mathbf{u}_{j|i}^{(r)}.
\end{align}
Moreover, because $\mathbf{u}_{i|j}^{(0)} = \mathbf{0}$, Eq. \eqref{eq:reformulation:u} holds for any round $r \geq 0$.

We define $\mathbf{b}_i^{(r)} \coloneqq \sum_{j\in\mathcal{N}_i} \alpha_{i|j} \mathbf{A}_{i|j} \mathbf{u}_{i|j}^{(r)}$.
From Eq. \eqref{eq:reformulation:u}, $\mathbf{b}_i^{(r)} = \sum_{j\in\mathcal{N}_i} \alpha_{i|j} \mathbf{A}_{i|j} \mathbf{u}_{j|i}^{(r)}$ holds for any round $r$.
The update formulas of the ECL are then rewritten as follows:
\begin{align*}
    \mathbf{x}_i^{(r+1)} &= \frac{1}{1 + \eta \sum_{j\in\mathcal{N}_i} \alpha_{i|j}} \left( \mathbf{x}_i^{(r)} + \eta \left( \sum_{j\in\mathcal{N}_i} \alpha_{i|j} \mathbf{x}_j^{(r)} \right) - \eta \nabla F_i(\mathbf{x}_i^{(r)} ; \xi_i^{(r)}) + \eta \mathbf{b}_i^{(r)} \right), \\
    \mathbf{b}_{i}^{(r+1)} &= \mathbf{b}_i^{(r)} + \frac{1}{2} \sum_{j\in\mathcal{N}_i} \alpha_{i|j} (\mathbf{x}_j^{(r)} - \mathbf{x}_i^{(r)} ).
\end{align*}
Defining $\mathbf{c}_i^{(r)} \coloneqq \mathbf{b}_i^{(r)} + \frac{1}{2} \sum_{j\in\mathcal{N}_i} \alpha_{i|j}(\mathbf{x}_j^{(r)} - \mathbf{x}_i^{(r)})$,
the update formulas of the ECL are rewritten as follows:
\begin{align*}
    \mathbf{x}_i^{(r+1)} &= \frac{1}{2 ( 1 + \eta \sum_{j\in\mathcal{N}_i} \alpha_{i|j})} \left( \left( 2 + \eta \sum_{j\in\mathcal{N}_i} \alpha_{i|j}  \right) \mathbf{x}_i^{(r)} + \eta \sum_{j\in\mathcal{N}_i} \alpha_{i|j} \mathbf{x}_j^{(r)} \right) \\
    &\qquad - \frac{\eta}{1 + \eta \sum_{j\in\mathcal{N}_i} \alpha_{i|j}} \left( \nabla F_i(\mathbf{x}_i^{(r)} ; \xi_i^{(r)}) - \mathbf{c}_i^{(r)} \right), \\
    \mathbf{c}_{i}^{(r+1)} &= \mathbf{c}_i^{(r)} + \frac{1}{2} \sum_{j\in\mathcal{N}_i} \alpha_{i|j} (\mathbf{x}_j^{(r+1)} - \mathbf{x}_i^{(r+1)} ).
\end{align*}
This concludes the proof.
\end{proof}
\subsection{Proof of Theorem \ref{theorem:reformulation2}}
\begin{proof}
We define $\tilde{\mathbf{x}}_i^{(r)} \in \mathbb{R}^d$ as follows:
\begin{align*}
    \tilde{\mathbf{x}}_i^{(r)} \coloneqq \sum_{j \in \mathcal{N}_i^+} W_{ij} \mathbf{x}_j^{(r)}.
\end{align*}
Then, from Lemma \ref{lemma:reformulation}, we have
\begin{align*}
    \mathbf{c}_{i}^{(r+1)} 
    &= \mathbf{c}_i^{(r)} + \frac{1}{2} \sum_{j\in\mathcal{N}_i} \alpha_{i|j} (\mathbf{x}_j^{(r+1)} - \mathbf{x}_i^{(r+1)} ) \\
    &= \mathbf{c}_i^{(r)} + \frac{1}{2} \sum_{j\in\mathcal{N}_i} \alpha_{i|j} \mathbf{x}_j^{(r+1)} - \frac{1}{2} \left( \sum_{j\in\mathcal{N}_i} \alpha_{i|j} \right) \mathbf{x}_i^{(r+1)} \\
    &= \mathbf{c}_i^{(r)} + \frac{1}{2} \sum_{j\in\mathcal{N}_i} \alpha_{i|j} \left( \tilde{\mathbf{x}}_j^{(r)} - \eta^\prime \left( \nabla F_j(\mathbf{x}_j^{(r)} ; \xi_j^{(r)}) - \mathbf{c}_j^{(r)} \right) \right) \\
    &\qquad - \frac{1}{2} \left( \sum_{j\in\mathcal{N}_i} \alpha_{i|j} \right) \left( \tilde{\mathbf{x}}_i^{(r)} - \eta^\prime \left( \nabla F_i(\mathbf{x}_i^{(r)} ; \xi_i^{(r)}) - \mathbf{c}_i^{(r)} \right) \right) \\
    &= \left( (1 - \frac{\eta^\prime \sum_{j\in\mathcal{N}_i} \alpha_{i|j}}{2})\mathbf{c}_i^{(r)} + \frac{1}{2} \sum_{j \in \mathcal{N}_i} \alpha_{i|j} \eta^\prime \mathbf{c}_j^{(r)} \right) \\
    &\qquad - \left( ( 1 - \frac{\eta^\prime \sum_{j\in\mathcal{N}_i} \alpha_{i|j}}{2}) \nabla F_i(\mathbf{x}_i^{(r)} ; \xi_i^{(r)}) + \frac{1}{2} \sum_{j\in \mathcal{N}_i} \alpha_{i|j} \eta^\prime \nabla F_j(\mathbf{x}_j^{(r)} ; \xi_j^{(r)}) \right) \\
    &\qquad+ \frac{1}{2} \left( \sum_{j \in \mathcal{N}_i} \alpha_{i|j} ( \tilde{\mathbf{x}}_j^{(r)} - \tilde{\mathbf{x}}_i^{(r)} ) \right) + \nabla F_i(\mathbf{x}_i^{(r)} ; \xi_i^{(r)}).
\end{align*}
Using Eq. \eqref{eq:definition_of_eta_prime}, we get
\begin{align*}
    \mathbf{c}_i^{(r+1)} 
    &= \sum_{j \in \mathcal{N}_i^+} W_{ij} \left( \mathbf{c}_j^{(r)} - \nabla F_j(\mathbf{x}_j^{(r)} ; \xi_j^{(r)}) \right) + \frac{1}{2} \left( \sum_{j \in \mathcal{N}_i} \alpha_{i|j} ( \tilde{\mathbf{x}}_j^{(r)} - \tilde{\mathbf{x}}_i^{(r)} ) \right) + \nabla F_i(\mathbf{x}_i^{(r)} ; \xi_i^{(r)}).
\end{align*}
This concludes the proof.
\end{proof}

\section{Proof of Theorem \ref{theorem:mixing_matrix}}
\begin{proof}
When $(i,j)\in\mathcal{E}$, we have
\begin{align*}
    W_{ij} 
    = \frac{\eta \alpha_{i|j}}{2 ( 1 + \eta \alpha)}
    = \frac{\eta \alpha_{j|i}}{2 ( 1 + \eta \alpha)}
    = W_{ji}.
\end{align*}
Therefore, $\mathbf{W}$ is symmetric.
Next, we prove that $\mathbf{W}$ is doubly stochastic.
We have
\begin{align*}
    \sum_{j=1}^{n} W_{ij} 
    &= \frac{2 + \eta \sum_{k\in\mathcal{N}_i} \alpha_{i|k}}{2 ( 1 + \eta \sum_{k\in\mathcal{N}_i} \alpha_{i|k})}
    + \sum_{j\in\mathcal{N}_i} \frac{\eta \alpha_{i|j}}{2 ( 1 + \eta \sum_{k\in\mathcal{N}_i} \alpha_{i|k})}
    = 1, \\
    \sum_{i=1}^{n} W_{ij} 
    &= \sum_{i=1}^{n} W_{ji} 
    = 1.
\end{align*}
This concludes the proof.
\end{proof}

\section{Proof of Lemma \ref{lemma:c_is_zero}}
\label{sec:proof_of_c_is_zero}
\begin{proof}
From Lemma \ref{lemma:reformulation}, we have
\begin{align*}
    \sum_{i=1}^{n} \mathbf{c}_{i}^{(r+1)} 
    &= \sum_{i=1}^{n} \mathbf{c}_i^{(r)} + \frac{1}{2} \sum_{i=1}^{n} \sum_{j\in\mathcal{N}_i} \alpha_{i|j} \mathbf{x}_j^{(r+1)} - \frac{1}{2} \sum_{i=1}^{n} \sum_{j\in\mathcal{N}_i} \alpha_{i|j} \mathbf{x}_i^{(r+1)}. 
\end{align*}
We define $\mathcal{E}^> \coloneqq \{(i,j)\in\mathcal{E}|i>j\}$ and $\mathcal{E}^< \coloneqq \{(i,j)\in\mathcal{E}|i<j\}$.
Then, we get
\begin{align*}
    \sum_{i=1}^{n} \sum_{j\in\mathcal{N}_i} \alpha_{i|j} \mathbf{x}_j^{(r+1)}
    &= \sum_{(i,j)\in\mathcal{E}^>} \alpha_{i|j} \mathbf{x}_j^{(r+1)} + \sum_{(i,j)\in\mathcal{E}^<} \alpha_{i|j} \mathbf{x}_j^{(r+1)} \\
    &= \sum_{(j,i)\in\mathcal{E}^<} \alpha_{i|j} \mathbf{x}_j^{(r+1)} + \sum_{(j,i)\in\mathcal{E}^>} \alpha_{i|j} \mathbf{x}_j^{(r+1)} \\
    &= \sum_{j=1}^{n} \sum_{i\in\mathcal{N}_j} \alpha_{i|j} \mathbf{x}_j^{(r+1)} \\
    &=\sum_{i=1}^{n} \sum_{j\in\mathcal{N}_i} \alpha_{j|i} \mathbf{x}_i^{(r+1)} \\
    &= \sum_{i=1}^{n} \sum_{j\in\mathcal{N}_i} \alpha_{i|j} \mathbf{x}_i^{(r+1)},
\end{align*}
where we use $\alpha_{i|j} = \alpha_{j|i}$ in the last equation.
Then, we get
\begin{align*}
    \sum_{i=1}^{n} \mathbf{c}_{i}^{(r+1)} 
    &= \sum_{i=1}^{n} \mathbf{c}_i^{(r)}.
\end{align*}
From the initial value of $\mathbf{c}_i$, we have $\sum_{i=1}^n \mathbf{c}_i^{(0)} = \mathbf{0}$.
Therefore, $\sum_{i=1}^{n} \mathbf{c}_{i}^{(r)}=\mathbf{0}$ for any round $r$.

From Theorems \ref{theorem:reformulation2} and \ref{theorem:mixing_matrix}, we have
\begin{align*}
    \frac{1}{n} \sum_{i=1}^{n} \mathbf{x}_i^{(r+1)} 
    &= \frac{1}{n} \sum_{i=1}^{n} \sum_{j\in\mathcal{N}_i^+}  W_{ij} \mathbf{x}_j^{(r)} 
    - \eta^\prime \frac{1}{n} \sum_{i=1}^{n} \left( \nabla F_i(\mathbf{x}_i^{(r)} ; \xi_i^{(r)}) - \mathbf{c}_i^{(r)} \right) \\
    &= \frac{1}{n} \sum_{i=1}^{n} \sum_{j=1}^{n}  W_{ij} \mathbf{x}_j^{(r)} 
    - \eta^\prime \frac{1}{n} \sum_{i=1}^{n} \nabla F_i(\mathbf{x}_i^{(r)} ; \xi_i^{(r)}) + \eta^\prime \frac{1}{n} \sum_{i=1}^{n}  \mathbf{c}_i^{(r)} \\
    &= \frac{1}{n} \sum_{j=1}^{n} \mathbf{x}_j^{(r)} \sum_{i=1}^{n}  W_{ij}  
    - \eta^\prime \frac{1}{n} \sum_{i=1}^{n} \nabla F_i(\mathbf{x}_i^{(r)} ; \xi_i^{(r)}) \\
    &= \frac{1}{n} \sum_{j=1}^{n} \mathbf{x}_j^{(r)} 
    - \eta^\prime \frac{1}{n} \sum_{i=1}^{n} \nabla F_i(\mathbf{x}_i^{(r)} ; \xi_i^{(r)}).
\end{align*}
This concludes the proof.
\end{proof}

\newpage
\section{Update Procedure of Generalized Edge-Consensus Learning}
\label{sec:algorithm_gecl}
\begin{algorithm}[h]
   \caption{Update procedure at node $i$ in the G-ECL.}
   \label{alg:reformulated_2_stochastic_ecl}
\begin{algorithmic}[1]
   \STATE {\bfseries Input:}  Set the step size $\eta^\prime>0$, the mixing matrix $\mathbf{W} \in [0, 1]^{n\times n}$, and $\{ \alpha_{i|j} \}_{ij}$ that satisfies $\alpha_{i|j}=\alpha_{j|i} \geq 0$ for all $(i,j)\in\mathcal{E}$ and $i\in[n]$. Initialize $\mathbf{x}_i^{(0)}$ with the same parameter for all $i \in [n]$. Note that $\eta^\prime$, $\mathbf{W}$, and $\{ \alpha_{i|j} \}_{ij}$ can be set independently as hyperparameters in the G-ECL.
   \FOR{$r = 0, 1, \ldots, R$}
   \FOR{$j \in \mathcal{N}_i$}
   \STATE $\textbf{Transmit}_{i\rightarrow j}(\mathbf{x}_{i}^{(r)})$.
   \STATE $\textbf{Receive}_{i\leftarrow j}(\mathbf{x}_{j}^{(r)})$.
   \ENDFOR
   \STATE $\tilde{\mathbf{x}}_i^{(r)} \leftarrow \sum_{j\in\mathcal{N}_i^+} W_{ij} \mathbf{x}_j^{(r)}$
   \STATE Sample $\xi_i^{(r)}$ and compute $\mathbf{g}_i^{(r)} \coloneqq \nabla F_i(\mathbf{x}_i^{(r)} ; \xi_i^{(r)})$.
   \STATE $\mathbf{x}_i^{(r+1)} \leftarrow \tilde{\mathbf{x}}_i^{(r)} - \eta^\prime \left( \mathbf{g}_i^{(r)} - \mathbf{c}_i^{(r)} \right)$.
   \FOR{$j \in \mathcal{N}_i$}
   \STATE $\textbf{Transmit}_{i\rightarrow j}(\tilde{\mathbf{x}}_i^{(r)}, \mathbf{c}_i^{(r)} - \mathbf{g}_{i}^{(r)})$.
   \STATE $\textbf{Receive}_{i\leftarrow j}(\tilde{\mathbf{x}}_j^{(r)}, \mathbf{c}_j^{(r)} - \mathbf{g}_{j}^{(r)})$.
   \ENDFOR
   \STATE $\mathbf{c}_i^{(r+1)} \leftarrow \sum_{j \in \mathcal{N}_i^+} W_{ij} \left( \mathbf{c}_j^{(r)} - \mathbf{g}_j^{(r)} \right) + \mathbf{g}_i^{(r)} + \sum_{j \in \mathcal{N}_i} \frac{\alpha_{i|j}}{2} ( \tilde{\mathbf{x}}_j^{(r)} - \tilde{\mathbf{x}}_i^{(r)} )$.
   \ENDFOR
\end{algorithmic}
\end{algorithm}
Alg. \ref{alg:reformulated_2_stochastic_ecl} shows the pseudo-code of the G-ECL.
Note that from Theorem \ref{theorem:reformulation2}, 
the update formulas of Alg. \ref{alg:stochastic_ecl} and Alg. \ref{alg:reformulated_2_stochastic_ecl} are equivalent when the hyperparameter $\theta$ is $\frac{1}{2}$, and $\eta^\prime$, $\mathbf{W}$, and $\{ \alpha_{i|j} \}_{ij}$ are set as in Eq. \eqref{eq:definition_of_eta_prime}.

\newpage
\section{Proof of Theorem \ref{theoram:convergence_rate}}
\label{sec:proof_of_main}

\subsection{G-ECL in Matrix Notation}
We define $\mathbf{X}^{(r)}, \mathbf{C}^{(r)}$, $\bar{\mathbf{X}}^{(r)} \in \mathbb{R}^{d\times n}$, $\nabla F(\mathbf{X}^{(r)} ; \xi^{(r)})$,
and $\nabla f(\mathbf{X}^{(r)})$ as follows:
\begin{gather*}
    \mathbf{X}^{(r)} \coloneqq\left[ \mathbf{x}_1^{(r)}, \cdots, \mathbf{x}_n^{(r)} \right], \;\;
    \mathbf{C}^{(r)} \coloneqq\left[ \mathbf{c}_1^{(r)}, \cdots, \mathbf{c}_n^{(r)} \right], \;\;
    \bar{\mathbf{X}}^{(r)} \coloneqq \left[ \bar{\mathbf{x}}^{(r)}, \cdots, \bar{\mathbf{x}}^{(r)} \right], \\
    \nabla F(\mathbf{X}^{(r)} ; \xi^{(r)}) \coloneqq \left[ \nabla F_1(\mathbf{x}_1^{(r)} ; \xi_1^{(r)}), \cdots, \nabla F_n(\mathbf{x}_n^{(r)} ; \xi_n^{(r)}) \right],  \\
    \nabla f(\mathbf{X}^{(r)}) \coloneqq \left[ \nabla f_1(\mathbf{x}_1^{(r)}), \cdots, \nabla f_n(\mathbf{x}_n^{(r)}) \right],
\end{gather*}
where $\bar{\mathbf{x}}^{(r)} \coloneqq \frac{1}{n} \sum_{i=1}^n \mathbf{x}_i^{(r)}$.
We define $\mathbf{E} \coloneqq \text{diag}(\sum_{k\in\mathcal{N}_1} \alpha_{k|1}, \cdots, \sum_{k\in\mathcal{N}_n} \alpha_{k|n}) \in \mathbb{R}^{n\times n}$ and $\mathbf{D} \in \mathbb{R}^{n\times n}$ whose $(i,j)$-element is
\begin{align}
\label{eq:definition_of_D}
    D_{ij} =
    \begin{dcases}
    \alpha_{i|j} & \text{if} \;\; (i,j) \in \mathcal{E} \\
    0 & \text{otherwise}
    \end{dcases}.
\end{align}
Note that because we assume that $\alpha_{i|j}=\alpha_{j|i}$ for all $(i,j)\in\mathcal{E}$,
$\mathbf{D}$ is symmetric.
Then, the update formulas Eqs. (\ref{eq:reformulated_2_stochastic_ecl_1}-\ref{eq:reformulated_2_stochastic_ecl_3}) can be rewritten as follows:
\begin{align}
    \label{eq:matrix_stochastic_ecl_1}
    \mathbf{X}^{(r+1)} &= \mathbf{X}^{(r)} \mathbf{W} - \eta^\prime (\nabla F(\mathbf{X}^{(r)} ; \xi^{(r)}) - \mathbf{C}^{(r)}), \\
    \label{eq:matrix_stochastic_ecl_2}
    \mathbf{C}^{(r+1)} &= (\mathbf{C}^{(r)} - \nabla F(\mathbf{X}^{(r)} ; \xi^{(r)})) \mathbf{W} + \frac{1}{2} \mathbf{X}^{(r)} \mathbf{W} (\mathbf{D} - \mathbf{E}) + \nabla F(\mathbf{X}^{(r)} ; \xi^{(r)}),
\end{align}

\subsection{Preliminary and Technical Lemma}
\begin{definition}[$\tau$-Slow Increasing \citep{stich2020error}]
The sequence $\{ a_r \}_{r \geq 0}$ of a positive value is called $\tau$-slow increasing if it holds that for any $r\geq 0$,
\begin{align*}
    a_r \leq a_{r+1} \leq \left( 1 + \frac{1}{2 \tau}\right) a_r. 
\end{align*}
\end{definition}

\begin{lemma}
For any $\mathbf{x}, \mathbf{y} \in \mathbb{R}^d, \gamma > 0$, it holds that
\begin{align}
    \label{eq:relaxed_triangle}
    \| \mathbf{x} + \mathbf{y} \|^2 \leq (1 + \gamma) \| \mathbf{x} \|^2 + (1 + \gamma^{-1}) \| \mathbf{y} \|^2.
\end{align}
\begin{lemma}
For any $\mathbf{x}_1, \cdots, \mathbf{x}_n \in \mathbb{R}^d$, it holds that
\begin{align}
    \label{eq:sum_of_n_vectors}
    \left\| \sum_{i=1}^{n} \mathbf{a}_i \right\|^2 \leq n \sum_{i=1}^{n} \| \mathbf{a}_i \|^2.
\end{align}
\end{lemma}
\end{lemma}
\begin{lemma}
For any $\mathbf{x}, \mathbf{y} \in \mathbb{R}^d$ and $\gamma > 0$, it holds that
\begin{align}
\label{eq:inner_prod}
    2 \langle \mathbf{x}, \mathbf{y} \rangle \leq \gamma \| \mathbf{x} \|^2 + \gamma^{-1} \| \mathbf{y} \|^2.
\end{align}
\end{lemma}

\begin{lemma}
Suppose that Assumption \ref{assumption:smoothness} holds
and $f_i$ is convex.
Then, it holds that for any $\mathbf{x}, \mathbf{y} \in \mathbb{R}^d$,
\begin{align}
\label{eq:smooth_convex}
    \| \nabla f_i(\mathbf{x}) - \nabla f_i (\mathbf{y}) \|^2
    \leq 2L (f_i(\mathbf{x}) - f_i(\mathbf{y}) - \langle \mathbf{x} - \mathbf{y}, \nabla f_i(\mathbf{y}) \rangle ).
\end{align}
\end{lemma}

\subsection{Convergence Analysis for Convex Cases}
\label{sec:convergence_analysis_for_convex}
\subsubsection{Additional Notation}
In Sec. \ref{sec:convergence_analysis_for_convex}, we define $\Xi^{(r)}, \mathcal{E}^{(r)}$, $b$, and $b^\prime$ as follows to simplify the notation:
\begin{gather*}
    \Xi^{(r)} \coloneqq \frac{1}{n} \mathbb{E} \sum_{i=1}^{n} \left\| \mathbf{x}_i^{(r)} - \bar{\mathbf{x}}^{(r)} \right\|^2,
    \mathcal{E}^{(r)} \coloneqq \frac{1}{n} \mathbb{E} \left\| \nabla f(\mathbf{X}^\star) - \mathbf{C}^{(r)} \right\|^2_F, \\
    b \coloneqq \left\| \frac{1}{2} (\mathbf{D} - \mathbf{E}) \right\|^2_F, 
    b^\prime \coloneqq \left\| \mathbf{I} - \mathbf{W} \right\|^2_F.
\end{gather*}

\subsubsection{Convergence Analysis}
\begin{lemma}[Descent Lemma for Convex Cases]
\label{lemma:descent_convex}
Suppose that Assumptions \ref{assumption:mixing_matrix}, \ref{assumption:smoothness}, \ref{assumption:stochastic_gradient} and \ref{assumption:mu_convexity} hold.
If ${\eta^\prime} \leq \frac{1}{12 L}$, we have:
\begin{align*}
    &\mathbb{E}_{r+1} \left\| \bar{\mathbf{x}}^{(r+1)} - \mathbf{x}^\star \right\|^2 \\
    &\qquad \leq \left(1 - \frac{\eta^\prime \mu}{2} \right) \left\| \bar{\mathbf{x}}^{(r)} - \mathbf{x}^\star \right\|^2
    + \frac{{\eta^\prime}^2 \sigma^2}{n} - \eta^\prime (f(\bar{\mathbf{x}}^{(r)}) - f (\mathbf{x}^\star)) + \frac{3L \eta^\prime}{n} \sum_{i=1}^{n} \left\| \mathbf{x}_i^{(r)} - \bar{\mathbf{x}}^{(r)} \right\|^2.
\end{align*}
\end{lemma}
\begin{proof}
The statement follows from Lemma 8 in \citep{koloskova2020unified}.
\end{proof}

\begin{lemma}[Recursion for Consensus Distance]
\label{lemma:consensus_convex}
Suppose that Assumptions \ref{assumption:mixing_matrix}, \ref{assumption:smoothness}, \ref{assumption:stochastic_gradient} and \ref{assumption:mu_convexity} hold,
and $\{ \alpha_{i|j} \}_{ij}$ is set such that $\alpha_{i|j} = \alpha_{j|i} \geq 0$ for all $(i,j) \in \mathcal{E}$.
Then, it holds that
\begin{align*}
    \Xi^{(r)} 
    \leq (1 - \frac{p}{2}) \Xi^{(r-1)}
    + \frac{9 L^2 {\eta^\prime}^2}{p} \Xi^{(r-1)}
    + \frac{18 L}{p} {\eta^\prime}^2 (\mathbb{E}f(\bar{\mathbf{x}}^{(r-1)}) - f(\mathbf{x}^\star))
    + \frac{9}{p} {\eta^\prime}^2 \mathcal{E}^{(r-1)}
    + {\eta^\prime}^2 \sigma^2.
\end{align*}
\end{lemma}
\begin{proof}
By using $\sum_{i=1}^{n} \| \mathbf{a}_i - \bar{\mathbf{a}} \|^2 \leq \sum_{i=1}^{n} \| \mathbf{a}_i \|^2$ for any $\mathbf{a}_1, \cdots, \mathbf{a}_n \in \mathbb{R}^d$, we have
\begin{align*}
    n \Xi^{(r)} &= \mathbb{E} \left\| \mathbf{X}^{(r)} - \bar{\mathbf{X}}^{(r)} \right\|^2_F \\
    &= \mathbb{E} \left\| \mathbf{X}^{(r)} - \bar{\mathbf{X}}^{(r-1)} -( \bar{\mathbf{X}}^{(r)} - \bar{\mathbf{X}}^{(r-1)} ) \right\|^2_F
    \leq \mathbb{E} \left\| \mathbf{X}^{(r)} - \bar{\mathbf{X}}^{(r-1)} \right\|^2_F.
\end{align*}
Then, we get
\begin{align*}
    &\mathbb{E}_{r+1} \left\| \mathbf{X}^{(r)} - \bar{\mathbf{X}}^{(r-1)} \right\|^2_F  \\
    &\leq \mathbb{E}_{r+1} \left\| \mathbf{X}^{(r-1)} \mathbf{W} - \eta^\prime (\nabla F(\mathbf{X}^{(r-1)} ; \xi^{(r-1)}) - \mathbf{C}^{(r-1)}) - \bar{\mathbf{X}}^{(r-1)} \right\|^2_F \\
    &\leq \left\| \mathbf{X}^{(r-1)} \mathbf{W} - \eta^\prime (\nabla f(\mathbf{X}^{(r-1)}) - \mathbf{C}^{(r-1)}) - \bar{\mathbf{X}}^{(r-1)} \right\|^2_F \\
    &\qquad + {\eta^\prime}^2 \mathbb{E}_{r+1} \left\| \nabla f(\mathbf{X}^{(r-1)}) - \nabla F(\mathbf{X}^{(r-1)} ; \xi^{(r-1)}) \right\|^2_F \\
    &\stackrel{(\ref{eq:assumption:stochastic_gradient}), (\ref{eq:relaxed_triangle})}{\leq} (1 + \gamma) \left\| \mathbf{X}^{(r-1)} \mathbf{W} - \bar{\mathbf{X}}^{(r-1)} \right\|^2_F  
    + (1 + \gamma^{-1}) {\eta^\prime}^2 \left\| \nabla f(\mathbf{X}^{(r-1)}) - \mathbf{C}^{(r-1)} \right\|^2_F
    + {\eta^\prime}^2 n \sigma^2 \\
    &\stackrel{(\ref{eq:assumption:mixing_matrix})}{\leq} (1 + \gamma) (1 - p) \left\| \mathbf{X}^{(r-1)} - \bar{\mathbf{X}}^{(r-1)} \right\|^2_F  
    + (1 + \gamma^{-1}) {\eta^\prime}^2 \left\| \nabla f(\mathbf{X}^{(r-1)}) - \mathbf{C}^{(r-1)} \right\|^2_F
    + {\eta^\prime}^2 n \sigma^2.
\end{align*}
By substituting $\gamma=\frac{p}{2}$, we get
\begin{align*}
    %n \Xi^{(r)} 
    &\mathbb{E}_{r+1} \left\| \mathbf{X}^{(r)} - \bar{\mathbf{X}}^{(r-1)} \right\|^2_F \\
    &\leq (1 - \frac{p}{2}) \left\| \mathbf{X}^{(r-1)} - \bar{\mathbf{X}}^{(r-1)} \right\|^2_F  
    + \frac{3}{p} {\eta^\prime}^2 \underbrace{\left\| \nabla f(\mathbf{X}^{(r-1)}) - \mathbf{C}^{(r-1)} \right\|^2_F}_{T}
    + {\eta^\prime}^2 n \sigma^2,
\end{align*}
where we use $p \in (0, 1]$.
Then, $T$ can be estimated as follows:
\begin{align*}
    T 
    &= \left\| \nabla f(\mathbf{X}^{(r-1)}) - \nabla f(\bar{\mathbf{X}}^{(r-1)}) + \nabla f(\bar{\mathbf{X}}^{(r-1)}) - \nabla f(\mathbf{X}^\star) + \nabla f(\mathbf{X}^\star) - \mathbf{C}^{(r-1)} \right\|^2_F \\
    &\stackrel{(\ref{eq:sum_of_n_vectors})}{\leq} 3 \left\| \nabla f(\mathbf{X}^{(r-1)}) - \nabla f(\bar{\mathbf{X}}^{(r-1)}) \right\|^2_F \\
    &\qquad + 3\left\| \nabla f(\bar{\mathbf{X}}^{(r-1)}) - \nabla f(\mathbf{X}^\star) \right\|^2_F
    + 3 \left\| \nabla f(\mathbf{X}^\star) - \mathbf{C}^{(r-1)} \right\|^2_F \\ 
    &\stackrel{(\ref{eq:assumption:smoothness})}{\leq} 3 L^2 \left\| \mathbf{X}^{(r-1)} - \bar{\mathbf{X}}^{(r-1)} \right\|^2_F
    + 6 L n ( f(\bar{\mathbf{x}}^{(r-1)}) - f(\mathbf{x}^\star) )
    + 3 \left\| \nabla f(\mathbf{X}^\star) - \mathbf{C}^{(r-1)} \right\|^2_F.
\end{align*}
This concludes the proof.
\end{proof}
\begin{lemma}
\label{lemma:e_convex}
Suppose that Assumptions \ref{assumption:mixing_matrix}, \ref{assumption:smoothness}, \ref{assumption:stochastic_gradient} and \ref{assumption:mu_convexity} hold,
and $\{ \alpha_{i|j} \}_{ij}$ is set such that $\alpha_{i|j} = \alpha_{j|i} \geq 0$ for all $(i,j) \in \mathcal{E}$.
Then, it holds that
\begin{align*}
    \mathcal{E}^{(r+1)}
    &\leq (1 - \frac{p}{2}) \mathcal{E}^{(r)} 
    + \left( \frac{9 (1 - p) b}{p} + \frac{9 L^2 b^\prime}{p} \right) \Xi^{(r)}
    + \frac{18 L b^\prime}{p} ( \mathbb{E}f(\bar{\mathbf{x}}^{(r)}) - f(\mathbf{x}^\star)) 
    + b^\prime \sigma^2.
\end{align*}
\end{lemma}
\begin{proof}
We have
\begin{align*}
    %&n \mathcal{E}^{(r+1)} \\
    &\mathbb{E}_{r+1} \left\| \mathbf{C}^{(r+1)} - \nabla f(\mathbf{X}^\star) \right\|^2_F \\
    &\stackrel{(\ref{eq:matrix_stochastic_ecl_2})}{=} \mathbb{E}_{r+1} \left\| (\mathbf{C}^{(r)} - \nabla F(\mathbf{X}^{(r)} ; \xi^{(r)})) \mathbf{W} + \frac{1}{2} \mathbf{X}^{(r)} \mathbf{W} (\mathbf{D} - \mathbf{E}) + \nabla F(\mathbf{X}^{(r)} ; \xi^{(r)}) - \nabla f(\mathbf{X}^\star) \right\|^2_F \\
    &\leq \left\| (\mathbf{C}^{(r)} - \nabla f(\mathbf{X}^{(r)})) \mathbf{W} + \frac{1}{2} \mathbf{X}^{(r)} \mathbf{W} (\mathbf{D} - \mathbf{E}) + \nabla f(\mathbf{X}^{(r)}) - \nabla f(\mathbf{X}^\star) \right\|^2_F \\
    &\qquad + \left\| (\nabla F(\mathbf{X}^{(r)} ; \xi^{(r)}) - \nabla f(\mathbf{X}^{(r)})) (\mathbf{W} - \mathbf{I} )\right\|^2_F \\
    &\stackrel{(\ref{eq:assumption:stochastic_gradient})}{\leq} \left\| (\mathbf{C}^{(r)} - \nabla f(\mathbf{X}^{(r)})) \mathbf{W} + \frac{1}{2} \mathbf{X}^{(r)} \mathbf{W} (\mathbf{D} - \mathbf{E}) + \nabla f(\mathbf{X}^{(r)}) - \nabla f(\mathbf{X}^\star) \right\|^2_F 
    + n b^\prime \sigma^2 \\
    &= \left\| (\mathbf{C}^{(r)} - \nabla f(\mathbf{X}^{\star})) \mathbf{W} + \frac{1}{2} \mathbf{X}^{(r)} \mathbf{W} (\mathbf{D} - \mathbf{E}) + (\nabla f(\mathbf{X}^\star) - \nabla f(\mathbf{X}^{(r)}) )(\mathbf{W} - \mathbf{I}) \right\|^2_F 
    + n b^\prime \sigma^2 \\
    &\stackrel{(\ref{eq:relaxed_triangle})}{\leq} (1 + \gamma) \left\| (\mathbf{C}^{(r)} - \nabla f(\mathbf{X}^{\star})) \mathbf{W} \right\|^2_F \\
    &\qquad + (1 + \gamma^{-1}) \left\| \frac{1}{2} \mathbf{X}^{(r)} \mathbf{W} (\mathbf{D} - \mathbf{E}) + (\nabla f(\mathbf{X}^\star) - \nabla f(\mathbf{X}^{(r)}) )(\mathbf{W} - \mathbf{I}) \right\|^2_F 
    + n b^\prime \sigma^2.
\end{align*}
From Lemma \ref{lemma:c_is_zero}, we have $\frac{1}{n} \mathbf{C}^{(r)} \mathbf{1} \mathbf{1}^\top = \mathbf{0}$ 
and $\frac{1}{n} \nabla f(\mathbf{X}^\star) \mathbf{1} \mathbf{1}^\top = \mathbf{0}$.
Then, by substituting $\gamma=\frac{p}{2}$, we get
\begin{align*}
    &\mathbb{E}_{r+1} \left\| \mathbf{C}^{(r+1)} - \nabla f(\mathbf{X}^\star) \right\|^2_F \\
    &\stackrel{(\ref{eq:assumption:mixing_matrix})}{\leq} (1 - \frac{p}{2}) \left\| \mathbf{C}^{(r)} - \nabla f(\mathbf{X}^{\star}) \right\|^2_F \\
    &\qquad + \frac{3}{p} \left\| \frac{1}{2} \mathbf{X}^{(r)} \mathbf{W} (\mathbf{D} - \mathbf{E}) + (\nabla f(\mathbf{X}^\star) - \nabla f(\mathbf{X}^{(r)}) )(\mathbf{W} - \mathbf{I}) \right\|^2_F 
    + n b^\prime \sigma^2 \\
    &= (1 - \frac{p}{2}) \left\| \mathbf{C}^{(r)} - \nabla f(\mathbf{X}^{\star}) \right\|^2_F \\
    &\qquad + \frac{3}{p} \left\| \frac{1}{2} \mathbf{X}^{(r)} \mathbf{W} (\mathbf{D} - \mathbf{E}) + (\nabla f(\mathbf{X}^\star) - \nabla f(\bar{\mathbf{X}}^{(r)}) + \nabla f(\bar{\mathbf{X}}^{(r)}) - \nabla f(\mathbf{X}^{(r)}) )(\mathbf{W} - \mathbf{I}) \right\|^2_F \\
    &\qquad + n b^\prime \sigma^2 \\
    &\stackrel{(\ref{eq:sum_of_n_vectors})}{\leq} (1 - \frac{p}{2}) \left\| \mathbf{C}^{(r)} - \nabla f(\mathbf{X}^{\star}) \right\|^2_F 
    + \frac{9}{p} \left\| \frac{1}{2} \mathbf{X}^{(r)} \mathbf{W} (\mathbf{D} - \mathbf{E}) \right\|^2_F \\
    &\qquad+ \frac{9}{p} \left\| (\nabla f(\mathbf{X}^\star) - \nabla f(\bar{\mathbf{X}}^{(r)}))(\mathbf{W} - \mathbf{I}) \right\|^2_F
    + \frac{9}{p} \left\| \nabla f(\bar{\mathbf{X}}^{(r)}) - \nabla f(\mathbf{X}^{(r)}) )(\mathbf{W} - \mathbf{I}) \right\|^2_F \\ 
    &\qquad + n b^\prime \sigma^2 \\
    &\stackrel{(\ref{eq:assumption:smoothness})}{\leq} (1 - \frac{p}{2}) \left\| \mathbf{C}^{(r)} - \nabla f(\mathbf{X}^{\star}) \right\|^2_F 
    + \frac{9}{p} \left\| \frac{1}{2} \mathbf{X}^{(r)} \mathbf{W} (\mathbf{D} - \mathbf{E}) \right\|^2_F 
    + \frac{18 L b^\prime n}{p} ( f(\bar{\mathbf{x}}^{(r)}) - f(\mathbf{x}^\star)) \\
    &\qquad + \frac{9 L^2 b^\prime}{p} \left\| \bar{\mathbf{X}}^{(r)} - \mathbf{X}^{(r)} \right\|^2_F 
    + n b^\prime \sigma^2.
\end{align*}
Using the definitions of $\mathbf{D}$ and $\mathbf{E}$, we have $\bar{\mathbf{X}}^{(r)} (\mathbf{D} - \mathbf{E})=\mathbf{0}$.
Then, we get
\begin{align*}
    %&n \mathcal{E}^{(r+1)} \\
    &\mathbb{E}_{r+1} \left\| \mathbf{C}^{(r+1)} - \nabla f(\mathbf{X}^\star) \right\|^2_F \\
    &\leq (1 - \frac{p}{2}) \left\| \mathbf{C}^{(r)} - \nabla f(\mathbf{X}^{\star}) \right\|^2_F 
    + \frac{9}{p} \left\| \frac{1}{2} (\mathbf{X}^{(r)} \mathbf{W} - \bar{\mathbf{X}}^{(r)}) (\mathbf{D} - \mathbf{E}) \right\|^2_F \\
    &\qquad + \frac{18 L b^\prime n}{p} ( f(\bar{\mathbf{x}}^{(r)}) - f(\mathbf{x}^\star)) 
    + \frac{9 L^2 b^\prime}{p} \left\| \bar{\mathbf{X}}^{(r)} - \mathbf{X}^{(r)} \right\|^2_F 
    + n b^\prime \sigma^2 \\
    &\stackrel{(\ref{eq:assumption:mixing_matrix})}{\leq} (1 - \frac{p}{2}) \left\| \mathbf{C}^{(r)} - \nabla f(\mathbf{X}^{\star}) \right\|^2_F 
    + \left( \frac{9 (1 - p) b}{p} + \frac{9 L^2 b^\prime}{p} \right) \left\| \mathbf{X}^{(r)} - \bar{\mathbf{X}}^{(r)} \right\|^2_F \\
    &\qquad + \frac{18 L b^\prime n}{p} ( f(\bar{\mathbf{x}}^{(r)}) - f(\mathbf{x}^\star)) 
    + n b^\prime \sigma^2.
\end{align*}
This concludes the proof.
\end{proof}

\begin{lemma}
\label{lemma:consensus_e_convex}
Suppose that Assumptions \ref{assumption:mixing_matrix}, \ref{assumption:smoothness}, \ref{assumption:stochastic_gradient} and \ref{assumption:mu_convexity} hold,
and $\{ \alpha_{i|j} \}_{ij}$ is set such that $\alpha_{i|j} = \alpha_{j|i} \geq 0$ for all $(i,j) \in \mathcal{E}$.
If $\eta^\prime$ satisfies
\begin{align}
    \label{eq:lemma:consensus_e_convex}
    \eta^\prime \leq \frac{p}{6 \sqrt{L^2 + \frac{36 \{ (1-p) b + L^2 b^\prime \} }{p^2}} },
\end{align}
then it holds that
\begin{align*}
    \Xi^{(r+1)} + \frac{36}{p^2} {\eta^\prime}^2 \mathcal{E}^{(r+1)} 
    &\leq \left(1 - \frac{p}{4} \right) ( \Xi^{(r)} + \frac{36}{p^2}{\eta^\prime}^2 \mathcal{E}^{(r)}) \\
    &\qquad + \left( \frac{18 L}{p} + \frac{648 L b^\prime}{p^3} \right) {\eta^\prime}^2 (\mathbb{E}f(\bar{\mathbf{x}}^{(r)}) - f(\mathbf{x}^\star))
    + \left( 1 + \frac{36 b^\prime}{p^2} \right) {\eta^\prime}^2 \sigma^2.
\end{align*}
\end{lemma}
\begin{proof}
From Lemmas \ref{lemma:consensus_convex} and \ref{lemma:e_convex}, we have
\begin{align*}
    \Xi^{(r+1)} 
    &\leq (1 - \frac{p}{2}) \Xi^{(r)}
    + \frac{9 L^2 {\eta^\prime}^2}{p} \Xi^{(r)}
    + \frac{18 L}{p} {\eta^\prime}^2 (\mathbb{E}f(\bar{\mathbf{x}}^{(r)}) - f(\mathbf{x}^\star))
    + \frac{9}{p} {\eta^\prime}^2 \mathcal{E}^{(r)}
    + {\eta^\prime}^2 \sigma^2, \\
    \frac{36}{p^2} \mathcal{E}^{(r+1)}
    &\leq (1 - \frac{p}{2}) \frac{36}{p^2} \mathcal{E}^{(r)} 
    + \frac{36}{p^2} \left( \frac{9 (1 - p) b}{p} + \frac{9 L^2 b^\prime}{p} \right) \Xi^{(r)} \\
    &\qquad + \frac{648 L b^\prime}{p^3} (\mathbb{E} f(\bar{\mathbf{x}}^{(r)}) - f(\mathbf{x}^\star)) 
    + \frac{36 b^\prime}{p^2} \sigma^2.
\end{align*}
Then, we have
\begin{align*}
    \Xi^{(r+1)} + \frac{36}{p^2} {\eta^\prime}^2 \mathcal{E}^{(r+1)}
    &\leq \left( (1 - \frac{p}{2}) + \frac{9 L^2}{p} {\eta^\prime}^2 + \frac{324 (1 - p) b + 324 L^2 b^\prime}{p^3} {\eta^\prime}^2 \right) \Xi^{(r)} \\
    &\qquad + \left( \frac{9}{p} {\eta^\prime}^2 +  \left( 1 - \frac{p}{2} \right) \frac{36}{p^2} {\eta^\prime}^2  \right) \mathcal{E}^{(r)} \\
    &\qquad + \left( \frac{18 L}{p} {\eta^\prime}^2 + \frac{648 L b^\prime}{p^3} {\eta^\prime}^2 \right) (\mathbb{E}f(\bar{\mathbf{x}}^{(r)}) - f(\mathbf{x}^\star)) \\
    &\qquad + \left( {\eta^\prime}^2 + \frac{36 b^\prime}{p^2} {\eta^\prime}^2 \right) \sigma^2.
\end{align*}
Using ${\eta^\prime}^2 \leq p^2 / ( 36 ( L^2 + \frac{36 \{ (1-p) b + L^2 b^\prime \} }{p^2} ))$, we get
\begin{align*}
    1 - \frac{p}{2} + \frac{9 L^2 {\eta^\prime}^2}{p} + \frac{324 (1 - p) b + 324 L^2 b^\prime}{p^3} {\eta^\prime}^2 \leq 1 - \frac{p}{4}.
\end{align*}
This concludes the proof.
\end{proof}

\begin{lemma}
\label{lemma:simplified_consensus_convex}
Suppose that Assumptions \ref{assumption:mixing_matrix}, \ref{assumption:smoothness}, \ref{assumption:stochastic_gradient} and \ref{assumption:mu_convexity} hold,
and $\{ \alpha_{i|j} \}_{ij}$ is set such that $\alpha_{i|j} = \alpha_{j|i} \geq 0$ for all $(i,j) \in \mathcal{E}$.
If $\{ w^{(r)} \}_r$ is $\frac{4}{p}$-slow increasing positive sequence of weights
and the step size $\eta^\prime$ satisfies
\begin{align}
    \label{eq:lemma:simplified_consensns_convex}
    \eta^\prime \leq \min \left\{ \frac{p}{6 \sqrt{L^2 + \frac{36 \{ (1-p) b + L^2 b^\prime \} }{p^2}} }, 
    \frac{p}{24 L \sqrt{2 + \frac{72 b^\prime}{p^2} }} \right\},
\end{align}
then it holds that
\begin{align*}
    3L \sum_{r=0}^{R} w^{(r)} \Xi^{(r)}
    &\leq
    \frac{1}{2} \sum_{j=0}^{R} w^{(j)} (\mathbb{E}f(\bar{\mathbf{x}}^{(j)}) - f(\mathbf{x}^\star)) 
    + {\eta^\prime}^2 \left( 12 + \frac{432 b^\prime}{p^2} \right) \frac{L \sigma^2}{p} \sum_{r=0}^{R} w^{(r)}.
\end{align*}
\end{lemma}
\begin{proof}
We define $\Theta^{(r)} \coloneqq \Xi^{(r)} + \frac{36}{p^2} {\eta^\prime}^2 \mathcal{E}^{(r)}$.
From Lemma \ref{lemma:consensus_e_convex}, we get
\begin{align*}
    \Theta^{(r)}
    &\leq
    {\eta^\prime}^2 \left( \frac{18 L}{p} + \frac{648 L b^\prime}{p^3} \right) \sum_{j=0}^{r-1} \left( 1 - \frac{p}{4} \right)^{r-j-1} (\mathbb{E}f(\bar{\mathbf{x}}^{(j)}) - f(\mathbf{x}^\star)) \\
    &\qquad + {\eta^\prime}^2 \left( 1 + \frac{36 b^\prime}{p^2} \right) \sigma^2 \sum_{j=0}^{r-1} \left( 1 - \frac{p}{4} \right)^{r-j-1} \\
    &\stackrel{\frac{3}{4} \leq 1 - \frac{p}{4}}{\leq} {\eta^\prime}^2 \left( \frac{24 L}{p} + \frac{864 L b^\prime}{p^3} \right) \sum_{j=0}^{r-1} \left( 1 - \frac{p}{4} \right)^{r-j} (\mathbb{E}f(\bar{\mathbf{x}}^{(j)}) - f(\mathbf{x}^\star))
    + {\eta^\prime}^2 \left( 4 + \frac{144 b^\prime}{p^2} \right) \frac{\sigma^2}{p},
\end{align*}
for any round $r\geq1$.
Then, we get
\begin{align*}
    \sum_{r=1}^{R} w^{(r)} \Theta^{(r)}
    &\leq
    {\eta^\prime}^2 \left( \frac{24 L}{p} + \frac{864 L b^\prime}{p^3} \right) \sum_{r=1}^{R} w^{(r)}  \sum_{j=0}^{r-1} \left( 1 - \frac{p}{4} \right)^{r-j} (\mathbb{E}f(\bar{\mathbf{x}}^{(j)}) - f(\mathbf{x}^\star)) \\
    &\qquad + {\eta^\prime}^2 \left( 4 + \frac{144 b^\prime}{p^2} \right) \frac{\sigma^2}{p} \sum_{r=1}^{R} w^{(r)}.
\end{align*}
By using that $\{ w^{(r)} \}_r$ is $\frac{4}{p}$-slow increasing
(i.e., $w^{(r)} \leq (1 + \frac{p}{8} )^{r-j} w^{(j)}$),
we get
\begin{align*}
    \sum_{r=1}^{R} w^{(r)} \Theta^{(r)}
    &\leq
    {\eta^\prime}^2 \left( \frac{24 L}{p} + \frac{864 L b^\prime}{p^3} \right) \sum_{r=1}^{R}  \sum_{j=0}^{r-1} \left( 1 - \frac{p}{8} \right)^{r-j} w^{(j)} (\mathbb{E}f(\bar{\mathbf{x}}^{(j)}) - f(\mathbf{x}^\star)) \\
    &\qquad + {\eta^\prime}^2 \left( 4 + \frac{144 b^\prime}{p^2} \right) \frac{\sigma^2}{p} \sum_{r=1}^{R} w^{(r)} \\
    &=
    {\eta^\prime}^2 \left( \frac{24 L}{p} + \frac{864 L b^\prime}{p^3} \right) \sum_{j=0}^{R-1} w^{(j)} (\mathbb{E}f(\bar{\mathbf{x}}^{(j)}) - f(\mathbf{x}^\star)) \sum_{r=j+1}^{R} \left( 1 - \frac{p}{8} \right)^{r-j} \\
    &\qquad + {\eta^\prime}^2 \left( 4 + \frac{144 b^\prime}{p^2} \right) \frac{\sigma^2}{p} \sum_{r=1}^{R} w^{(r)} \\
    &\leq
    {\eta^\prime}^2 \left( \frac{192 L}{p^2} + \frac{6912 L b^\prime}{p^4} \right) \sum_{j=0}^{R-1} w^{(j)} (\mathbb{E}f(\bar{\mathbf{x}}^{(j)}) - f(\mathbf{x}^\star)) \\
    &\qquad + {\eta^\prime}^2 \left( 4 + \frac{144 b^\prime}{p^2} \right) \frac{\sigma^2}{p} \sum_{r=1}^{R} w^{(r)}.
\end{align*}
By using that $\Theta^{(r)} \geq \Xi^{(r)}$, we get
\begin{align*}
    &\sum_{r=1}^{R} w^{(r)} \Xi^{(r)} \\
    &\leq
    {\eta^\prime}^2 \left( \frac{192 L}{p^2} + \frac{6912 L b^\prime}{p^4} \right) \sum_{j=0}^{R-1} w^{(j)} (\mathbb{E}f(\bar{\mathbf{x}}^{(j)}) - f(\mathbf{x}^\star)) 
    + {\eta^\prime}^2 \left( 4 + \frac{144 b^\prime}{p^2} \right) \frac{\sigma^2}{p} \sum_{r=1}^{R} w^{(r)}.
\end{align*}
By using that $\Xi^{(0)}=0$, $w^{(R)} (\mathbb{E}f(\bar{\mathbf{x}}^{(R)}) - f(\mathbf{x}^\star)) \geq 0$, and $w^{(0)} \geq 0$, we get
\begin{align*}
    &\sum_{r=0}^{R} w^{(r)} \Xi^{(r)} \\
    &\leq
    {\eta^\prime}^2 \left( \frac{192 L}{p^2} + \frac{6912 L b^\prime}{p^4} \right) \sum_{j=0}^{R} w^{(j)} (\mathbb{E}f(\bar{\mathbf{x}}^{(j)}) - f(\mathbf{x}^\star)) 
    + {\eta^\prime}^2 \left( 4 + \frac{144 b^\prime}{p^2} \right) \frac{\sigma^2}{p} \sum_{r=0}^{R} w^{(r)}.
\end{align*}
By multiplying the above equation by $3L$, we get
\begin{align*}
    &3L \sum_{r=0}^{R} w^{(r)} \Xi^{(r)} \\
    &\leq
    {\eta^\prime}^2 \left( 576 + \frac{20736 b^\prime}{p^2} \right) \frac{L^2}{p^2} \sum_{j=0}^{R} w^{(j)} (\mathbb{E}f(\bar{\mathbf{x}}^{(j)}) - f(\mathbf{x}^\star)) 
    + {\eta^\prime}^2 \left( 12 + \frac{432 b^\prime}{p^2} \right) \frac{L \sigma^2}{p} \sum_{r=0}^{R} w^{(r)}.
\end{align*}
Using that ${\eta^\prime}^2 \leq p^2 / (1152 L^2 (1 + \frac{36 b^\prime}{p^2}))$, we get the statement.
\end{proof}

\begin{lemma}
\label{lemma:recursion_convex}
Suppose that Assumptions \ref{assumption:mixing_matrix}, \ref{assumption:smoothness}, \ref{assumption:stochastic_gradient} and \ref{assumption:mu_convexity} hold,
and $\{ \alpha_{i|j} \}_{ij}$ is set such that $\alpha_{i|j} = \alpha_{j|i} \geq 0$ for all $(i,j) \in \mathcal{E}$.
If $\{ w^{(r)} \}_r$ is $\frac{4}{p}$-slow increasing positive sequence of weights
and the step size $\eta^\prime$ satisfies
\begin{align*}
    \eta^\prime \leq \min \left\{ \frac{1}{12L}, \frac{p}{6 \sqrt{L^2 + \frac{36 \{ (1-p) b + L^2 b^\prime \} }{p^2}} }, 
    \frac{p}{24 L \sqrt{ 2 + \frac{72 b^\prime}{p^2} }} \right\},
\end{align*}
then it holds that
\begin{align*}
    &\frac{1}{2 W_R} \sum_{r=0}^{R} w^{(r)}(\mathbb{E}f(\bar{\mathbf{x}}^{(r)}) - f (\mathbf{x}^\star)) \\
    &\leq \frac{1}{\eta^\prime W_R} \sum_{r=0}^{R} w^{(r)} \left( \left(1 - \frac{\eta^\prime \mu}{2} \right) \mathbb{E} \left\| \bar{\mathbf{x}}^{(r)} - \mathbf{x}^\star \right\|^2 
    - \mathbb{E} \left\| \bar{\mathbf{x}}^{(r+1)} - \mathbf{x}^\star \right\|^2 \right) \\
    &\qquad + {\eta^\prime} \frac{\sigma^2}{n}
    + {\eta^\prime}^2 \left( 12 + \frac{432 b^\prime}{p^2} \right) \frac{L \sigma^2}{p}.
\end{align*}
where $W_R \coloneqq \sum_{r=0}^{R} w^{(r)}$.
\end{lemma}
\begin{proof}
From Lemma \ref{lemma:descent_convex}, we have
\begin{align*}
    \sum_{r=0}^{R} w^{(r)}(\mathbb{E}f(\bar{\mathbf{x}}^{(r)}) - f (\mathbf{x}^\star)) 
    &\leq \frac{1}{\eta^\prime} \sum_{r=0}^{R} w^{(r)} \left( \left(1 - \frac{\eta^\prime \mu}{2} \right) \mathbb{E} \left\| \bar{\mathbf{x}}^{(r)} - \mathbf{x}^\star \right\|^2
    - \mathbb{E} \left\| \bar{\mathbf{x}}^{(r+1)} - \mathbf{x}^\star \right\|^2 \right) \\
    &\qquad + \frac{{\eta^\prime} \sigma^2}{n} \sum_{r=0}^{R} w^{(r)}
    + 3L \sum_{r=0}^{R} w^{(r)} \Xi^{(r)}.
\end{align*}
Using Lemma \ref{lemma:simplified_consensus_convex}, we get
\begin{align*}
    \frac{1}{2} \sum_{r=0}^{R} w^{(r)}(\mathbb{E}f(\bar{\mathbf{x}}^{(r)}) - f (\mathbf{x}^\star)) 
    &\leq \frac{1}{\eta^\prime} \sum_{r=0}^{R} w^{(r)} \left( \left(1 - \frac{\eta^\prime \mu}{2} \right) \mathbb{E} \left\| \bar{\mathbf{x}}^{(r)} - \mathbf{x}^\star \right\|^2 
    - \mathbb{E} \left\| \bar{\mathbf{x}}^{(r+1)} - \mathbf{x}^\star \right\|^2 \right) \\
    &\qquad + \frac{{\eta^\prime} \sigma^2}{n} \sum_{r=0}^{R} w^{(r)}
    + {\eta^\prime}^2 \left( 12 + \frac{432 b^\prime}{p^2} \right) \frac{L \sigma^2}{p} \sum_{r=0}^{R} w^{(r)}.
\end{align*}
This concludes the proof.
\end{proof}

\begin{lemma}[Convergence Rate for Strongly Convex Cases]
Suppose that Assumptions \ref{assumption:mixing_matrix}, \ref{assumption:smoothness}, \ref{assumption:stochastic_gradient} and \ref{assumption:mu_convexity} hold with $\mu>0$,
and $\{ \alpha_{i|j} \}_{ij}$ is set such that $\alpha_{i|j} = \alpha_{j|i} \geq 0$ for all $(i,j) \in \mathcal{E}$.
For weights $w^{(r)} \coloneqq (1 - \frac{\mu \eta^\prime}{2})^{-(r+1)}$ and $W_R \coloneqq \sum_{r=0}^{R} w^{(r)}$,
there exists a step size $\eta^\prime < \frac{1}{d}$ such that it holds that
\begin{align*}
    &\frac{1}{2 W_R} \sum_{r=0}^{R} w^{(r)} (\mathbb{E}f(\bar{\mathbf{x}}^{(r)}) - f (\mathbf{x}^\star)) 
    + \frac{\mu}{2} \mathbb{E} \left\| \bar{\mathbf{x}}^{(R+1)} - \mathbf{x}^\star \right\|^2 \\
    &\leq 
    \tilde{\mathcal{O}} \left( r_0 d \exp \left[ \frac{- \mu (R+1)}{d} \right] 
    + \frac{\sigma^2}{n \mu R} 
    + \frac{(1 + \frac{b^\prime}{p^2}) L \sigma^2}{\mu^2 R^2 p} \right),
\end{align*}
where $r_0 \coloneqq \| \bar{\mathbf{x}}^{(0)} - \mathbf{x}^\star \|^2$.
\end{lemma}
\begin{proof}
We define $w^{(r)} \coloneqq (1 - \frac{\mu \eta^\prime}{2})^{-(r+1)}$.
Suppose $\eta^\prime \leq \min \{ \frac{2p}{\mu(8+p)}, \frac{2}{\mu} \}$.
Then, we have
\begin{align*}
    \frac{1}{(1 - \frac{\mu \eta^\prime}{2})} &\leq 1 + \frac{p}{8}, \\
    1 - \frac{\mu \eta^\prime}{2} &> 0.
\end{align*}
Therefore, $\{ w^{(r)} \}_r$ is $\frac{4}{p}$-slow increasing.
From Lemma \ref{lemma:recursion_convex}, substituting $w^{(r)} \coloneqq (1 - \frac{\mu \eta^\prime}{2})^{-(r+1)}$,
we get
\begin{align*}
    &\frac{1}{2 W_R} \sum_{r=0}^{R} w^{(r)} (\mathbb{E}f(\bar{\mathbf{x}}^{(r)}) - f (\mathbf{x}^\star)) \\
    &\leq \frac{1}{\eta^\prime W_R} \left( \left\| \bar{\mathbf{x}}^{(0)} - \mathbf{x}^\star \right\|^2
    \!\!\!
    - w^{(R)} \mathbb{E} \left\| \bar{\mathbf{x}}^{(R+1)} - \mathbf{x}^\star \right\|^2 \right) 
    + {\eta^\prime} \frac{\sigma^2}{n} 
    + {\eta^\prime}^2 \left( 12 + \frac{432 b^\prime}{p^2} \right) \frac{L \sigma^2}{p}.
\end{align*}
Unrolling the above equation, we get
\begin{align*}
    &\frac{1}{2 W_R} \sum_{r=0}^{R} w^{(r)} (\mathbb{E}f(\bar{\mathbf{x}}^{(r)}) - f (\mathbf{x}^\star)) 
    + \frac{w^{(R)}}{\eta^\prime W_R} \mathbb{E} \left\| \bar{\mathbf{x}}^{(R+1)} - \mathbf{x}^\star \right\|^2 \\
    &\leq \frac{ \left\| \bar{\mathbf{x}}^{(0)} - \mathbf{x}^\star \right\|^2 }{\eta^\prime W_R}
    + {\eta^\prime} \frac{\sigma^2}{n} 
    + {\eta^\prime}^2 \left( 12 + \frac{432 b^\prime}{p^2} \right) \frac{L \sigma^2}{p}.
\end{align*}
Using the followings:
\begin{align*}
    \frac{1}{W_R} &\leq (1 - \frac{\mu \eta^\prime}{2})^{(R+1)} \leq \exp \{- \frac{\mu \eta^\prime}{2} (R+1) \}, \\
    W_R &= (1 - \frac{\mu \eta^\prime}{2})^{-(R+1)} \sum_{r=0}^{R} (1 - \frac{\mu \eta^\prime}{2})^{r} \leq \frac{2 w^{(R)}}{ \mu \eta^\prime},
\end{align*}
we get
\begin{align*}
    &\frac{1}{2 W_R} \sum_{r=0}^{R} w^{(r)} (\mathbb{E}f(\bar{\mathbf{x}}^{(r)}) - f (\mathbf{x}^\star)) 
    + \frac{\mu}{2} \mathbb{E} \left\| \bar{\mathbf{x}}^{(R+1)} - \mathbf{x}^\star \right\|^2  \\
    &\leq \frac{1}{\eta^\prime} \left\| \bar{\mathbf{x}}^{(0)} - \mathbf{x}^\star \right\|^2 \exp \left( \frac{- \mu \eta^\prime}{2} (R+1) \right)
    + {\eta^\prime} \frac{\sigma^2}{n} 
    + {\eta^\prime}^2 \left( 12 + \frac{432 b^\prime}{p^2} \right) \frac{L \sigma^2}{p}.
\end{align*}
Then, by tuning $\eta^\prime$ as in Lemma 15 in \citep{koloskova2020unified} and Lemma 2 in \citep{stich2019unified},
we can get the statement.
\end{proof}

\begin{lemma}[Convergence Rate for General Convex Cases]
Suppose that Assumptions \ref{assumption:mixing_matrix}, \ref{assumption:smoothness}, \ref{assumption:stochastic_gradient} and \ref{assumption:mu_convexity} hold with $\mu=0$,
and $\{ \alpha_{i|j} \}_{ij}$ is set such that $\alpha_{i|j} = \alpha_{j|i} \geq 0$ for all $(i,j) \in \mathcal{E}$.
There exists a step size $\eta^\prime < \frac{1}{d}$ such that it holds that
\begin{align*}
    &\frac{1}{2 (R+1)} \sum_{r=0}^{R} (\mathbb{E}f(\bar{\mathbf{x}}^{(r)}) - f (\mathbf{x}^\star)) \\
    &\leq \mathcal{O} \left( 
    \left( \frac{\sigma^2 \| \bar{\mathbf{x}}^{(0)} - \mathbf{x}^\star \|^2 }{n(R+1)} \right)^{\frac{1}{2}}
    + \left( \frac{(1 + \frac{b^\prime}{p^2})L \sigma^2}{p} \right)^{\frac{1}{3}} \left( \frac{\| \bar{\mathbf{x}}^{(0)} - \mathbf{x}^\star \|^2}{R+1} \right)^{\frac{2}{3}}
    + \frac{d \| \bar{\mathbf{x}}^{(0)} - \mathbf{x}^\star \|^2}{R+1}
    \right).
\end{align*}
\end{lemma}
\begin{proof}
From Lemma \ref{lemma:recursion_convex}, by substituting $w^{(r)} \coloneqq 1$, we get
\begin{align*}
    \frac{1}{2 (R+1)} \sum_{r=0}^{R} (\mathbb{E}f(\bar{\mathbf{x}}^{(r)}) - f (\mathbf{x}^\star)) 
    &\leq \frac{1}{\eta^\prime (R+1) }  \left\| \bar{\mathbf{x}}^{(0)} - \mathbf{x}^\star \right\|^2 
    + {\eta^\prime} \frac{\sigma^2}{n}
    + {\eta^\prime}^2 \left( 12 + \frac{432 b^\prime}{p^2} \right) \frac{L \sigma^2}{p}.
\end{align*}
Then, using Lemma 17 in \citep{koloskova2020unified}, we can get the statement.
\end{proof}

\subsection{Convergence Analysis for Non-convex Case}
\label{sec:convergence_analysis_for_nonconvex}
\subsubsection{Additional Notation}
In Sec. \ref{sec:convergence_analysis_for_nonconvex}, we define $\Xi^{(r)}, \mathcal{E}^{(r)}$, $b$, and $b^\prime$ as follows to simplify the notation:
\begin{gather*}
    \Xi^{(r)} \coloneqq \frac{1}{n} \mathbb{E} \sum_{i=1}^{n} \left\| \mathbf{x}_i^{(r)} - \bar{\mathbf{x}}^{(r)} \right\|^2,
    \;\;
    \mathcal{E}^{(r)} \coloneqq \frac{1}{n} \mathbb{E} \left\| \nabla f(\bar{\mathbf{X}}^{(r)}) - \mathbf{C}^{(r)} - \frac{1}{n} \nabla f(\bar{\mathbf{X}}^{(r)}) \mathbf{1}\mathbf{1}^\top \right\|^2_F,
    \;\; \\
    b \coloneqq \left\| \frac{1}{2} (\mathbf{D} - \mathbf{E}) \right\|^2_F, \;\;
    b^\prime \coloneqq \left\| \mathbf{W} - \mathbf{I} \right\|^2_F.
\end{gather*}

\subsubsection{Convergence Analysis}
\begin{lemma}[Descent Lemma for Non-convex Case]
\label{lemma:descent_nonconvex}
Suppose that Assumptions \ref{assumption:mixing_matrix}, \ref{assumption:smoothness}, and \ref{assumption:stochastic_gradient} hold,
and $\{ \alpha_{i|j} \}_{ij}$ is set such that $\alpha_{i|j}=\alpha_{j|i} \geq 0$ for all $(i,j) \in \mathcal{E}$.
If $\eta^\prime \leq \frac{1}{4 L}$, it holds that
\begin{align*}
    \mathbb{E}_{r+1} f(\bar{\mathbf{x}}^{(r+1)}) \leq f(\bar{\mathbf{x}}^{(r)}) - \frac{\eta^\prime}{4} \left\| \nabla f(\bar{\mathbf{x}}^{(r)}) \right\|^2
    + \frac{L^2 \eta^\prime}{n} \sum_{i=1}^{n} \left\| \mathbf{x}_i^{(r)} - \bar{\mathbf{x}}^{(r)} \right\|^2
    + \frac{L\sigma^2 {\eta^\prime}^2}{2n}.
\end{align*}
\end{lemma}
\begin{proof}
We have
\begin{align*}
    &\mathbb{E}_{r+1} f(\bar{\mathbf{x}}^{(r+1)}) \\
    &\stackrel{(\ref{eq:average_in_ecl})}{=} \mathbb{E}_{r+1} f(\bar{\mathbf{x}}^{(r)} - \frac{\eta^\prime}{n} \sum_{i=1}^{n} \nabla F_i(\mathbf{x}_i^{(r)} ; \xi^{(r)}_i)) \\
    &\stackrel{(\ref{eq:assumption:smoothness})}{\leq} f (\bar{\mathbf{x}}^{(r)}) 
    \underbrace{- \mathbb{E}_{r+1} \left\langle \nabla f(\bar{\mathbf{x}}^{(r)}), \frac{\eta^\prime}{n} \sum_{i=1}^{n} \nabla F_i (\mathbf{x}_i^{(r)} ; \xi_i^{(r)}) \right\rangle}_{T_1}
    + \frac{L}{2} {\eta^\prime}^2 \underbrace{\mathbb{E}_{r+1} \left\| \frac{1}{n} \sum_{i=1}^{n} \nabla F_i(\mathbf{x}_i^{(r)} ; \xi_i^{(r)} ) \right\|^2}_{T_2}.
\end{align*}
Then, we can estimate $T_1$ as follows:
\begin{align*}
    T_1 &= - \frac{\eta^\prime}{n} \sum_{i=1}^{n}  \left\langle \nabla f(\bar{\mathbf{x}}^{(r)}), \nabla f_i (\mathbf{x}_i^{(r)}) \right\rangle \\
    &= - \frac{\eta^\prime}{n} \sum_{i=1}^{n}  \left\langle \nabla f(\bar{\mathbf{x}}^{(r)}), \nabla f_i (\mathbf{x}_i^{(r)}) - \nabla f_i (\bar{\mathbf{x}}^{(r)}) + \nabla f_i (\bar{\mathbf{x}}^{(r)}) \right\rangle \\
    &= - {\eta^\prime} \left\| \nabla f(\bar{\mathbf{x}}^{(r)}) \right\|^2 
    + \frac{\eta^\prime}{n} \sum_{i=1}^{n}  \left\langle \nabla f(\bar{\mathbf{x}}^{(r)}),  \nabla f_i (\bar{\mathbf{x}}^{(r)}) - \nabla f_i (\mathbf{x}_i^{(r)}) \right\rangle \\
    &\stackrel{(\ref{eq:inner_prod})}{\leq} - \frac{\eta^\prime}{2} \left\| \nabla f(\bar{\mathbf{x}}^{(r)}) \right\|^2
    + \frac{\eta^\prime}{2n} \sum_{i=1}^{n} \left\| \nabla f_i (\mathbf{x}_i^{(r)}) - \nabla f_i (\bar{\mathbf{x}}^{(r)}) \right\|^2 \\
    &\stackrel{(\ref{eq:assumption:smoothness})}{\leq} - \frac{\eta^\prime}{2} \left\| \nabla f(\bar{\mathbf{x}}^{(r)}) \right\|^2 
    + \frac{L^2 \eta^\prime}{2n} \sum_{i=1}^{n} \left\| \mathbf{x}_i^{(r)} - \bar{\mathbf{x}}^{(r)} \right\|^2.
\end{align*}
Then, we can estimate $T_2$ as follows:
\begin{align*}
    T_2 
    &= \mathbb{E}_{r+1} \left\| \frac{1}{n} \sum_{i=1}^{n} \nabla F_i(\mathbf{x}_i^{(r)} ; \xi_i^{(r)} ) \right\|^2 \\
    &\leq \left\| \frac{1}{n} \sum_{i=1}^{n} \nabla f_i(\mathbf{x}_i^{(r)}) \right\|^2 
    + \mathbb{E}_{r+1} \left\| \frac{1}{n} \sum_{i=1}^{n} ( \nabla F_i(\mathbf{x}_i^{(r)} ; \xi_i^{(r)} ) - \nabla f_i(\mathbf{x}_i^{(r)}) )\right\|^2 \\
    &\stackrel{(\ref{eq:assumption:stochastic_gradient})}{\leq} \left\| \frac{1}{n} \sum_{i=1}^{n} \nabla f_i(\mathbf{x}_i^{(r)}) \right\|^2 + \frac{\sigma^2}{n} \\
    &= \left\| \frac{1}{n} \sum_{i=1}^{n} \nabla f_i(\mathbf{x}_i^{(r)}) - \nabla f_i(\bar{\mathbf{x}}^{(r)}) + \nabla f_i(\bar{\mathbf{x}}^{(r)}) \right\|^2 + \frac{\sigma^2}{n} \\
    &\stackrel{(\ref{eq:sum_of_n_vectors})}{\leq} 2 \left\| \frac{1}{n} \sum_{i=1}^{n} \nabla f_i(\mathbf{x}_i^{(r)}) - \nabla f_i(\bar{\mathbf{x}}^{(r)}) \right\|^2 + 
    2 \left\| \nabla f(\bar{\mathbf{x}}^{(r)}) \right\|^2 + \frac{\sigma^2}{n} \\
    &\stackrel{(\ref{eq:assumption:smoothness})}{\leq} \frac{2 L^2}{n} \sum_{i=1}^{n} \left\| \mathbf{x}_i^{(r)} - \bar{\mathbf{x}}^{(r)} \right\|^2 + 
    2 \left\| \nabla f(\bar{\mathbf{x}}^{(r)}) \right\|^2 + \frac{\sigma^2}{n}.
\end{align*}
Combining the above equations, we get
\begin{align*}
    &\mathbb{E}_{r+1} f(\bar{\mathbf{x}}^{(r+1)})  \\
    &\leq 
    f (\bar{\mathbf{x}}^{(r)}) 
    - \left( \frac{\eta^\prime}{2} - L {\eta^\prime}^2 \right) \left\| \nabla f(\bar{\mathbf{x}}^{(r)}) \right\|^2 
    + \left( \frac{L^2 \eta^\prime}{2n} + \frac{L^3 {\eta^\prime}^2}{n} \right) \sum_{i=1}^{n} \left\| \mathbf{x}^{(r)} - \bar{\mathbf{x}}^{(r)} \right\|^2
    + \frac{L \sigma^2}{2 n} {\eta^\prime}^2.
\end{align*}
Using that $\eta^\prime \leq \frac{1}{4L}$, we get the statement.
\end{proof}

\begin{lemma}[Recursion for Consensus Distance]
\label{lemma:consensus_nonconvex}
Suppose that Assumptions \ref{assumption:mixing_matrix}, \ref{assumption:smoothness}, and \ref{assumption:stochastic_gradient} hold,
and $\{ \alpha_{i|j} \}_{ij}$ is set such that $\alpha_{i|j}=\alpha_{j|i} \geq 0$ for all $(i,j) \in \mathcal{E}$.
Then, it holds that
\begin{align*}
    \Xi^{(r+1)}
    \leq (1 - \frac{p}{2}) \Xi^{(r)} 
    + \frac{6L^2}{p} {\eta^\prime}^2  \Xi^{(r)}  
    + \frac{6}{p} {\eta^\prime}^2 \mathbb{E} \| \nabla f(\bar{\mathbf{x}}^{(r)}) \|^2
    + \frac{6}{p} {\eta^\prime}^2 \mathcal{E}^{(r)} 
    + {\eta^\prime}^2 \sigma^2.
\end{align*}
\end{lemma}
\begin{proof}
As in Lemma \ref{lemma:consensus_convex}, we get
\begin{align*}
    n \Xi^{(r+1)} \leq \mathbb{E} \left\| \mathbf{X}^{(r+1)} - \bar{\mathbf{X}}^{(r)} \right\|^2_F.
\end{align*}
Then, we can estimate as follows:
\begin{align*}
    &\mathbb{E}_{r+1} \left\| \mathbf{X}^{(r+1)} - \bar{\mathbf{X}}^{(r)} \right\|^2_F \\
    &= \mathbb{E}_{r+1} \left\| \mathbf{X}^{(r)} \mathbf{W} - \eta^\prime (\nabla F(\mathbf{X}^{(r)} ; \xi^{(r)}) - \mathbf{C}^{(r)}) - \bar{\mathbf{X}}^{(r)} \right\|^2_F \\
    &\stackrel{(\ref{eq:assumption:stochastic_gradient})}{\leq} \left\| \mathbf{X}^{(r)} \mathbf{W} - \eta^\prime (\nabla f(\mathbf{X}^{(r)}) - \mathbf{C}^{(r)}) - \bar{\mathbf{X}}^{(r)} \right\|^2_F
    + {\eta^\prime}^2 n \sigma^2 \\
    &\stackrel{(\ref{eq:relaxed_triangle})}{\leq} (1 + \gamma) \left\| \mathbf{X}^{(r)} \mathbf{W} - \bar{\mathbf{X}}^{(r)} \right\|^2_F
    + (1 + \gamma^{-1}) {\eta^\prime}^2 \left\| \nabla f(\mathbf{X}^{(r)}) - \mathbf{C}^{(r)} \right\|^2_F
    + {\eta^\prime}^2 n \sigma^2 \\
    &\stackrel{(\ref{eq:assumption:mixing_matrix})}{\leq} (1 + \gamma) (1 - p) \left\| \mathbf{X}^{(r)} - \bar{\mathbf{X}}^{(r)} \right\|^2_F
    + (1 + \gamma^{-1}) {\eta^\prime}^2 \left\| \nabla f(\mathbf{X}^{(r)}) - \mathbf{C}^{(r)} \right\|^2_F
    + {\eta^\prime}^2 n \sigma^2.
\end{align*}
By substituting $\gamma=\frac{p}{2}$, we get
\begin{align*}
    &\mathbb{E}_{r+1} \left\| \mathbf{X}^{(r+1)} - \bar{\mathbf{X}}^{(r)} \right\|^2_F \\
    &\leq (1 - \frac{p}{2}) \left\| \mathbf{X}^{(r)} - \bar{\mathbf{X}}^{(r)} \right\|^2_F
    + \frac{3}{p} {\eta^\prime}^2 \left\| \nabla f(\mathbf{X}^{(r)}) - \mathbf{C}^{(r)} \right\|^2_F
    + {\eta^\prime}^2 n \sigma^2 \\
    &= (1 - \frac{p}{2}) \left\| \mathbf{X}^{(r)} - \bar{\mathbf{X}}^{(r)} \right\|^2_F
    + \frac{3}{p} {\eta^\prime}^2 \left\| \nabla f(\mathbf{X}^{(r)}) - \nabla f(\bar{\mathbf{X}}^{(r)}) + \nabla f(\bar{\mathbf{X}}^{(r)}) - \mathbf{C}^{(r)} \right\|^2_F
    + {\eta^\prime}^2 n \sigma^2 \\
    &\stackrel{(\ref{eq:relaxed_triangle})}{\leq} (1 - \frac{p}{2}) \left\| \mathbf{X}^{(r)} - \bar{\mathbf{X}}^{(r)} \right\|^2_F
    + \frac{6}{p} {\eta^\prime}^2 \left\| \nabla f(\mathbf{X}^{(r)}) - \nabla f(\bar{\mathbf{X}}^{(r)}) \right\|^2_F \\ 
    &\qquad + \frac{6}{p} {\eta^\prime}^2 \left\| \nabla f(\bar{\mathbf{X}}^{(r)}) - \mathbf{C}^{(r)} \right\|^2_F 
    + {\eta^\prime}^2 n \sigma^2 \\
    &\stackrel{(\ref{eq:assumption:smoothness})}{\leq} (1 - \frac{p}{2}) \left\| \mathbf{X}^{(r)} - \bar{\mathbf{X}}^{(r)} \right\|^2_F
    + \frac{6 L^2}{p} {\eta^\prime}^2 \left\| \mathbf{X}^{(r)} - \bar{\mathbf{X}}^{(r)} \right\|^2_F \\
    &\qquad + \frac{6}{p} {\eta^\prime}^2 \underbrace{\left\| \nabla f(\bar{\mathbf{X}}^{(r)}) - \mathbf{C}^{(r)} \right\|^2_F}_{T}
    + {\eta^\prime}^2 n \sigma^2.
\end{align*}
From Lemma \ref{lemma:c_is_zero}, we have $\sum_{i=1}^n \mathbf{c}_i^{(r)} = \mathbf{0}$.
Using $\sum_{i=1}^{n} \| \mathbf{a}_i - \bar{\mathbf{a}} \|^2 = \sum_{i=1}^{n} \| \mathbf{a}_i \|^2 - n \| \bar{\mathbf{a}} \|^2$ for any $\mathbf{a}_1, \ldots, \mathbf{a}_n \in \mathbb{R}^d$, we get
\begin{align*}
    T 
    &= \sum_{i=1}^{n} \left\| \nabla f_i(\bar{\mathbf{x}}^{(r)}) - \mathbf{c}_i^{(r)} \right\|^2 \\
    &= \sum_{i=1}^{n} \left\| \nabla f_i(\bar{\mathbf{x}}^{(r)}) - \mathbf{c}_i^{(r)} - \nabla f (\bar{\mathbf{x}}^{(r)}) \right\|^2
    + n \| \nabla f (\bar{\mathbf{x}}^{(r)}) \|^2 \\
    &= \left\| \nabla f(\bar{\mathbf{X}}^{(r)}) - \mathbf{C}^{(r)} - \frac{1}{n} \nabla f (\bar{\mathbf{X}}^{(r)}) \mathbf{1} \mathbf{1}^\top \right\|^2_F
    + n \| \nabla f (\bar{\mathbf{x}}^{(r)}) \|^2.
\end{align*}
Then, we can get the statement.
\end{proof}
\begin{lemma}
\label{lemma:e_nonconvex}
Suppose that Assumptions \ref{assumption:mixing_matrix}, \ref{assumption:smoothness}, and \ref{assumption:stochastic_gradient} hold,
and $\{ \alpha_{i|j} \}_{ij}$ is set such that $\alpha_{i|j}=\alpha_{j|i} \geq 0$ for all $(i,j) \in \mathcal{E}$.
Then, it holds that
\begin{align*}
    \mathcal{E}^{(r+1)} 
    &\leq (1 - \frac{p}{2}) \mathcal{E}^{(r)}
    + \left( \frac{12 L^2 b^\prime}{p} 
    + \frac{12  (1 - p) b}{p} \Xi^{(r)}
    + \frac{48 L^4}{p} {\eta^\prime}^2 \right) \Xi^{(r)} \\ 
    &\qquad + \frac{48 L^2}{p} {\eta^\prime}^2 \left\| \nabla f(\bar{\mathbf{x}}^{(r)}) \right\|^2 
    + \left( \frac{24 L^2 {\eta^\prime}^2}{n p}
    + \frac{12 b^\prime }{p} \right) \sigma^2. 
\end{align*}
\end{lemma}
\begin{proof}
We have
\begin{align*}
    & \mathbb{E}_{r+1} \left\| \nabla f(\bar{\mathbf{X}}^{(r+1)}) - \mathbf{C}^{(r+1)} - \frac{1}{n} \nabla f(\bar{\mathbf{X}}^{(r+1)}) \mathbf{1}\mathbf{1}^\top \right\|^2_F \\
    %&\stackrel{(\ref{eq:matrix_stochastic_ecl_2})}{=} \mathbb{E}_{r+1} \left\| \nabla f(\bar{\mathbf{X}}^{(r+1)}) - (\mathbf{C}^{(r)} - \nabla F(\mathbf{X}^{(r)} ; \xi^{(r)})) \mathbf{W} \right. \\
    %&\qquad \left. - \frac{1}{2} \mathbf{X}^{(r)} \mathbf{W} (\mathbf{D} - \mathbf{E}) - \nabla F(\mathbf{X}^{(r)} ; \xi^{(r)}) - \frac{1}{n} \nabla f (\bar{\mathbf{X}}^{(r+1)}) \mathbf{1} \mathbf{1}^\top \right\|^2_F \\
    &\stackrel{(\ref{eq:matrix_stochastic_ecl_2})}{=} \mathbb{E}_{r+1} \left\| \nabla f(\bar{\mathbf{X}}^{(r+1)}) - (\mathbf{C}^{(r)} - \nabla F(\mathbf{X}^{(r)} ; \xi^{(r)})) \mathbf{W} \right. \\
    &\qquad \left. - \frac{1}{2} \mathbf{X}^{(r)} \mathbf{W} (\mathbf{D} - \mathbf{E}) - \nabla F(\mathbf{X}^{(r)} ; \xi^{(r)}) - \frac{1}{n} \nabla f (\bar{\mathbf{X}}^{(r+1)}) \mathbf{1} \mathbf{1}^\top \right. \\
    &\qquad \left. + \nabla f(\bar{\mathbf{X}}^{(r)}) (\mathbf{W} - \mathbf{I}) - \nabla f(\bar{\mathbf{X}}^{(r)}) (\mathbf{W} - \mathbf{I})
    + \frac{1}{n} \nabla f(\bar{\mathbf{X}}^{(r)}) \mathbf{1}\mathbf{1}^\top - \frac{1}{n} \nabla f(\bar{\mathbf{X}}^{(r)}) \mathbf{1}\mathbf{1}^\top \right\|^2_F \\
    &\stackrel{(\ref{eq:relaxed_triangle})}{\leq} (1 + \gamma) \left\| 
    (\nabla f(\bar{\mathbf{X}}^{(r)}) - \mathbf{C}^{(r)}) \mathbf{W} - \frac{1}{n} \nabla f(\bar{\mathbf{X}}^{(r)}) \mathbf{1}\mathbf{1}^\top \right\|^2_F \\
    &\qquad + ( 1 + \gamma^{-1}) \mathbb{E}_{r+1} \left\| \nabla f(\bar{\mathbf{X}}^{(r+1)}) - \nabla f(\bar{\mathbf{X}}^{(r)}) 
    + (\nabla F(\mathbf{X}^{(r)} ; \xi^{(r)}) - \nabla f(\bar{\mathbf{X}}^{(r)})) (\mathbf{W} - \mathbf{I}) \right. \\
    &\qquad \left. - \frac{1}{2} \mathbf{X}^{(r)} \mathbf{W} (\mathbf{D} - \mathbf{E})  - \frac{1}{n} \nabla f (\bar{\mathbf{X}}^{(r+1)}) \mathbf{1} \mathbf{1}^\top + \frac{1}{n} \nabla f(\bar{\mathbf{X}}^{(r)}) \mathbf{1}\mathbf{1}^\top \right\|^2_F \\
    &\stackrel{(\ref{eq:assumption:mixing_matrix})}{\leq} (1 + \gamma) (1 - p) \left\| 
    \nabla f(\bar{\mathbf{X}}^{(r)}) - \mathbf{C}^{(r)} - \frac{1}{n} \nabla f(\bar{\mathbf{X}}^{(r)}) \mathbf{1}\mathbf{1}^\top \right\|^2_F \\
    &\qquad + ( 1 + \gamma^{-1}) \mathbb{E}_{r+1} \left\| \nabla f(\bar{\mathbf{X}}^{(r+1)}) - \nabla f(\bar{\mathbf{X}}^{(r)}) 
    + (\nabla F(\mathbf{X}^{(r)} ; \xi^{(r)}) - \nabla f(\bar{\mathbf{X}}^{(r)})) (\mathbf{W} - \mathbf{I}) \right. \\
    &\qquad \left. - \frac{1}{2} \mathbf{X}^{(r)} \mathbf{W} (\mathbf{D} - \mathbf{E})  - \frac{1}{n} \nabla f (\bar{\mathbf{X}}^{(r+1)}) \mathbf{1} \mathbf{1}^\top + \frac{1}{n} \nabla f(\bar{\mathbf{X}}^{(r)}) \mathbf{1}\mathbf{1}^\top \right\|^2_F.
\end{align*}
Substituting $\gamma = \frac{p}{2}$, we get
\begin{align*}
    & \mathbb{E}_{r+1} \left\| \nabla f(\bar{\mathbf{X}}^{(r+1)}) - \mathbf{C}^{(r+1)} - \frac{1}{n} \nabla f(\bar{\mathbf{X}}^{(r+1)}) \mathbf{1}\mathbf{1}^\top \right\|^2_F \\
    &\leq (1 - \frac{p}{2}) \left\| 
    \nabla f(\bar{\mathbf{X}}^{(r)}) - \mathbf{C}^{(r)} - \frac{1}{n} \nabla f(\bar{\mathbf{X}}^{(r)}) \mathbf{1}\mathbf{1}^\top \right\|^2_F \\
    &\qquad + \frac{3}{p} \mathbb{E}_{r+1} \left\| \nabla f(\bar{\mathbf{X}}^{(r+1)}) - \nabla f(\bar{\mathbf{X}}^{(r)}) 
    + (\nabla F(\mathbf{X}^{(r)} ; \xi^{(r)}) - \nabla f(\bar{\mathbf{X}}^{(r)})) (\mathbf{W} - \mathbf{I}) \right. \\
    &\qquad \left. - \frac{1}{2} \mathbf{X}^{(r)} \mathbf{W} (\mathbf{D} - \mathbf{E})  - \frac{1}{n} \nabla f (\bar{\mathbf{X}}^{(r+1)}) \mathbf{1} \mathbf{1}^\top + \frac{1}{n} \nabla f(\bar{\mathbf{X}}^{(r)}) \mathbf{1}\mathbf{1}^\top \right\|^2_F \\
    &\stackrel{(\ref{eq:sum_of_n_vectors})}{\leq} (1 - \frac{p}{2}) \left\| 
    \nabla f(\bar{\mathbf{X}}^{(r)}) - \mathbf{C}^{(r)} - \frac{1}{n} \nabla f(\bar{\mathbf{X}}^{(r)}) \mathbf{1}\mathbf{1}^\top \right\|^2_F \\
    &\qquad + \frac{12}{p} \mathbb{E}_{r+1} \left\| \nabla f(\bar{\mathbf{X}}^{(r+1)}) - \nabla f(\bar{\mathbf{X}}^{(r)}) \right\|^2_F \\
    &\qquad + \frac{12}{p} \mathbb{E}_{r+1} \left\| (\nabla F(\mathbf{X}^{(r)} ; \xi^{(r)}) - \nabla f(\bar{\mathbf{X}}^{(r)})) (\mathbf{W} - \mathbf{I}) \right\|^2_F \\
    &\qquad + \frac{12}{p} \left\| \frac{1}{2} \mathbf{X}^{(r)} \mathbf{W} (\mathbf{D} - \mathbf{E}) \right\|^2_F 
    + \frac{12}{p} \mathbb{E}_{r+1} \left\| \frac{1}{n} \nabla f (\bar{\mathbf{X}}^{(r+1)}) \mathbf{1} \mathbf{1}^\top - \frac{1}{n} \nabla f(\bar{\mathbf{X}}^{(r)}) \mathbf{1}\mathbf{1}^\top \right\|^2_F \\
    &\stackrel{(\ref{eq:assumption:stochastic_gradient})}{\leq} (1 - \frac{p}{2}) \left\| 
    \nabla f(\bar{\mathbf{X}}^{(r)}) - \mathbf{C}^{(r)} - \frac{1}{n} \nabla f(\bar{\mathbf{X}}^{(r)}) \mathbf{1}\mathbf{1}^\top \right\|^2_F \\
    &\qquad + \frac{12}{p} \mathbb{E}_{r+1} \left\| \nabla f(\bar{\mathbf{X}}^{(r+1)}) - \nabla f(\bar{\mathbf{X}}^{(r)}) \right\|^2_F
    + \frac{12 b^\prime}{p} \left\| \nabla f (\mathbf{X}^{(r)}) - \nabla f(\bar{\mathbf{X}}^{(r)}) \right\|^2_F + \frac{12 n b^\prime \sigma^2}{p} \\
    &\qquad + \frac{12}{p} \left\| \frac{1}{2} \mathbf{X}^{(r)} \mathbf{W} (\mathbf{D} - \mathbf{E}) \right\|^2_F 
    + \frac{12}{p} \mathbb{E}_{r+1} \left\| \frac{1}{n} \nabla f (\bar{\mathbf{X}}^{(r+1)}) \mathbf{1} \mathbf{1}^\top - \frac{1}{n} \nabla f(\bar{\mathbf{X}}^{(r)}) \mathbf{1}\mathbf{1}^\top \right\|^2_F.
\end{align*}
Then, using Assumption \ref{assumption:smoothness}, we get
\begin{align*}
    & \mathbb{E}_{r+1} \left\| \nabla f(\bar{\mathbf{X}}^{(r+1)}) - \mathbf{C}^{(r+1)} - \frac{1}{n} \nabla f(\bar{\mathbf{X}}^{(r+1)}) \mathbf{1}\mathbf{1}^\top \right\|^2_F \\
    &\leq (1 - \frac{p}{2}) \left\| 
    \nabla f(\bar{\mathbf{X}}^{(r)}) - \mathbf{C}^{(r)} - \frac{1}{n} \nabla f(\bar{\mathbf{X}}^{(r)}) \mathbf{1}\mathbf{1}^\top \right\|^2_F
    + \frac{12 L^2 b^\prime}{p} \left\| \mathbf{X}^{(r)} - \bar{\mathbf{X}}^{(r)} \right\|^2_F \\ 
    &\qquad + \frac{24 L^2}{p} \underbrace{\mathbb{E}_{r+1} \left\| \bar{\mathbf{X}}^{(r+1)} - \bar{\mathbf{X}}^{(r)} \right\|^2_F}_{T_1}
    + \frac{12}{p} \underbrace{\left\| \frac{1}{2} \mathbf{X}^{(r)} \mathbf{W} (\mathbf{D} - \mathbf{E}) \right\|^2_F}_{T_2} 
    + \frac{12 n b^\prime \sigma^2}{p}. 
\end{align*}
From Lemma \ref{lemma:c_is_zero}, we get
\begin{align*}
    n T_1
    &= \mathbb{E}_{r+1} \left\| \frac{\eta^\prime}{n} \sum_{i=1}^{n} \nabla F_i(\mathbf{x}_i^{(r)} ; \xi_i^{(r)}) \right\|^2 \\
    &\leq \left\| \frac{\eta^\prime}{n} \sum_{i=1}^{n} \nabla f_i(\mathbf{x}_i^{(r)}) \right\|^2 
    + \mathbb{E}_{r+1} \left\| \frac{\eta^\prime}{n} \sum_{i=1}^{n} (\nabla F_i(\mathbf{x}_i^{(r)} ; \xi_i^{(r)}) - \nabla f_i(\mathbf{x}_i^{(r)})) \right\|^2 \\
    &\stackrel{(\ref{eq:assumption:stochastic_gradient})}{\leq} {\eta^\prime}^2 \left\| \frac{1}{n} \sum_{i=1}^{n} \nabla f_i(\mathbf{x}_i^{(r)}) \right\|^2 + \frac{{\eta^\prime}^2 \sigma^2}{n} \\
    &= {\eta^\prime}^2 \left\| \frac{1}{n} \sum_{i=1}^{n} (\nabla f_i(\mathbf{x}_i^{(r)}) - \nabla f_i(\bar{\mathbf{x}}^{(r)}) + \nabla f_i(\bar{\mathbf{x}}^{(r)})) \right\|^2 + \frac{{\eta^\prime}^2 \sigma^2}{n} \\
    &\stackrel{(\ref{eq:sum_of_n_vectors})}{\leq} 2 {\eta^\prime}^2 \left\| \frac{1}{n} \sum_{i=1}^{n} (\nabla f_i(\mathbf{x}_i^{(r)}) - \nabla f_i(\bar{\mathbf{x}}^{(r)})) \right\|^2
    + 2 {\eta^\prime}^2 \left\| \nabla f(\bar{\mathbf{x}}^{(r)}) \right\|^2 + \frac{{\eta^\prime}^2 \sigma^2}{n} \\
    &\stackrel{(\ref{eq:assumption:smoothness})}{\leq}  \frac{2 L^2 {\eta^\prime}^2}{n}  \sum_{i=1}^{n} \left\| \mathbf{x}_i^{(r)} - \bar{\mathbf{x}}^{(r)} \right\|^2
    + 2 {\eta^\prime}^2 \left\| \nabla f(\bar{\mathbf{x}}^{(r)}) \right\|^2 + \frac{{\eta^\prime}^2 \sigma^2}{n}.
\end{align*}
From the definitions of $\mathbf{E}$ and $\mathbf{D}$, we have
\begin{align*}
    T_2 
    = \left\| \frac{1}{2} (\mathbf{X}^{(r)} \mathbf{W} - \bar{\mathbf{X}}^{(r)}) (\mathbf{D} - \mathbf{E}) \right\|^2_F 
    \leq (1 - p) b \left\| \mathbf{X}^{(r)} - \bar{\mathbf{X}}^{(r)} \right\|^2_F.
\end{align*}
Then, we get the statement.
\end{proof}

\begin{lemma}
\label{lemma:consensus_e_nonconvex}
Suppose that Assumptions \ref{assumption:mixing_matrix}, \ref{assumption:smoothness}, and \ref{assumption:stochastic_gradient} hold,
and $\{ \alpha_{i|j} \}_{ij}$ is set such that $\alpha_{i|j}=\alpha_{j|i} \geq 0$ for all $(i,j) \in \mathcal{E}$.
Then, if $\eta^\prime$ satisfies
\begin{align}
    \label{eq:lemma:consensus_e_nonconvex}
    \eta^\prime \leq \min \left\{ \frac{p}{2 \sqrt{6 (L^2 + \frac{48 L^2 b^\prime + 12 L^2 + 48 (1 - p) b }{p^2})}}, \frac{\sqrt{n p^3}}{24 L}, \frac{1}{4L} \right\},
\end{align}
it holds that
\begin{align*}
    &\Xi^{(r+1)} + \frac{24}{p^2} {\eta^\prime}^2 \mathcal{E}^{(r+1)} \\
    &\leq \left( 1 - \frac{p}{4} \right) \left(  \Xi^{(r)} + \frac{24}{p^2} {\eta^\prime}^2 \mathcal{E}^{(r)} \right)
    + \frac{54}{p} {\eta^\prime}^2 \mathbb{E} \| \nabla f(\bar{\mathbf{x}}^{(r)}) \|^2
    + \left( 2 + \frac{288 b^\prime }{p^3} \right) {\eta^\prime}^2 \sigma^2.
\end{align*}
\end{lemma}
\begin{proof}
From Lemma \ref{lemma:e_nonconvex}, we have
\begin{align*}
    \frac{24}{p^2} \mathcal{E}^{(r+1)} 
    &\leq (1 - \frac{p}{2}) \frac{24}{p^2} \mathcal{E}^{(r)}
    + \frac{24}{p^2} \left( \frac{12 L^2 b^\prime}{p} 
    + \frac{12  (1 - p) b}{p}
    + \frac{48 L^4}{p} {\eta^\prime}^2 \right) \Xi^{(r)} \\ 
    &\qquad + \frac{1152 L^2}{p^3} {\eta^\prime}^2 \left\| \nabla f(\bar{\mathbf{x}}^{(r)}) \right\|^2 
    + \frac{24}{p^2} \left( \frac{24 L^2 {\eta^\prime}^2}{n p}
    + \frac{12 b^\prime }{p} \right) \sigma^2. 
\end{align*}
By using that ${\eta^\prime}^2 \leq \frac{1}{16 L^2}$ and ${\eta^\prime}^2 \leq \frac{n p^3}{576 L^2}$, we get
\begin{align*}
    \frac{24}{p^2} \mathcal{E}^{(r+1)} 
    &\leq (1 - \frac{p}{2}) \frac{24}{p^2} \mathcal{E}^{(r)}
    + \frac{24}{p^2} \left( \frac{12 L^2 b^\prime}{p} 
    + \frac{12  (1 - p) b}{p}
    + \frac{3 L^2}{p} \right) \Xi^{(r)} \\ 
    &\qquad + \frac{1152 L^2}{p^3} {\eta^\prime}^2 \left\| \nabla f(\bar{\mathbf{x}}^{(r)}) \right\|^2 
    + \left( 1 + \frac{288 b^\prime }{p^3} \right) \sigma^2. 
\end{align*}
Because ${\eta^\prime}^2 \leq p^2 / ( 24 (L^2 + \frac{48 L^2 b^\prime + 12 L^2 + 48 (1 - p) b}{p^2}))$ implies that ${\eta^\prime}^2 \leq \frac{p^2}{24 L^2}$, we get
\begin{align*}
    \frac{24}{p^2} \mathcal{E}^{(r+1)} 
    &\leq (1 - \frac{p}{2}) \frac{24}{p^2} \mathcal{E}^{(r)}
    + \frac{24}{p^2} \left( \frac{12 L^2 b^\prime}{p} 
    + \frac{12  (1 - p) b}{p}
    + \frac{3 L^2}{p} \right) \Xi^{(r)} \\ 
    &\qquad + \frac{48}{p} \left\| \nabla f(\bar{\mathbf{x}}^{(r)}) \right\|^2 
    + \left( 1 + \frac{288 b^\prime }{p^3} \right) \sigma^2. 
\end{align*}
Combining this with Lemma \ref{lemma:consensus_nonconvex}, we get
\begin{align*}
    \Xi^{(r+1)} + \frac{24}{p^2} {\eta^\prime}^2 \mathcal{E}^{(r+1)} 
    &\leq \left( (1 - \frac{p}{2}) + \frac{6L^2}{p} {\eta^\prime}^2 + \frac{6}{p^2} \left( \frac{48 L^2 b^\prime}{p} + \frac{48 (1 - p) b}{p} + \frac{12 L^2}{p} \right) {\eta^\prime}^2  \right)  \Xi^{(r)}  \\
    &\qquad + \left( \frac{6}{p} {\eta^\prime}^2 + \left( 1 - \frac{p}{2} \right) \frac{24}{p^2} {\eta^\prime}^2 \right) \mathcal{E}^{(r)} \\
    &\qquad + \left( \frac{6}{p} {\eta^\prime}^2 + \frac{48}{p} {\eta^\prime}^2 \right) \mathbb{E} \| \nabla f(\bar{\mathbf{x}}^{(r)}) \|^2
    + \left( {\eta^\prime}^2 + \left( 1 + \frac{288 b^\prime }{p^3} \right)  {\eta^\prime}^2 \right) \sigma^2.
\end{align*}
Using that ${\eta^\prime}^2 \leq p^2 / ( 24 (L^2 + \frac{48 L^2 b^\prime + 12 L^2 + 48 (1 - p)b}{p^2}))$, we get
\begin{align*}
    &\Xi^{(r+1)} + \frac{24}{p^2} {\eta^\prime}^2 \mathcal{E}^{(r+1)} \\
    &\leq \left( 1 - \frac{p}{4} \right) \left(  \Xi^{(r)} + \frac{24}{p^2} {\eta^\prime}^2 \mathcal{E}^{(r)} \right)
    + \frac{54}{p} {\eta^\prime}^2 \mathbb{E} \| \nabla f(\bar{\mathbf{x}}^{(r)}) \|^2
    + \left( 2 + \frac{288 b^\prime }{p^3} \right) {\eta^\prime}^2 \sigma^2.
\end{align*}
This concludes the proof.
\end{proof}

\begin{lemma}
\label{lemma:recursion_nonconvex}
Suppose that Assumptions \ref{assumption:mixing_matrix}, \ref{assumption:smoothness}, and \ref{assumption:stochastic_gradient} hold,
and $\{ \alpha_{i|j} \}_{ij}$ is set such that $\alpha_{i|j}=\alpha_{j|i} \geq 0$ for all $(i,j) \in \mathcal{E}$.
Then, if $\eta^\prime$ satisfies
\begin{align}
    \label{eq:lemma:simplified_consensus_nonconvex}
    \eta^\prime \leq \min \left\{ \frac{p}{24L \sqrt{3}}, \frac{p}{2 \sqrt{6 (L^2 + \frac{48 L^2 b^\prime + 12 L^2 + 48 (1 - p) b }{p^2})}}, \frac{\sqrt{n p^3}}{24 L}, \frac{1}{4L} \right\},    
\end{align}
we have
\begin{align*}
    \frac{L^2}{R+1} \sum_{r=0}^{R} \Xi^{(r)}
    &\leq \frac{1}{8(R+1)} \sum_{r=0}^{R} \mathbb{E} \| \nabla f(\bar{\mathbf{x}}^{(r)}) \|^2 
    + \left( 8 + \frac{1152 b^\prime}{p^3} \right) \frac{L^2 \sigma^2}{p} {\eta^\prime}^2.
\end{align*}
\end{lemma}
\begin{proof}
We define $\Theta^{(r)} \coloneqq \Xi^{(r)} + \frac{24}{p^2} {\eta^\prime}^2 \mathcal{E}^{(r)}$.
From Lemma \ref{lemma:consensus_e_nonconvex}, we get
\begin{align*}
    \Theta^{(r)}
    &\leq \left( 1 - \frac{p}{4} \right) \Theta^{(r-1)}
    + \frac{54}{p} {\eta^\prime}^2 \mathbb{E} \| \nabla f(\bar{\mathbf{x}}^{(r-1)}) \|^2
    + \left( 2 + \frac{288 b^\prime }{p^3} \right) {\eta^\prime}^2 \sigma^2 \\
    &= \frac{54}{p} {\eta^\prime}^2 \sum_{j=0}^{r-1} \left( 1 - \frac{p}{4} \right)^{r-j-1} \mathbb{E} \| \nabla f(\bar{\mathbf{x}}^{(j)}) \|^2
    + \left( 2 + \frac{288 b^\prime }{p^3} \right) {\eta^\prime}^2 \sigma^2 \sum_{j=0}^{r-1} \left( 1 - \frac{p}{4} \right)^{r-j-1} \\
    &\leq \frac{54}{p} {\eta^\prime}^2 \sum_{j=0}^{r-1} \left( 1 - \frac{p}{4} \right)^{r-j-1} \mathbb{E} \| \nabla f(\bar{\mathbf{x}}^{(j)}) \|^2
    + \left( 8 + \frac{1152 b^\prime}{p^3} \right) \frac{\sigma^2}{p} {\eta^\prime}^2,
\end{align*}
for any round $r>0$.
By recursively adding both sides, we get
\begin{align*}
    \sum_{r=1}^{R} \Theta^{(r)}
    &\leq \frac{54}{p} {\eta^\prime}^2 \sum_{r=1}^{R} \sum_{j=0}^{r-1} \left( 1 - \frac{p}{4} \right)^{r-j-1} \mathbb{E} \| \nabla f(\bar{\mathbf{x}}^{(j)}) \|^2
    + \left( 8 +  \frac{1152 b^\prime}{p^3} \right) \frac{\sigma^2}{p} {\eta^\prime}^2 R \\
    &= \frac{54}{p} {\eta^\prime}^2 \sum_{j=0}^{R-1} \mathbb{E} \| \nabla f(\bar{\mathbf{x}}^{(j)}) \|^2 \sum_{r=j+1}^{R} \left( 1 - \frac{p}{4} \right)^{r-j-1}
    + \left( 8 +  \frac{1152 b^\prime}{p^3} \right) \frac{\sigma^2}{p} {\eta^\prime}^2 R \\
    &\leq \frac{216}{p^2} {\eta^\prime}^2 \sum_{j=0}^{R-1} \mathbb{E} \| \nabla f(\bar{\mathbf{x}}^{(j)}) \|^2 
    + \left( 8 +  \frac{1152 b^\prime}{p^3} \right) \frac{\sigma^2}{p} {\eta^\prime}^2 R.
\end{align*}
By using $\Xi^{(r)} \leq \Theta^{(r)}$, we get
\begin{align*}
    \sum_{r=1}^{R} \Xi^{(r)}
    &\leq \frac{216}{p^2} {\eta^\prime}^2 \sum_{j=0}^{R-1} \mathbb{E} \| \nabla f(\bar{\mathbf{x}}^{(j)}) \|^2 
    + \left( 8 + \frac{1152 b^\prime}{p^3} \right) \frac{\sigma^2}{p} {\eta^\prime}^2 R.
\end{align*}
Using $\Xi^{(0)} = 0$, we get
\begin{align*}
    \frac{1}{R+1} \sum_{r=0}^{R} \Xi^{(r)}
    &\leq \frac{216}{p^2 (R+1)} {\eta^\prime}^2 \sum_{r=0}^{R} \mathbb{E} \| \nabla f(\bar{\mathbf{x}}^{(r)}) \|^2 
    + \left( 8 + \frac{1152 b^\prime}{p^3} \right) \frac{\sigma^2}{p} {\eta^\prime}^2.
\end{align*}
By multiplying the above equation by $L^2$, we get
\begin{align*}
    \frac{L^2}{R+1} \sum_{r=0}^{R} \Xi^{(r)}
    &\leq \frac{216 L^2}{p^2 (R+1)} {\eta^\prime}^2 \sum_{r=0}^{R} \mathbb{E} \| \nabla f(\bar{\mathbf{x}}^{(r)}) \|^2 
    + \left( 8 + \frac{1152 b^\prime}{p^3} \right) \frac{L^2 \sigma^2}{p} {\eta^\prime}^2.
\end{align*}
By using ${\eta^\prime}^2 \leq \frac{p^2}{1728 L^2}$, we get
\begin{align*}
    \frac{L^2}{R+1} \sum_{r=0}^{R} \Xi^{(r)}
    &\leq \frac{1}{8(R+1)} \sum_{r=0}^{R} \mathbb{E} \| \nabla f(\bar{\mathbf{x}}^{(r)}) \|^2 
    + \left( 8 + \frac{1152 b^\prime}{p^3} \right) \frac{L^2 \sigma^2}{p} {\eta^\prime}^2.
\end{align*}
This concludes the proof.
\end{proof}

\begin{lemma}[Convergence Rate for Non-convex Case]
Suppose that Assumptions \ref{assumption:mixing_matrix}, \ref{assumption:smoothness}, and \ref{assumption:stochastic_gradient} hold,
and $\{ \alpha_{i|j} \}_{ij}$ is set such that $\alpha_{i|j}=\alpha_{j|i} \geq 0$ for all $(i,j) \in \mathcal{E}$.
Then, there exists $\eta^\prime < \frac{1}{d}$ such that it holds that
\begin{align*}
    \frac{1}{R+1} \sum_{r=0}^{R} \mathbb{E} \left\| \nabla f(\bar{\mathbf{x}}^{(r)}) \right\|^2
    &\leq 
    \mathcal{O} \left( \left( \frac{L \sigma^2 r_0}{n (R+1)} \right)^{\frac{1}{2}} 
    + \left( \frac{(1 + \frac{b^\prime}{ p^3 }) L^2 \sigma^2}{p} \right)^{\frac{1}{3}}
    \left( \frac{r_0}{R+1} \right)^{\frac{2}{3}}
    + \frac{d r_0}{R+1}
    \right),
\end{align*}
where $r_0 \coloneqq f(\bar{\mathbf{x}}^{(0)}) - f^\star$.
\end{lemma}
\begin{proof}
From Lemma \ref{lemma:descent_nonconvex}, we have
\begin{align*}
    &\frac{1}{4(R+1)} \sum_{r=0}^{R} \mathbb{E} \left\| \nabla f(\bar{\mathbf{x}}^{(r)}) \right\|^2 \\
    &\leq \frac{1}{\eta^\prime (R+1)} \sum_{r=0}^{R} \left( \mathbb{E} [f(\bar{\mathbf{x}}^{(r)})] - \mathbb{E} [f(\bar{\mathbf{x}}^{(r+1)})] \right)
    + \frac{L^2}{R+1} \sum_{r=0}^{R} \Xi^{(r)}
    + \frac{L\sigma^2 {\eta^\prime}}{2n} \\
    &\leq \frac{f(\bar{\mathbf{x}}^{(0)}) - f^\star}{\eta^\prime (R+1)}
    + \frac{L^2}{R+1} \sum_{r=0}^{R} \Xi^{(r)}
    + \frac{L\sigma^2 {\eta^\prime}}{2n}.
\end{align*}
Using Lemma \ref{lemma:recursion_nonconvex}, we get
\begin{align*}
    &\frac{1}{8(R+1)} \sum_{r=0}^{R} \mathbb{E} \left\| \nabla f(\bar{\mathbf{x}}^{(r)}) \right\|^2 
    \leq \frac{f(\bar{\mathbf{x}}^{(0)}) - f^\star}{\eta^\prime (R+1)}
    + \frac{L\sigma^2}{2n} \eta^\prime
    + \left( 8 + \frac{1152 b^\prime}{p^3} \right) \frac{L^2 \sigma^2}{p} {\eta^\prime}^2.
\end{align*}
Using Lemma 16 in \citep{koloskova2020unified}, we get the statement.
\end{proof}

\section{Limitation of Theorem \ref{theoram:convergence_rate}}
\label{sec:limitation}
In this section, we discuss the limitations of Theorem \ref{theoram:convergence_rate}
and describe why the convergence rates shown in Theorem \ref{theoram:convergence_rate} can not be regarded as that of the ECL.

In Lemmas \ref{lemma:consensus_e_convex}, \ref{lemma:simplified_consensus_convex}, \ref{lemma:recursion_convex},
\ref{lemma:consensus_e_nonconvex}, and \ref{lemma:recursion_nonconvex},
we assume that the step size $\eta^\prime$ is upper bounded.
In the G-ECL and Gossip algorithm, there exists a step size $\eta^\prime$ that satisfies the assumptions of Lemmas \ref{lemma:consensus_e_convex}, \ref{lemma:simplified_consensus_convex}, \ref{lemma:recursion_convex},
\ref{lemma:consensus_e_nonconvex}, and \ref{lemma:recursion_nonconvex}
because the mixing matrix $\mathbf{W}$ and step size $\eta^\prime$ can be set independently as hyperparameters.
However, in the ECL, $\mathbf{W}$ and $\eta^\prime$ are determined by $\eta$ and $\{ \alpha_{i|j} \}_{ij}$ as in Eq. \eqref{eq:definition_of_eta_prime}
and depend on one another.
That is, $\eta^\prime$, $p$ in Assumption \ref{assumption:mixing_matrix}, and $b^\prime (\coloneqq \|\mathbf{W} - \mathbf{I} \|^2)$
depend on each other.
Therefore, to prove that the convergence rates of the ECL are that shown in Theorem \ref{theoram:convergence_rate},
we need to prove that there exists a step size $\eta^\prime$ that satisfies the assumptions of Lemmas \ref{lemma:consensus_e_convex}, \ref{lemma:simplified_consensus_convex}, \ref{lemma:recursion_convex},
\ref{lemma:consensus_e_nonconvex}, and \ref{lemma:recursion_nonconvex}.
In this work, it is left to future work to prove whether there exists a step size $\eta^\prime$ that satisfies the assumptions of Lemmas \ref{lemma:consensus_e_convex}, \ref{lemma:simplified_consensus_convex}, \ref{lemma:recursion_convex},
\ref{lemma:consensus_e_nonconvex}, and \ref{lemma:recursion_nonconvex}
and we experimentally demonstrate that the ECL converges at the same convergence rate as the G-ECL in Sec. \ref{sec:experimental_results}.

\end{document}